\documentclass{amsart}

\numberwithin{equation}{section}

\usepackage{amssymb, amsmath, amsthm}
\usepackage{enumerate, xspace}
\usepackage[normalem]{ulem}
\usepackage{geometry}
\usepackage{esint}
\usepackage{tikz-cd}
\usepackage{tikz}
\usetikzlibrary{calc}
\usetikzlibrary{intersections}
\usetikzlibrary{shapes.misc}
\tikzset{cross/.style={cross out, 
                       draw=black,
		       minimum size=2*(#1-\pgflinewidth),
		       inner sep=0pt,
		       outer sep=0pt},  
         cross/.default={1pt}}
\usepackage{subcaption}
\usepackage{marvosym} 
\usepackage{bbm}       
\usepackage[bottom,first]{draftcopy}
\usepackage{changebar}
\usepackage[backgroundcolor=blue!20!white, linecolor=blue!20!white,
textsize=footnotesize]{todonotes}
\usepackage[colorlinks]{hyperref}

\newtheorem{thm}{Theorem}[section]
\newtheorem{lem}[thm]{Lemma}
\newtheorem{sub-lem}[thm]{Sub-Lemma}
\newtheorem{cor}[thm]{Corollary}

\newtheorem{prop}[thm]{Proposition}
\newtheorem{prob}[thm]{\bf Problem}

\newtheorem{defin}[thm]{Definition}

\newtheorem{rem}[thm]{Remark}

\newcommand\cA{{\mathcal A}}
\newcommand\cB{{\mathcal B}}
\newcommand\cC{{\mathcal C}}
\newcommand\cD{{\mathcal D}}

\newcommand\cF{{\mathcal F}}
\newcommand\cG{{\mathcal G}}

\newcommand\cK{{\mathcal K}}
\newcommand\cL{{\mathcal L}}
\newcommand\cO{{\mathcal O}}
\newcommand\cM{{\mathcal M}}

\newcommand{\cP}{{\mathcal P}}
\newcommand{\cQ}{{\mathcal Q}}
\newcommand\cS{{\mathcal S}}

\newcommand\cW{{\mathcal W}}

\newcommand\bC{{\mathbb C}}

\newcommand\bE{{\mathbb E}}

\newcommand{\bH}{{\mathbb H}}

\newcommand\bN{{\mathbb N}}
\newcommand\bM{{\mathbb M}}
\newcommand\bP{{\mathbb P}}

\newcommand\bR{{\mathbb R}}
\newcommand\bS{{\mathbb S}}
\newcommand\bT{{\mathbb T}}
\newcommand\bV{{\mathbb V}}

\newcommand\ve{\varepsilon}

\newcommand\vf{\varphi}

\newcommand{\be}{\mathbbm e}
\newcommand{\bt}{\mathbbm t}
\newcommand{\bbm}{\mathbbm m}

\newcommand\Id{{\mathbbm{1}}}
\newcommand{\tri}{{|\:\!\!|\:\!\!|}}

\newcommand\up{\varkappa}

\newcommand{\musrb}{\mu_{\mbox{\tiny SRB}}} 
\newcommand{\diam}{\mbox{diam}}

\newcommand\hW{{\widehat{\cW}}}
\newcommand{\tnu}{\tilde{\nu}}

\newcommand{\htop}{h_{\mbox{\scriptsize top}}}
\newcommand{\oldbetaparam}{q} 
\newcommand{\oldq}{\beta}
\newcommand{\oldbeta}{{1/q}}

\newcommand{\oldh}{g}
\newcommand{\oldf}{f}

\begin{document}

\title[Dispersing Billiards]{Recent 
Progress in the Application of Transfer Operators to Dispersing Billiards}
\author{Mark F. Demers}
\address{Mark F. Demers\\
Department of Mathematics,\\
 Fairfield University, Fairfield CT 06824, USA.}
\email{{\tt  mdemers@fairfield.edu}}
\author{Carlangelo Liverani}
\address{Carlangelo Liverani\\
Dipartimento di Matematica\\
II Universit\`{a} di Roma (Tor Vergata)\\
Via della Ricerca Scientifica, 00133 Roma, Italy.}
\address{\vskip-16pt \noindent Mathematics Department\\
University of Maryland\\
4176 Campus Drive - William E. Kirwan Hall\\
College Park, MD 20742-4015}
\email{{\tt liverani@mat.uniroma2.it}}
\date{\today}
\begin{abstract}
This is a review paper about the use of transfer operators to study the statistical properties of hyperbolic billiards. The main focus is on equilibrium states, sequential billiards, and decay of correlation for billiard flows. In addition to reviewing the literature, we try to flesh out the main ideas, highlight the relation between them, and provide the larger context in which the problems discussed here are situated.
\end{abstract}
\thanks{This work was supported by the PRIN Grant ``Regular and stochastic behaviour in dynamical systems" (PRIN 2017S35EHN). L.C. acknowledges the MIUR Excellence Department Project awarded to the Department of Mathematics, University of Rome Tor Vergata, MatMod@TOV.  M.D. was partially supported by National Science Foundation Grant DMS 2350079.}
\maketitle

\tableofcontents

\section{Introduction}

This paper presents a survey of recently developed functional analytic techniques for dispersing billiards, which have 
opened up new methods to approach some long-standing open problems.  In particular, it has allowed the 
recent development of a thermodynamic formalism to study a variety of equilibrium states for these systems
\cite{BD1, BD2, Ca}. 
In addition, they give access to information about the flow \cite{bdl, bcd} and allow us to study sequential billiard systems via projective cones \cite{dl cones, dl clt}. Beyond surveying this recent progress, we state some
open problems and directions of future study, which we hope can be informed by these developments.

We focus our exposition on the simplest case:  the finite-horizon Sinai billiard without corner points.  We will indicate where some of these techniques have also been implemented in the case of infinite-horizon billiards and dispersing tables with corner points. Nevertheless, to facilitate the non-expert reader, we begin with a section that recalls the main billiard properties and terminology.

The paper is organized as follows.  In Section~\ref{sec:disp_bill}, we define the billiard flow
and recall relevant terminology.  We also provide an overview and brief derivations of essential properties, 
such as the uniform
hyperbolicity of the flow, the Poincar\'e map and the associated invariant foliations.  Section~\ref{sec:geometric}
focuses on equilibrium states for the billiard map for a class of 
geometric potentials, $\phi_t = - t \log J^uT$, $t \ge 0$,
where $J^uT$ denotes the unstable Jacobian of the billiard map $T$.  First we review the case
$t>0$ where the associated transfer operators have a spectral gap, and in Section~\ref{sec:entropy}
we turn to the case $t=0$, where the proof of a spectral gap fails; nevertheless, we describe how
an equilibrium state can be recovered, which turns out to be the measure of maximal entropy for the billiard map.

In Section~\ref{sec:abramov}, we review the recent work of Carrand \cite{Ca}, which generalizes the
class of potentials under consideration, and discuss the family $\phi_t = - t \tau$, where $\tau$ denotes the 
intercollision time of the billiard map.  In this case, one obtains the measure of maximal entropy
for the billiard flow by constructing an equilibrium state for $t= \htop(\Phi_1)$, where $\Phi_1$ is the time-one
map of the billiard flow.

In Section~\ref{sec:sequential} we review the recent construction of both real and
complex projective cones for billiards, and provide applications of this technique to a sequential
Central Limit Theorem, loss of memory, chaotic scattering, and a special case of a random Lorentz gas.

In Section~\ref{sec:equivalence}, we prove a new result which may be of independent interest:
the norms for the Banach spaces constructed in Section~\ref{sec:geometric} are equivalent to the norms
defined from the ordering in the cones defined in Section~\ref{sec:sequential}.
Although this paper is mainly a review of existing results, we think it is appropriate to present
a proof of this equivalence in a context where both Banach spaces and projective cones are discussed in
depth.  This shows that the two main techniques discussed here are different aspects of the same philosophy. 
In fact, while this equivalence is proven for billiards in a special case, the logic is rather general and 
provides a blueprint for how to go from cones to Banach spaces and vice versa in other situations.

Section~\ref{sec:flow} reviews the application of the transfer operator to prove exponential decay of
correlations for the billiard flow.  We describe the use of the resolvent operator to obtain genuine 
Lasota-Yorke inequalities for a modified set of norms for the flow 
and describe the principal elements leading up to the
famed Dolgopyat estimate.  This estimate is too technical for the present review, and the reader is
referred to the appropriate references for details.

Finally, in Section~\ref{sec:open} we present a collection of open problems motivated by the
topics in the present review.  These problems vary in difficulty, but are all problems we believe
can be approached using the transfer operator techniques described here.

\section{Dispersing Billiards}\label{sec:disp_bill}
By {\em dispersing billiards} we mean a dynamical system consisting of a particle moving in straight lines and colliding elastically against convex obstacles. Such a motion can take place either on a compact set (typically $\bT^d$) or on a noncompact space (typically $\bR^d$). Various generalizations exist, for example, forces can be present (e.g. an electric or magnetic field), or the motion can take place on a manifold, along geodesics. Different collision rules have also been considered. Yet, in this review, we will limit ourselves to the two scenarios mentioned above. That is, the flow $\Phi_t$ will take place in a space $\tilde X\in\{\bT^d,\bR^d\}$.

If $\{B_i\}$ is the set of strictly convex {\em obstacles} (or {\em scatterers}), with $\cC^2$ boundary, then the configuration space is given by $X=\tilde X\setminus (\cup_i B_i)$. By energy conservation, the velocity is constant, and by a simple time rescaling, we can assume it to be one. Then, the phase space is given by $\Omega_0=X\times \bS^{d-1}\subset \bR^{2d}$; in what follows, we will identify $\bS^{d-1} = \{p\in\bR^d\;:\;\|p\|=1\}$.
We can then use the coordinates $(q, p) \in\Omega_0\subset \bR^{2d}$, 
where $q$ is the position and $p$ the velocity (or momentum; we will always assume the particle's mass is $1$). It is sometimes convenient to define the quotient
\begin{equation}\label{eq:Omega_def}
\Omega := \{ (q, p)\mid q \in X\, , \quad p \in \bS^{d-1}\}  /\sim
\end{equation}
as the phase space for the billiard flow $\Phi_t:\Omega\to \Omega$, where
$\sim$ identifies   ingoing and outgoing velocity
vectors at the collisions, via the reflection with respect to the boundary
of the scatterer at $q\in \partial X$. In this way, $\Phi_t$ becomes a continuous flow.
\subsection{The flow}\ \label{sec:flowdef}\\
The dynamics is a flow with collisions: if no collision takes place, then 
\begin{equation}\label{eq:no_collisions}
\begin{split}
&q(t)=q+pt\\
&p(t)=p.
\end{split}
\end{equation}
While, if $q\in \partial B_i$ for some obstacle $B_i$ and $\langle p, n_i(q)\rangle <0$,
where $n_i(q)$ is the external unit normal to $B_i$ at the point $q$,\footnote{This means that the particle is just before collision and not just after.} then the position and velocity $(q_+,p_+)$ just after the collision are given by
\begin{equation}\label{eq:coll}
\begin{split}
&q_+=q\\
&p_+=p-2\langle p,n_i(q)\rangle n_i(q) .
\end{split}
\end{equation}
If $\langle p, n_i(q)\rangle =0$ (tangential collision), we can treat it as no collision.

Note that equation \eqref{eq:coll} may be ambiguous if $q$ belongs to the boundary of more than one obstacle. To simplify matters, here we will consider the case in which $B_i\cap B_j\neq \emptyset$ implies $i=j$, i.e., the obstacles are disjoint. This implies that the boundary has no {\em corner points}.
The above dynamics are Hamiltonian, hence the Liouville measure is invariant. However, many other invariant measures exist. This will be addressed in detail in the next section.

 Let $\tau: \Omega\to \bR_+$ be the time at which the first collision with an obstacle takes place. 
\begin{rem}
By \emph{finite horizon} we mean that there exists $\tau_{max}<\infty$ such that, for each $q\in\cup_i\partial B_i$, $\tau(q,p)\leq\tau_{max}$. By \emph{no-corners} we mean that there exists 
$\tau_{min}>0$ such that, for each $q\in\cup_i\partial B_i$, $\langle p,n(q)\rangle >0$, $\tau(q,p)\geq \tau_{min}$. This is the case we will discuss in the following.
\end{rem}

A key property of dispersing billiards is that they are hyperbolic (see \cite{sinai, SC87} for the original proof). That is, the tangent space of $\Omega$ splits into three invariant sub-bundles such that the vectors in one are expanded by the dynamics, in another are contracted, while the third is given by the flow direction. 
To illustrate this, consider an initial condition $(q,p)$ and a time $t$ at which the particle has undergone exactly one collision. Let $B$ be the obstacle against which the collision takes place, and let $U\ni (q,p)$ be an open set for which all the elements have also exactly one collision with the obstacle $B$, up to time $t$. 
We define 
\begin{equation}\label{eq:q-one}
(q_-,p_-)=(q+\tau(q,p) p,p)
\end{equation}
 the coordinates just before the collision, and $(q_+,p_+)$ the coordinates just after the collision so that 
\begin{equation}\label{eq:q-two} 
(q(t),p(t))=(q_++(t-\tau(q,p))p_+, p_+).
\end{equation} 
If we differentiate with respect to the tangent vector $(\delta q,\delta p)$, and $n$ is the outer normal to $B$, it must be
\[
0=\langle\delta q_-, n(q_-)\rangle=\langle \delta q+\langle\nabla \tau, (\delta q,\delta p)\rangle p+\tau\delta p, n(q_-)\rangle,
\]
which implies
\begin{equation}\label{eq:tauder}
\nabla \tau=\frac{-1}{\langle p,n(q_-)\rangle}(n(q_-),\tau n(q_-)).
\end{equation}
We can then compute $(\delta q(t),\delta p(t))$, by differentiating \eqref{eq:coll} at the point $(q_-,p_-)$, which, setting $\Pi(q,p)\xi:=\xi -\frac{\langle \xi,n(q)\rangle}{\langle p, n(q)\rangle}p$, yields
\begin{equation}\label{eq:collder}
\begin{split}
\delta q_+=&\delta q_-=\delta q+\tau \delta p-\frac{\langle n(q_+), \delta q+\tau\delta p\rangle}{\langle p,n(q_+)\rangle} p=\Pi(q_+,p)(\delta q+\tau\delta p)\\
\delta p_+=&\delta p -2\langle  \delta  p,n(q_-)\rangle n(q_-)-2\langle p, n(q_+)\rangle \Pi(q_+,p_+)^*\cK(q_+)\delta q_-
\end{split}
\end{equation}
where $\cK(q):T_q\partial B\to T_q\partial B$ is the curvature of the boundary of the obstacle. By equations \eqref{eq:no_collisions}, \eqref{eq:coll}, \eqref{eq:q-one}, \eqref{eq:q-two} and \eqref{eq:collder} we obtain
\begin{equation}\label{eq:deriv_flow}
\begin{split}
&\delta q(t)= \delta q_-+(t-\tau)\delta p_+-\langle\nabla \tau,(\delta q, \delta p)\rangle p_+\\
&\delta p(t)=\delta p_+.
\end{split}
\end{equation}
Recalling that the kinetic energy is conserved, the motion takes place on the $2d-1$ dimensional constant energy surface, that is $p\in \bS^{d-1}=\{p\in\bR^d\;:\;\|p\|=1\}$,\footnote{ The kinetic energy $\frac 12\|p\|^2=\frac 12$ can always be achieved by rescaling the time.} hence, $\langle \delta p,p\rangle=0$.
In addition, the flow direction corresponds to a zero Lyapunov exponent. To verify the hyperbolicity, we have to show that all the other directions yield a non-zero Lyapunov exponent.
The reader can check that 
\begin{equation}\label{eq:contact}
\langle \delta q, p\rangle=0\quad\Longrightarrow \quad \langle \delta q(t),p(t)\rangle=0, \; \textrm{ for all } t>0,
\end{equation}
 this is a manifestation of the fact that the flow is a contact flow. We can then restrict to the case $\langle \delta q, p\rangle=0$.  For such vectors, we can define the cone field
\[
\cC(q,p)=\{(\delta q,\delta p)\in T_{(q,p)}\bM\;:\; \langle \delta q,\delta p\rangle\geq 0\}.
\]
Using the above facts, we have
\[
\begin{split}
\langle \delta q(t),\delta p(t)\rangle&=\langle  \delta q_-+(t-\tau)\delta p_+-\langle\nabla \tau,(\delta q, \delta p)\rangle p_+, \delta p_+\rangle\\
&=\langle  \delta q_-,\delta p_+\rangle+(t-\tau)\|\delta p_+\|^2\\
&=\langle  \delta q_-,\delta p_-\rangle-2\langle  p_-,n(q_-)\rangle\langle \delta q_-, \cK(q_-)\delta q_-\rangle+(t-\tau)\|\delta p_+\|^2\\
&=\langle  \delta q,\delta p\rangle+\tau\|\delta p\|^2-2\langle  p_-,n(q_-)\rangle\langle \delta q_-, \cK(q_-)\delta q_-\rangle+(t-\tau)\|\delta p_+\|^2\\
&>\langle  \delta q,\delta p\rangle.
\end{split}
\]
In other words, the cone field $\cC(q,p)$ is strictly invariant for the dynamics. By a general theorem, this, and the fact that the flow is Hamiltonian (hence symplectic), implies that all the Lyapunov exponents, apart from the one in the flow direction, are non-zero (see \cite[Theorem 6.8]{LW95} for details). 

In particular, it implies uniform hyperbolicity: Lebesgue almost surely, the tangent space at a point $\xi=(q,p)$ can be split in $E^u(\xi)\otimes E^0(\xi)\otimes E^s(\xi)$ such that $E^0$ is the flow direction, and  there exists $C>0$ and $\lambda>1$ such that, for all $t>0$,
\begin{equation}\label{eq:hyp_flow}
\begin{split}
&\|d\Phi_t v\|\geq C\lambda^t \|v\|\quad\textrm{ for all } v\in E^u\\
&\|d\Phi_{-t} v\|\geq C\lambda^t \|v\|\quad\textrm{ for all } v\in E^s.
\end{split}
\end{equation}
\subsection{The Poincar\'e map}\ \label{sec:Poicare}\\
It is often convenient to analyze the Poincar\'e section associated with the flow. This allows us to establish many properties for the Poincar\'e section and transfer them to the flow. A main example is ergodicity. However, it is well known that, in general, the mixing of the Poincar\'e section does not imply the mixing of the flow. The obstruction is the possibility that the intercollision time is cohomologous to a constant. But even if this is not the case, the speed of mixing can be seriously affected by the properties of the inter-collision time (see Ruelle \cite{Ru83} for an example).

A Poincaré map for a billiard can be constructed in many ways; historically, it has been used to map from just after one collision to just after the next collision. In this case, it is convenient to use a coordinate $r$ that describes the position on the boundary and spherical coordinate $(\alpha,\vf)$ for the velocity $p$, where $\langle p, n(q)\rangle=\cos \vf\geq 0$. In other words, the Poincaré map (called the \emph{billiard map}) $T$ acts on the space $M=\cup_i\partial \{(q, p)\in B_i\times \bS^{d-1}\;:\; \langle n(q), p\rangle\geq 0\}$, where $n(q)$ is the external unit normal to $\partial B_i$ at $q$.

The map $T$ has singularities corresponding to tangential collisions; in fact, crossing the preimage of a tangential collision, the next collision of the particle will take place with a different obstacle. It is an essential feature of billiard dynamics that not only is the map discontinuous, but also the derivative blows up at tangential collisions. 

To see this, let us consider the case $d=2$, which is the one discussed in the following. Then, calling $\bt$ the unit tangent to the boundary of the obstacle, directed counterclockwise, we have $p=n(r)\cos\vf+\bt(r)\sin\vf$. Let $(r_1,\vf_1)=T(r,\vf)$. Using the notation of the previous section we have $\delta q=\delta r\bt(r)$, then $\delta q_+=\delta r_1\bt(r_1)$ and, since  $\Pi(q_1,p)p=0$,\footnote{The second equality follows by expressing $p$ in the cordinates $\{\bt(r_1), n(q)\}$ and $\{\bt(r_1),n(q_1)\}$.}
\begin{equation}\label{eq:ausiliary1}
\begin{split}
\Pi^*(q_1,p)\bt(r_1)&=\frac 1{\cos\vf_1}\left[\bt(r_1)\cos\vf_1+n(q_1)\sin\vf_1\right]\\
&=\frac 1{\cos\vf_1}\left[-\bt(r)\cos\vf+n(q)\sin\vf\right]\\
\Pi^*(q_1,p_1)\bt(r_1)&=\frac 1{\cos\vf_1}\left[\bt(r_1)\cos\vf_1-n(q_1)\sin\vf_1\right]
\end{split}
\end{equation}
Also, differentiating $\langle p_1,n(q_1)\rangle=\cos \vf_1$, we have
\begin{equation}\label{eq:auxiliary2}
\langle \delta p_1,n_1\rangle=-\langle p,\bt(r_1)\rangle \cK(r_1)\delta r_1-\sin\vf_1\delta \vf_1
\end{equation}
where we have abused notation and used $\cK$ for both the scalar curvature and the curvature operator. A similar equation holds for $\delta p$, and since $\langle\delta p, p\rangle =0$, the above implies 
\begin{equation}\label{eq:auxiliary3}
\begin{split}
\delta p&=[\cK(r)\delta r+\delta \vf](\cos\vf \bt(r)-\sin\vf n(r))\\
&=-[\cK(r)\delta r+\delta \vf](\cos\vf_1 \bt(r_1)+\sin\vf_1 n(r_1))
\end{split}
\end{equation}
Then equation \eqref{eq:collder} implies\footnote{First compute $\langle \bt(r_1),\delta q_1\rangle$ and $\langle n(q_1), \delta p_1\rangle$ using \eqref{eq:collder}, then, for the second equality recall \eqref{eq:auxiliary2}, \eqref{eq:auxiliary3}.}
\begin{equation}\label{eq:poinc_der}
\begin{split} 
&\delta r_1=-\frac {\cos\vf+\tau \cK(r)}{\cos\vf_1}\delta r-\frac{\tau}{\cos\vf_1}\delta \vf\\
&\delta\vf_1=-\frac{\cK(r_1)\cos \vf +\tau\cK(r_1)\cK(r)+\cK(r)\cos\vf_1}{\cos\vf_1}\delta r-\frac{\cos\vf_1+\tau\cK(r_1)}{\cos\vf_1}\delta \vf.
\end{split}
\end{equation}
Finally, note that although the billiard map $T$ has singularities caused by grazing collisions, the SRB measure is smooth. Indeed, the Liouville measure is invariant for the flow (since the flow is Hamiltonian), hence it induces a smooth invariant probability measure on $M$. The SRB measure is thus given by
\begin{equation}\label{eq:SRB}
d\mu_{SRB}=c\cos \vf dr d\vf\;;\quad c^{-1}=\int_M\cos \vf dr d\vf=2 \sum_i |\partial B_i|.
\end{equation}

\subsection{Invariant foliations}\ \\
We have seen that there exists an invariant splitting of the tangent space. A natural question is whether there exist invariant manifolds. This is essential if one wants, for example, to try to use a Hopf-like argument to prove ergodicity of $T$. The answer is positive in some generality \cite{KS86}. However, the manifolds can be very short, and the foliation is only measurable. For uniformly hyperbolic billiards, this suffices to use a properly modified version of the Hopf argument to show ergodicity, as first shown by Sinai \cite{sinai} (see \cite{LW95} for a simple explanation). However, to establish exponential mixing for the flow, some more precise knowledge of the foliation regularity is needed. 
This has been achieved in \cite{bdl} for the case of two-dimensional billiards, where it is shown that the invariant foliations can be approximated by Lipschitz foliations on rather large sets. This is a consequence of a quantitative estimate on the curves of a given length that remain uncut and in the unstable direction for a long time. The precise statement, \cite[Theorem 6.2]{bdl}, is a bit technical and contains precise information on the regularity of approximate foliations, but can be roughly summarized as follows.
\begin{thm}[{\cite[Theorem 6.2]{bdl}}]\label{thm:app_fol_reg}
For each $\up,\rho>0$ and curve $W$ in the stable cone which does not belong to a homogeneity strip larger than $\rho^{-\frac 15}$,\footnote{See equation \eqref{eq:homogeneous} for the definition of homegenity strips; here the parameter $q=2$ is chosen.} there exists a foliation $\cF_\up$ of unstable curves of length at least $\rho$,\footnote{ That is, of curves in the unstable cone and perpendicular to the flow direction.} such that the foliation is made of $C^1$ curves whose tangents vary in a uniformly Lipschitz manner with Lipschitz constant bounded by $C\rho^{-\frac 45}$. Moreover, there exists $\Delta_\up\subset W$, such that if $\gamma\in\cF_\up$ and $\gamma\cap W\in\Delta_\up$, then $\Phi_{-t}(\gamma)$ is a connected unstable curve for all $t\in (0,\up)$.
In addition, $m_W(W \setminus \Delta_\up) \le C \rho$, where $m_W$ denotes arclength along $W$.
\end{thm}
The above is what is needed in the study of the ergodic properties of the billiard flow. Moreover, it also provides a rather detailed description of the regularity of invariant foliations.
\begin{thm}[{\cite[Remark 1.1]{bdl}}]
For each small enough $\rho > 0$, there exists a set of Lebesgue measure at most $C
\rho^{\frac 45}$, the complement of which is foliated by leaves of the
stable (or unstable) foliation with length at least $\rho$ such that the
foliation is Whitney--Lipschitz, with
Lipschitz constant not larger than $C\rho^{-\frac 45}$. An analogous
statement holds for the billiard map.
\end{thm}


\section{Geometric Potentials: Topological Entropy and Pressure}
\label{sec:geometric}


\subsection{Finite horizon Sinai billiard}

From this point on, we will restrict ourselves to the case $d=2$ in this survey.
Our setting is defined precisely as follows.  A billiard table is formed by placing finitely many closed, disjoint, convex obstacles (scatterers) on $\mathbb{T}^2$, whose boundaries are $C^3$ curves,  so that the resulting billiard has finite horizon. 

We identify the pre- and post-collision velocity vectors and parametrize the boundary of each scatterer according to arclength.  With these conventions, the coordinates for the billiard map are $(r, \vf)$, where $r$ denotes the position
on the boundary and $\vf \in [-\pi/2, \pi/2]$ denotes the angle made by the post-collision velocity vector with the normal to the boundary.  The phase space for the map is then
$M = \cup_{i=1}^d \partial B_i \times [-\pi/2, \pi/2]$, where $B_i$, $i=1, \ldots, d$ are the scatterers.

Letting $\cK(r)$ denote the curvature of $\partial B_i$ at $r$, we assume that $\cK \ge \cK_{\min}$ for some
$\cK_{\min}>0$, i.e. the boundaries have strictly positive curvature.  Let $\tau(r, \vf)$ denote the time between the collision at $(r, \vf)$ and that at $T(r, \vf)$.  
Our assumptions imply $0<\tau_{\min}\leq \tau\leq \tau_{\max}<\infty$.


\subsection{Transfer operators for geometric potentials}

We begin our discussion with the transfer operators associated to the family of geometric potentials
$\phi_t = - t \log J^uT$, where $J^uT$ denotes the Jacobian of $T$ along unstable manifolds (also called the
unstable Jacobian).  An invariant measure $\mu$ for $T$ is called an equilibrium state with respect to a potential
$\phi$ if $\mu$ attains the supremum in the expression
\begin{equation}
\label{eq:pressure}
P(\phi) = \sup \{ h_\mu(T) + \int \phi \, d\mu : \mu \mbox{ is a $T$-invariant probability measure} \}.
\end{equation}
The quantity $P(\phi)$ is called the pressure of the potential $\phi$.

The importance of the geometric family stems from the fact that the equilibrium state
$t=1$ corresponds to the Sinai-Ruelle-Bowen (SRB) measure for $T$, while the equilibrium state for $t=0$
corresponds to the measure of maximal entropy.  

In analogy with Anosov systems, the transfer operator $\widetilde{\cL}_t$ associated with the potential
$- t \log J^uT$ is given by
\[
\widetilde{\cL}_t f = \frac{f \circ T^{-1}}{((J^uT)^t J^sT) \circ T^{-1} } \, ,
\] 
where $J^sT$  is the stable Jacobian and $f$ is in a suitable class to be defined later.
However, it is convenient for the norms we will define to express the weight entirely in terms of the stable Jacobian.  Thus,
denoting by $\alpha(x)$ the angle between the stable and unstable subspaces of $T$ at $x$, we have,
following
\cite[eq. (1.5) and (1.13)]{BD2}, 
\begin{equation}
\label{eq:translate}
J^sT(x) J^uT(x) \frac{\sin \alpha(Tx)}{\sin \alpha(x)} = J_{ \mbox{\tiny Leb}}T(x) = \frac{\cos \vf(x)}{\cos(\vf(Tx))} \, ,
\end{equation}
where $J_{\mbox{\tiny Leb}}T$ is the Jacobian with respect to Lebesgue measure, and we have used \eqref{eq:SRB}.  This implies in particular that
\[
(J^uT)^t J^sT = \left( \frac{ \sin \alpha \cos \vf } { (\sin \alpha \cos \vf) \circ T } \right)^t (J^sT)^{1-t} \, .
\]
Since the presence of a coboundary should not affect the spectrum, we will use the above
calculation as motivation to work with the
transfer operators for $t \ge 0$ defined by
\begin{equation}
\label{eq:transfer}
\cL_t f = \frac{f \circ T^{-1}}{(J^s T)^{1-t} \circ T^{-1} } \, .
\end{equation}
Since the Jacobians are not smooth, we do not prove directly that $\cL_t$ and $\widetilde{\cL}_t$
have the same spectrum; however, we shall see from Theorerm~\ref{thm:thermo gap} that
$\cL_t$ produces the desired equilibrium state for the potential $-t \log J^uT$.

The simplest case is $t=1$ corresponding to the SRB measure for then the weight function in the transfer operator 
is simply 1.  For
$t \ne 1$, the presence of $J^sT$ provides significant technical difficulties since the potential can be arbitrarily small and is 
not continuous on any open set.

The study of the ergodic properties of $T$ with respect to $\musrb$ are classical, beginning with
\cite{sinai}.  Subsequent works used a variety of techniques such as Markov partitions and sieves
\cite{BSC1, BSC2}, Young towers \cite{Y98, chernov tower, chernov01} and coupling \cite{chernov06, chernov zhang}.
These studies proved statistical properties such as a dynamical Central Limit Theorem and eventually
exponential decay of correlations for both finite and infinite horizon Sinai billiards as well as dispersing billiard tables
in $\mathbb{R}^2$ with corner points (see \cite{de simoi toth} for the eventual resolution of the issue of corner points).

Yet it was not until \cite{demzhang11} that transfer operator techniques (still for $t=1$) were applied to 
 the finite and
infinite horizon Sinai billiard; these were then extended to perturbations of such systems \cite{demzhang13}
and to dispersing billiards with corner points and various return maps for billiards with focusing boundaries
\cite{demzhang14}.  They were also leveraged to prove exponential decay of correlations for the finite
horizon Sinai billiard flow, a long-standing open problem \cite{bdl}.
Developments then followed for $t=0$, corresponding to the measure of maximal entropy \cite{BD1},
and then for more general $t>0$ \cite{BD2}.  They were then extended to transfer operators for more general
potentials in \cite{Ca, bcd}.

It is this series of developments that we shall focus on primarily in the next sections: the recent application of transfer
operator techniques to make progress on these and related problems for dispersing billiards.   

The main objective of such an approach rests on the construction of Banach spaces for which
one can prove so-called D\"oblin-Fortet or Lastota-Yorke inequalities.  Obtaining such inequalities
for appropriately constructed norms establishes that the operator 
$\cL_t$ is {\em quasi-compact}:  its spectral radius is strictly larger than its essential spectral radius.  
From this, one obtains the above mentioned results by studying the peripheral spectrum of $\cL_t$, which
consists of isolated eigenvalues.  All this structure rests on a celebrated theorem of Hennion \cite{hennion},
applying Nussbaum's formula \cite{nussbaum} for the essential spectral radius.

\begin{thm}
\label{thm:hennion}[\cite{hennion}]
Let $(\cB, \| \cdot \|)$ be a Banach space and $\cL \in L(\cB, \cB)$. 
Assume there exists a continuous norm $| \cdot |_w$ on $\cB$, and
$0 < \theta < M$, $A, B, C > 0$, such that, for all $n \in \bN$ and all $f \in \cB$,
\[
| \cL^n f |_w \le C M^n |f|_w ; \quad \| \cL^n f \| \le A \theta^n \| f\| + BM^n |f|_w \, . 
\]
Then the spectral radius of $\cL \in L(\cB, \cB)$ is bounded by $M$. If, in addition, $\cL$ is
$| \cdot|_w$-compact,\footnote{By  $| \cdot|_w$-compact we mean that there exists an $n\in \bN$ such that $\cL^n\{f\in\cB\;:\;\|f\|\leq 1\}$ is reatively compact in the topology induce by the $|\cdot|_w$ norm.}
then the essential spectral radius of $T$ is bounded by $\theta$.
\end{thm}

The main trick, of course, is constructing appropriate norms for which the required contraction follows
from the dynamical properties of the map.  This is what we shall describe in the following sections.
For a self-contained exposition of the above abstract theory (including a proof of Theorem \ref{thm:hennion} and some generalizations), see \cite[Appendix~B]{DKL}.  For a quick introduction to the use of the Theorem \ref{thm:hennion}, see \cite{Lipisa, DKL}, as well as \cite{demers intro}, which provides a simple application of the technique in the hyperbolic setting, while \cite{Ba1,Ba2} provide a more in-depth discussion of various applications. 


\subsection{Topological pressure for $t>0$}
\label{sec:top t>0}

We define the notion of topological pressure for the geometric weights $(J^sT)^t$ as formulated in 
\cite{BD2}.  This notion is connected to the equilibrium state constructed from the associated transfer
operator by a variational principle.

For $t >0$ and $q>0$ such that $qt \ge 2$, we define homogeneity strips for $k \ge k_0$ by,
\begin{equation}
\label{eq:homogeneous}
H_k = \{ (r,\vf) \in M : (k+1)^{-q} \le \frac{\pi}{2} - \vf \le k^{-q} \},
\end{equation}
where $k_0 \in \mathbb{N}$ is chosen to guarantee the one-step expansion in Lemma~\ref{lem:one step}.  A similar definition applies to $H_{-k}$,
the homogeneity strips approaching $\vf = -\pi/2$.

The introduction of homogeneity strips allows for distortion control:  $J^sT$ is log-H\"older continuous with
exponent $1/(q+1)$ along trajectories that land in the same homogeneity at each iterate.  The homogeneity
strips also introduce a partition of the phase space on which we base our definition of topological pressure.
Define 
\[
\cS_0 = \{ (r, \vf) \in M : \vf = \pm \frac{\pi}{2} \}
\quad \mbox{and} \quad
\cS_0^{\bH} = \cS_0 \cup \left( \cup_{|k| \ge k_0} \partial H_k \right) \, .
\]
The extended singularity set for $T^n$, $n \in \mathbb{Z}$ is defined by $\cS_n^{\bH} = \cup_{i=0}^n T^{-i} \cS_0^{\bH}$.  Let
\[
\cM_0^{n,\bH} := \{ \mbox{maximal connected components of $M \setminus (\cS_{n-1}^{\bH} \cup T^{-n} \cS_0$} \},
\]
and define the weighted sum,
\[
\cQ_n(t) = \sum_{A \in \cM_0^{n,\bH}} \sup_{x \in A \cap M'} |J^sT^n(x)|^t,
\]
where $M' \subset M$ is the set of points where the stable Jacobian is defined.

\begin{defin}
\label{def:top press}
Define $P_*(t) = \lim_{n \to \infty} \frac 1n \log \cQ_n(t)$.
\end{defin}

The above limit exists and equals the lim inf due to the subadditivity of $\log \cQ_n(t)$.  It follows
that $\cQ_n(t) \ge e^{n P_*(t)}$ for all $n \ge 0$.  In addition, \cite[Proposition~2.4]{BD2} implies that
\begin{equation}
\label{eq:pstar}
P_*(t) \ge P(t) := \sup \{ h_\mu(T) - t \int \log J^uT \, d\mu : \mu \mbox{ \small is a $T$-invariant probability} \}. 
\end{equation}
The lower bound is straightforward using a classical argument (see \cite[Proposition~9.10]{walters}) and exploiting that $\int \log J^sT \, d\mu = - \int \log J^uT \, d\mu$ for any
invariant measure due to \eqref{eq:translate}.  The main work of \cite{BD2} is constructing a measure $\mu_t$ which achieves
the supremum in \eqref{eq:pstar}  and whose pressure equals $P_*(t)$.  For this purpose, one
studies the transfer operator $\cL_t$ acting on an appropriately constructed anisotropic Banach space
of distributions.  Indeed, \cite{BD2} studies $t$ in the interval $(0, t_*)$, where $t_*>1$ is defined by
\begin{equation}
\label{eq:t star}
t_* = \sup \{ t>0 : t > - P(t)/\log \Lambda \} \, ,
\end{equation}
and $\Lambda = 1 + 2\cK_{\min} \tau_{\min}$ is the minimum hyperbolicity constant
for the billiard table.  For such $t$, one has the following theorem.

\begin{thm}[\cite{BD2}, Theorems 1.1 and 4.1]
\label{thm:thermo gap}
For each $t_0 >0$ and $t_1 < t_*$, there exists a Banach space $\cB = \cB(t_0, t_1)$ such that for all $t \in [t_0, t_1]$, the transfer operator $\cL_t$ acting on $\cB$ has spectral radius equal to 
$e^{P_*(t)}$.  Moreover, $\cL_t$ has a spectral gap.\footnote{That is, $e^{P_*(t)}$ is a simple eigenvalue and the rest of the spectrum is contained in a disk of strictly smaller radius.} 

The eigenspace of $e^{P_*(t)}$ in $\cB$ contains a unique probability measure $\nu_t$.  Let $\tilde \nu_t$ denote the corresponding eigenvector
for the dual operator $\cL_t^*$. Then,\footnote{We use the notation $\langle \nu, \tilde \nu \rangle$ to denote the
pairing of elements $\nu \in \cB$ and $\tilde \nu \in \cB^*$.}
\begin{itemize}
  \item[a)] $\mu_t( \psi ) := \frac{ \langle \psi  \nu_t, \tilde \nu_t \rangle}{\langle \nu_t, \tilde \nu_t \rangle}$ for any smooth function $\psi$ on $M$ defines by density a $T$-invariant probability measure;
  \item[b)] $\mu_t$ enjoys exponential decay of correlations and charges all open sets;
  \item[c)] $\mu_t$ is the unique equilibrium measure for $- t \log J^uT$, and $P_*(t) = P(t)$.
\end{itemize}
\end{thm}

\begin{rem}
\label{rem:pressure gap}
The definition of $t_*$ is motivated by the fact that $\Lambda^{-t}$ controls the growth in local 
complexity due to the 
(extended) singularity set for the map subject to the weight $|J^sT|^t$, while $\cQ_n(t)$ controls the global growth
rate of the topological pressure.  The requirement $t < t_*$ implies $\Lambda^{-t} < e^{P_*(t)}$, which is a type of 
pressure gap condition that says the local complexity grows more slowly than the global pressure.

Since $P(t)$ is decreasing with $P(1)=0$ and $\Lambda >1$, it must be that $t_* > 1$.
\end{rem}

The fact that the Banach space can be fixed for all $t \in [t_0, t_1]$ for any $t_0 >0$ and $t_1 < t_*$ allows
one to prove as well that $P(t)$ is analytic on the interval $(0, t_*)$ \cite[Theorem~1.2]{BD2}.


\subsection{Uniform growth estimates for $t>0$}
\label{sec:growth t>0}

In this section, we provide some of the key ideas which allow one to control the iteration of the transfer operator
in the appropriate norms.  These involve obtaining a type of uniform bound on the evolution of local
stable manifolds under the dynamics.

To this end, define $\hW^s$ to be the set of cone-stable\footnote{By cone-stable curve, we mean a $C^2$ curve
whose tangent vector at each point lies in the stable cone.} curves whose curvature is bounded by a fixed constant,
which is chosen large enough so that $\hW^s$ is `invariant' under $T^{-1}$, i.e. for each $W \in \hW^s$,
$T^{-1}W$ is a union of elements of $\hW^s$.  Similarly, define $\cW^s \subset \hW^s$ to be the set of local
stable manifolds.

For $V \in \hW^s$, let $V_i$ denote the homogeneous connected components of
$T^{-1}V$ (which depend on $q$ and $k_0$ via \eqref{eq:homogeneous}).

\begin{lem}[Modified One-step Expansion]
\label{lem:one step}
Fix $t_0 >0$ and $q \ge 2/t_0$.  There exists $\theta(t_0) < 1$, $k_0(t_0), \delta_0(t_0) >0 $ such that 
for all $V \in \hW^s$,
\[
\sup_{|V| \le \delta_0} \sum_{V_i}  |J_{V_i}T|^t_{*} < \theta^t \, , \quad \mbox{for all $t \ge t_0$,}
\]
where $J_{V_i}T$ is the Jacobian of $T$ along the stable curve $V_i$, and $| \cdot |_*$ denotes the sup norm 
of the Jacobian with respect to an adapted metric that is uniformly equivalent to the Euclidean metric (see \cite[Section~5.10]{chernov book}).
Note that $\theta$ can be chosen arbitrarily close to $\Lambda^{-1}$ by making $k_0$ large.
\end{lem}

This is a generalization of the usual one-step expansion formulated by Chernov in the case $t=1$ (see, for example,
\cite[Lemma~5.56]{chernov book}).

Fixing $\delta_0$ from Lemma~\ref{lem:one step}, henceforth we subdivide all curves in $\hW^s$ to have length
at most $\delta_0$.

Lemma~\ref{lem:one step} is important since it establishes that the hyperbolicity of the billiard map dominates the
cutting due to singularities.  This in turn allows one to prove uniform growth estimates for $\cQ_n(t)$ by
establishing that most of the weight in the expressions for pressure is carried by `long pieces,'  i.e. pieces longer than
some fixed length scale.  This is summarized in the following lemma.

For $\delta_1 < \delta_0$ and $W \in \hW^s$, let $\cG_1^{\delta_1}(W)$ denote the maximal homogeneous connected components of $T^{-1}W$, subdivided as needed to have length at most $\delta_1$. 
Then inductively, given $\cG_{n-1}^{\delta_1}(W)$, we define $\cG_n^{\delta_1}$ comprising
$\cG_1^{\delta_1}(W_i)$ for each $W_i \in \cG_{n-1}^{\delta_1}(W)$.
(We call this the $n$th \emph{generation} of $W$.)

Since we will subdivide according to homogeneity strips, let us denote by $\cW^s_{\bH}$ those elements of
$\widehat \cW^s$ that lie in a single strip, and similarly for $\hW^s_{\bH}$.

\begin{lem}
\label{lem:long carry}
(Long elements carry most weight.)

\begin{itemize}
  \item[a)]   $\forall \ve>0$  $\exists \delta_1, n_1 >0$ s.t $\forall W \in \hW^s$ with $|W| \ge
\delta_1/3$ and all $n \ge n_1$,
\[
\sum_{\substack{W_i \in \cG_n^{\delta_1}(W_i) \\ |W_i| < \delta_1/3}}
|J_{W_i}T^n|^t_{C^0(W_i)} \le \ve \sum_{W_i \in \cG_n^{\delta_1}(W)} |J_{W_i}T^n|^t_{C^0(W_i)}
\; \; \;  \forall t \in [t_0,t_1] .
\]
\end{itemize}
Define $\cA_n(\delta) = \{ A \in \cM_0^{n, \bH} : \diam^u(T^nA) \ge \delta/3 \}$.

\begin{itemize}
  \item[b)] There exist $\delta_2 >0$ and $c_0 >0$ such that 
\[
\sum_{A \in \cA_n(\delta_2) }\sup_{x \in A} |J^sT^n(x)|^t 
\ge c_0 \, \cQ_n(t) \,,\,\, \forall n \in \mathbb{N}\, ,
\,\,  \forall t\in [t_0,t_1]  \, .
\]
\end{itemize}
\end{lem}

Part (a) follows from the one-step expansion and choice of $t_1 < t_*$.  Part (b) uses in addition
a type of distortion bound on elements of $\cM_0^{n ,\bH}$: 

\medskip
$\exists C>0$ s.t. for all $n \ge 1$, if $W_1, W_2 \in \hW^s_{\bH}$ are such that 
$W_1, W_2 \subset A \in \cM_0^{n, \bH}$, and all $x \in W_1$, $y \in W_2$,
\begin{equation}
\label{eq:general distortion}
\left| \log \frac{J_{W_1}T^n(x)}{J_{W_2}T^n(y)} \right| \le C \, .
\end{equation}
This bound uses that the stable cones are globally defined, even though the stable subspace is only measurable.

The dominance of long pieces can be leveraged to produce the following uniform growth estimates and
exact exponential growth of $\cQ_n(t)$.  In what follows, $\ve = 1/4$ is chosen and the corresponding
$\delta_1 > 0$ is fixed from Lemma~\ref{lem:long carry}.  This value of $\delta_1$ is used throughout and
$\cG_n^{\delta_1}(W) = \cG_n(W)$ is denoted for simplicity.

\begin{prop}
\label{prop:uniform growth}
a)There exists $c_1 >0$ s.t. for  any
$W \in \hW^s$ with $|W| \ge \delta_1/3$, 
\[
\sum_{W_i \in \cG_n(W)} |J_{W_i}T^n|^t_{C^0(W_i)} 
\ge c_1 \cQ_n(t) \, , \, \forall n \ge 1 \,, \, \forall t \in [t_0,t_1] \, .
\]
b) There exists $c_2>1$ s.t. for all $n \ge 1$,
\[
e^{nP_*(t)} \le \cQ_n(t) \le c_2 e^{nP_*(t)} \quad  \forall t \in [t_0, t_1] \, .
\]
\end{prop}
 
\begin{proof}[Sketch of Proof]
 We give a brief justification for the lower bound on growth (a), which consists of the following strategy.
\begin{itemize}      
  \item `Cover' $M$ with $k(\delta_2)$ Cantor rectangles $R_i$ such that any stable or unstable curve of length $\delta_2/3$ properly 
crosses at least one $R_i$.
  \item Define $\cA_n^i := \{ A \in \cA_n(\delta_2) \subset \cM_0^{n, \bH} : T^nA \mbox{ properly crosses } R_i \}$.
  
  \item  There must exist $i_*$ such that $\sum_{A \in \cA_n^{i_*}} \sup_A |J^sT^n|^t \ge \frac{c_0}{k} \cQ_n(t)$.
  \item Take $W \in \hW^s$ with $|W| \ge \delta_1/3 \ge \delta_2/3$.  It must properly cross one $R_j$.
  \item We use the mixing property of the SRB measure to ensure that a curve $V \subset T^{-N}W$ crosses $R_{i_*}$, where $N$ depends only on $\delta_2$.
  \item Then $\sum_{W_i \in \cG_n(V)} |J_{W_i}T^n|^t_{C^0}$ will be comparable to $\sum_{A \in \cA_n^{i_*}} \sup_A |J^sT^n|^t$ using the generalized distortion bound \eqref{eq:general distortion}.
\end{itemize}
To justify (b), note that denoting by $L_n^{\delta_1}(W)$ those elements of $\cG_n(W)$ having length at least
$\delta_1/3$,
\[
\sum_{W_i \in \cG^{\delta_1}_{n+k}(W)} |J_{W_i}T^{n+k}|^t_{C^0}
\ge C \!\! \sum_{V_j \in L^{\delta_1}_n(W)} \!\! |J_{V_j}T^n|^t_{C^0} \sum_{W_i \in \cG^{\delta_1}_k(V_j)} \!\! |J_{W_i}T^k|^t_{C^0} \, ,
\]
which together with part (a) and Lemma~\ref{lem:long carry} implies 
\[
\cQ_{n+k}(t) \ge \tilde{c}_2 \cQ_n(t) \cQ_k(t) \qquad \mbox{for some $\tilde{c}_2>0$}.
\]
This is a type of supermultiplicativity property for $\cQ_n(t)$, which immediately implies (b) \cite[Proposition~3.15]{BD2}.
\end{proof}

\begin{rem}
\label{rem:interpolation}
Although we present Proposition~\ref{prop:uniform growth} for all $t \in [t_0,t_1]$, in fact these properties are proved
separately for $t \le 1$ and $t>1$.  The first ingredient in the proof of Lemma~\ref{lem:long carry} is to obtain a good uniform lower bound on the growth of $\sum_{W_i \in \cG_n^{\delta_1}(W)} |J_{W_i}T^n|^t_{C^0}$.  For $t\le 1$, a 
sufficient first estimate is $\ge C \Lambda^{n(1-t)} |W| \delta_1^{-1}$ \cite[eq. (3.9)]{BD2}.

For $t>1$, an interpolation is necessary.  For $t>1$, we choose $s <1$ and $\alpha \in (0,1)$ such that
$1 = \alpha t + (1-\alpha)s$.  Then using the H\"older inequality, for any series of positive terms,
$\sum_i a_i \le (\sum_i a_i^t)^\alpha (\sum_i a_i^s)^{1-\alpha}$.  Applying this with $a_i = |J_{W_i}T^n|$
and using the appropriate lower bound for $s<1$, yields
\[
\sum_{W_i \in \cG_n^{\delta_1}(W)} |J_{W_i}T^n|^t_{C^0} \ge \frac{( \sum_i |J_{W_i}T^n|_{C^0} )^{1/\alpha}}
{(\sum_i |J_{W_i}T^n|^s_{C^0} )^{(1-\alpha)/\alpha} } \ge C e^{n P_*(s) (\alpha -1)/\alpha}.
\]
Note $\alpha = \frac{1-s}{t-s}$ implies $\frac{\alpha-1}{\alpha} = \frac{1-t}{1-s}$, so we can make the lower bound
arbitrarily close to $e^{- n (t-1) \chi}$ where $\chi = \lim_{s \to 1^-} \frac{P_*(s)}{1-s}$.   This means we can extend
Lemma~\ref{lem:long carry} and Proposition~\ref{prop:uniform growth} to a point $s_1 >1$ where
the right subtangent to $P_*(t)$ intersects $t \log \theta$.  Then we interpolate again using $s_1$ in place of 1.  This 
process can be continued until we reach any specified $t_1 < t_*$ as in Proposition~\ref{prop:uniform growth},
i.e. as long as the pressure gap persists (see Remark~\ref{rem:pressure gap}).
\end{rem}


\subsection{Banach spaces for geometric potentials}
\label{sec:geo banach}

The geometric estimates established above allow one to control the evolution of the transfer operator
$\cL_t$ on certain anisotropic Banach spaces, whose definition we recall for the reader's convenience.  
A useful perspective on the definition of these spaces is that they are constructed as a type of dual (at least in
spirit) to the notion of standard pairs.  A standard pair is a cone-unstable curve together with a log-H\"older density whose
regularity is preserved when pushed forward by a hyperbolic map; by contrast, our norms are defined as integration
against H\"older test functions on stable curves.
The definition of these norms is inspired by those developed first for smooth systems
in \cite{liverani gouezel, GL2} and then extended to piecewise hyperbolic maps in \cite{demers liverani}.

The transfer operators $\cL_t$ have weights given by of the stable Jacobian, $(J^sT)^{1-t}$.  While $J^sT$ is not continuous
on any open set, it is smooth along local stable manifolds.  To exploit this fact, we define our norms to integrate
along local stable manifolds in $\cW^s$ rather than the wider collection of cone-stable curves $\hW^s$.

\begin{rem}  
In the case $t=1$, the weight in the transfer operator
is simply 1 and the stable Jacobian does not appear.  This is the simplest (SRB) case in which cone-stable
curves $\hW^s$ can be used.  See for example, \cite{demzhang11, demzhang13, demzhang14}.  
One advantage of using $\hW^s$ is that one can easily study perturbations of the billiard map:  for many
classes of perturbations, a single set $\hW^s$ is common to all the relevant maps.  See Section~\ref{sec:sequential}
for an application of this idea to study sequential billiards \cite{dl cones}.
\end{rem} 

\subsubsection{Definition of Banach Spaces for $t>0$}\label{sec:Bt>0}

Before defining the relevant norms, we define notions of distance between curves and test functions on those curves.

Since the stable cone is bounded away from the vertical, we view each $W \in \cW^s$ as the graph of a function of the arclength coordinate $r$ over some interval $I_W$,
\begin{equation}
\label{eq:graph}
W = \{ G_W(r) : r \in I_W \} = \{ (r, \vf_W(r)) : r \in I_W \} \, .
\end{equation}

Given $W_1, W_2 \in \cW^s$, with defining functions $\vf_{W_1}$, $\vf_{W_2}$, let
\[
d_{\cW^s}(W_1, W_2) = | I_{W_1} \bigtriangleup I_{W_2} | + |\vf_{W_1} - \vf_{W_2} |_{C^1(I_{W_1} \cap I_{W_2})} \, ,
\]
if $W_1$, $W_2$ lie in the same homogeneity strip, and $d_{\cW^s}(W_1, W_2) = \infty$ otherwise.  While this distance
does not satisfy the triangle inequality, it does serve as an adequate measure of distance between curves in $\cW^s$.

If $d_{\cW^s}(W_1, W_2) < \infty$ then we also define a notion of distance between test functions on these curves.  Given
$\psi_i \in C^0(W_i)$, define
\begin{equation}\label{eq:psi_dist}
d_0(\psi_1, \psi_2) = | \psi_1 \circ G_{W_1} - \psi_2 \circ G_{W_2} |_{C^0(I_{W_1} \cap I_{W_2})} \, .
\end{equation}

Recall from \eqref{eq:homogeneous} that $q \ge 2/t_0$ defines the homogeneity strips. 
Also, $\cW^s_{\bH} \subset \cW^s$ denotes those curves in $\cW^s$ that lie in a single homogeneity strip.

Fix $0< \alpha \le 1/(q+1)$. 
For $f \in C^1(M)$, define the \emph{weak norm} of $\oldh$ by\footnote{ Recall that $m_W$ denotes arclength along $W$.}
\begin{equation}
\label{eq:weak}
|\oldh|_w = \sup_{W \in \cW^s_{\bH}} \sup_{\substack{\psi \in \cC^\alpha(W) \\ |\psi|_{\cC^\alpha(W)} \leq 1}} \int_W \oldh \, \psi \, dm_W \, .
\end{equation}
We define the weak space $\mathcal{B}_w$ to be the completion of $C^1(M)$ in the 
$| \cdot |_w$ norm.

Next choose $p > q+1$, $\beta \in (1/p, \alpha)$ and $\gamma < \min \{ 1/p, \alpha - \beta \}$.
Define the \emph{strong stable norm} of $\oldh$ by
\begin{equation}
\label{eq:stable t>0}
\| \oldh \|_s = \sup_{W \in \cW^s_{\bH}} \sup_{\substack{\psi \in \cC^\beta(W) \\ |\psi |_{\cC^\beta(W)} \leq |W|^{-1/p}}} \int_W \oldh \, \psi \, dm_W \, .
\end{equation}

Define the \emph{strong unstable norm} of $\oldh$ by
\begin{equation}
\label{eq:unstable t>0}
\|\oldh\|_u = \sup_{\ve \leq \ve_0} \sup_{\substack{W_1, W_2 \in \cW^s_{\bH} \\ d_{\cW^s}(W_1,W_2) \leq \ve}}
\sup_{\substack{|\psi_i|_{\cC^\alpha(W_i)} \leq 1 \\ d_0(\psi_1,\psi_2)=0}}
\ve^{-\gamma} \left| \int_{W_1}  \! \! \! \oldh \psi_1 \,  - \int_{W_2} \! \! \! \oldh \psi_2 \,  \right| \, .
\end{equation}

The \emph{strong norm} of $\oldh$ is defined to be 
$\| \oldh \|_{\cB} = \| \oldh\|_s + c_u \|\oldh\|_u$, for some $c_u > 0$ which is chosen in order to faciltiate the proof of the
Lasota-Yorke inequalities.
The strong space $\mathcal{B}$ is then defined as the completion of $C^1(M)$ in the $\| \cdot \|_{\cB}$ norm.

\begin{rem}
\label{rem:meaning}
Intuitively, one can understand the strong and weak norms as follows.  The strong stable norm of $\oldh$ uses test functions
which are less smooth, $C^\beta$, than those in the weak norm, $C^\alpha$.  Since the unit ball of $C^\alpha(W)$ is
relatively compact in $C^\beta(W)$ for $\beta < \alpha$, this relation provides compactness of our function spaces
restricted to a single curve.  The strong unstable norm imposes a type of averaged H\"older regularity for $\oldh$
in a direction approximately aligned with the unstable cone.  The test functions are again in $C^\alpha$, which
allows the norm to control the difference in weak norm on comparable elements of $\cW^s$.

The weight $|W|^{-1/p}$ in the strong stable norm forces the weak norm of an integral on a short stable curve to be small.  This is used both for compactness, allowing one to slide a stable curve slightly in the stable cone without changing the value of the weak norm much, as well as for controlling the effect of short pieces created by
discontinuities of the billiard map when the stable curves are iterated to prove the Lasota-Yorke inequalities.
\end{rem}


\subsubsection{Quasicompactness of $\cL_t$ and a Spectral Gap}

The definitions above and choices of parameters yield the following proposition, which establishes the
first step in the proof of Theorem~\ref{thm:thermo gap}.

\begin{prop}[\cite{BD2}]
\label{prop:LY}
Fix $0< t_0\leq t_1 < t_*$.  We may choose parameters for the norms so that $\cB = \cB(t_0, t_1)$ satisfies
the following.
\begin{itemize}
     \item[a)] There is a sequence of continuous inclusions,
\[
\cC^1(M) \subset \cB \subset \cB_w \subset (\cC^\alpha(M))^* .
\]
    \item[b)] The embedding of the unit ball of $\cB$ in  $\cB_w$ is compact. 
    \item[c)] (Lasota-Yorke inequalities) There exist
    $C, C_n >0$ such that for all $f \in \cB$, $t\in [t_0,t_1]$, and $n \ge 0$,\footnote{ See Lemma \ref{lem:one step} for the defintion of $\theta$.}
    \[
    \begin{split}
    | \cL_t^n \oldh |_w & \le C \cQ_n(t)  |\oldh|_w \, , \\
    \| \cL_t^n \oldh \|_s & \le C \big( \Lambda^{-(\beta - 1/p)n} \cQ_n(t) + \theta^{(t-1/p)n} \big) \| \oldh \|_s + C_n |\oldh|_w ) \\
    \| \cL_t^n \oldh \|_u & \le C \cQ_n(t) \big( n^\gamma \Lambda^{- \gamma n} \| \oldh\|_u  + C_n \| \oldh\|_s \big) \, .
    \end{split}
    \]
\end{itemize} 
\end{prop}

Thanks to Theorem~\ref{thm:hennion}, the geometric estimates given by Proposition~\ref{prop:uniform growth}-(b) together with the
pressure gap condition $\theta^t < e^{P_*(t)}$ ensures that we may choose $1/p < t_0$ so that 
 the spectral radius of $\cL_t$ on $\cB$ is at most
$e^{P_*(t)}$ and its essential spectral radius is strictly less than $e^{P_*(t)}$.

To conclude that $\cL_t$ is quasi-compact, we need a lower bound on the
spectral radius.
This follows from the uniform lower bound given by Proposition~\ref{prop:uniform growth}-(a) since
for any $W \in \cW^s_H$ and $n \ge 1$,
\[
\begin{split}
\int_W \cL_t^n 1 & = \sum_{W_i \in \cG_n(W)} \int_{W_i} |J_{W_i}T^n|^t \ge e^{-C_d} \sum_{W_i \in \cG_n(W)} |J_{W_i}T^n|^t_{C^0(W_i)} \\
& \ge e^{-C_d} c_1 \cQ_n(t) \ge e^{- C_d} c_1 e^{n P_*(t)} \, , 
\end{split}
\]
where we have used bounded distortion for the first inequality.
Thus $\| \cL_t^n 1 \|_s \ge C e^{n P_*(t)}$ and the spectral radius of $\cL_t$ is $e^{P_*(t)}$.

These results culminate in the following.

\begin{thm}[\cite{BD2}]
\label{thm:spectral t}
For each $0< t_0\leq t_1 < t_*$, there exists a Banach space $\cB = \cB(t_0,t_1)$
such that, for each $t\in[t_0,t_1]$, $\cL_t$ has a spectral gap and maximal simple eigenvector $e^{P_*(t)}$. 

Letting $\nu_t$ and $\tilde\nu_t$ denote the maximal right and and left eigenvectors for
$\cL_t$, define
\[
\mu_t(\psi) = \frac{\langle \nu_t, \psi \tilde\nu_t \rangle}{\langle \nu_t, \tilde\nu_t \rangle} \, ,
\quad \psi \in C^\alpha(M) \, .
\]
Then $\mu_t$ is an invariant probability measure for $T$, and enjoys 
exponential decay of correlations against H\"older observables;
$\mu_t$ has no atoms, gives 0 weight to any $C^1$ curve and is positive on open sets.  
Moreover, $\int | \log d(x, \cS_{\pm 1}) | \, d\mu_t < \infty$.  
\end{thm}

\begin{proof}[Idea of proof]
The exact exponential growth from Proposition~\ref{prop:uniform growth} implies
\[
\| \cL_t^n \|_{\cB} \le C \cQ_n(t) \le C' e^{n P_*(t)} \, ,
\]
so that the peripheral spectrum of $\cL_t$ has no Jordan blocks.  This plus quasi-compactness implies
the following spectral decomposition:
There exist a finite set $\{ \theta_j \}_{j=0}^{N}$, $\theta_0=0$, linear operators $\Pi_j, R: \cB \circlearrowleft$
satisfying $\Pi_i \Pi_j = \Pi_j R = R \Pi_j = 0$ with spectral radius of $R < 1$, such that
\[
e^{-P_*(t)} \cL_t = \sum_{j=1}^N e^{2\pi i \theta_j} \Pi_j + R \, .
\] 

From here, the proof of a spectral gap follows the usual lines.  (See \cite[Section~4.2]{demers intro} for details in a simple
context.)  Defining $\nu_t = \Pi_0 1$, one shows that all eigenvectors corresponding to the peripheral spectrum are measures absolutely continuous wrt $\nu_t$, and $\theta_j$ must be rational.  
Then working in $L^1(\nu_t)$, one uses the topological mixing property of the map to show 
1 is a simple eigenvalue for $e^{-P_*(t)} \cL_t^k$ for $k\ge1$.  This eliminates any possibility of other elements in the
peripheral spectrum. 
\end{proof}


\subsubsection{Entropy of $\mu_t$ and a Variational Principle}

With the existence of $\mu_t$ established by Theorem~\ref{thm:spectral t}, we next formulate a
related variational principle proving that $\mu_t$ is an equilibrium state for $- t \log J^uT$.

For $x \in M$, $n \in \bN$ and $\ve >0$, define the 
Bowen ball $B(x,n,\ve) = \{ y \in M : d(T^{-i}x, T^{-i}y) \le \ve, \forall i \in [0,n] \}$.

\begin{prop}[Measure of Bowen Balls]
\label{prop:bowen}
There exists $C >0$ s.t. for all $x \in M$, $n \ge1$, and $y \in B(x, n, \ve)$,
\[
\mu_t(B(x,n,\ve)) \le C e^{-n P_*(t) + t \log J^sT^n(T^{-n}y) } .
\]
\end{prop}

While the analogous lower bound fails for billiards, the upper bound on the measure of the Bowen balls is sufficient
to prove the desired variational principle.  Using \cite[Main Theorem]{brin katok}, we have
for $\mu_t$-a.e. $x \in M$,
  \[
   \lim_{\ve \to 0} \limsup_{n \to \infty} - \frac 1n \log \mu_t(B(x,n, \ve)) =
   h_{\mu_t}(T) .
   \]
 This plus Proposition~\ref{prop:bowen} implies 
  \[ 
  h_{\mu_t}(T)  \ge  P_*(t) - t \int \log J^sT \, d\mu_t 
  = P_*(t) + t \int \log J^uT \, d\mu_t  .
  \]
  But $P_*(t) \ge P(t)$ by \eqref{eq:pstar} implies $P_*(t)  \ge h_{\mu_t}(T) - t \int \log J^uT \, d\mu_t$. 
We conclude that  $P_*(t) = h_{\mu_t}(T) -  t \int \log J^uT \, d\mu_t = P(t)$. 


\subsubsection{Uniqueness of Equilibrium State}

The last item of Theorem~\ref{thm:thermo gap} left to prove is the uniqueness of the equilibrium state.
This can be proved using the concept of tangent measure.
By equations \eqref{eq:pressure}, and \eqref{eq:pstar} it follows $P(t) = P(-t \log J^uT)$.
We say $\mu$ is a $C^1$ {\em tangent measure} at $t$ if
\[
P(-t \log J^uT + \phi) \ge P(t) + \int \phi \, d\mu \, , \quad \mbox{for all $\phi \in C^1(M)$.}
\]
Using \cite[Theorem~9.14]{walters}, if $\mu$ is an equilibrium state for 
$-t \log J^uT$, then $\mu$ is a tangent measure

Next,
for each $\phi \in C^1(M)$, the
perturbed transfer operator defined by
\[
\cL_{t, z\phi} f = \frac{ f \circ T^{-1}}{(J^sT)^{1-t} \circ T^{-1}} e^{z \phi \circ T^{-1}} \, , \quad z \in \mathbb{C} \, ,
\]
is an analytic perturbation of $\cL_t$ with spectral radius $e^{P(t, z\phi)}$ \cite[Proposition~5.9]{BD2}, where
\[
P(t, z\phi) = \max \left\{ h_\mu(T) + \int (-t \log J^uT + z\phi) \, d\mu : \mu \mbox{ is $T$-invariant probability } \right\} \,.
\]  
This implies in particular that 
\[
\frac{d}{dz} e^{P(t,z\phi)} |_{z=0} = e^{P(t)} \int \phi d\mu_t \, .
\]
Taking $z \in \bR$, this together with the definition of tangent measure implies
\[
\begin{split}
\int \phi \, d\mu_t & = \lim_{z \downarrow 0} z^{-1} \big( P(t,z\phi) - P(t,0) \big) \ge \int \phi \, d\mu \\
\int \phi \, d\mu_t & = \lim_{z \uparrow 0} z^{-1} \big( P(t,z\phi) - P(t,0) \big) \le \int \phi \, d\mu
\end{split}
\]
Thus $\int \phi \, d\mu_t = \int \phi \, d\mu$ for all $C^1$ functions $\phi$ and since $C^1(M)$ dense in
$C^0(M)$, $\mu = \mu_t$ and the equilibrium state is unique.  This result is summarized and generalized
in \cite[Theorem~5.8]{BD2}.

\subsubsection{Analyticity of $P(t)$}

Since $J^sT$ is not piecewise H\"older, unfortunately \cite[Proposition~5.9]{BD2} does not imply that
$\cL_t$ is analytic as a function of $t \in (0,t_*)$.  For this, a separate set of arguments is needed. 
This is summarized in \cite[Theorem~1.2]{BD2}.

\begin{thm}
The function $t \mapsto P(t)$ is analytic on $(0, t_*)$, with
\[
P'(t) = \int \log J^sT \, d\mu_t = - \int \log J^uT \, d\mu_t < 0 \, ,
\]
\[
P''(t) = \sum_{k \ge 0} \left[ \int (\log J^sT \circ T^k) \log J^sT \, d\mu_t -
(P'(t))^2 \right] \ge 0 \, .
\]
Moreover, $P''(t) = 0$ if and only if $\log J^sT = f - f \circ T + P'(t)$ for some
$f \in L^2(\mu_t)$.  Finally, both $t \mapsto \int \log J^uT \, d\mu_t$ and 
$t \mapsto h_{\mu_t}$ are decreasing functions of $t$.
\end{thm}

The formula for $P'(t)$ implies that,
if there exist $t_1 \neq t_2$ in $(0, t_*)$ such that $\mu_{t_1} = \mu_{t_2}$, then
$P(t)$ is affine on $(0, t_*)$ and $\log J^sT$ is $\mu_t$-a.e. cohomologous to its
average $\int \log J^sT \, d\mu_t$
for all $t \in (0, t_*)$.


\subsection{Topological entropy: the case $t=0$}
\label{sec:entropy}

The spectral picture established by Theorem~\ref{thm:thermo gap} uses that $t>0$.  It is clear from the
pressure gap condition $\Lambda^{-t} < e^{P(t)}$ that this gap vanishes in the limit as $t \to 0$ since $P(0)=0$.
For this reason, the case $t=0$ is handled separately in \cite{BD1}, where no spectral gap is proved.  Nevertheless,
one can recover the existence of a measure of maximal entropy as well as a variational principle.

Prior to \cite{BD1}, Chernov studied the topological entropy of several classes of finite horizon billiards
in \cite{chernov entropy}.  He showed that the topological entropy of the finite horizon Sinai billiard is at least
as large as the topological entropy of the symbolic shift which is semi-conjugate to the billiard map via the construction 
of a countable Markov partition.  This entropy also gives a lower bound on the growth of periodic orbits.

In the results we survey here, we will describe a different approach, which approaches the problem from the
point of view of the associated transfer operator.
Notice that in this case, the transfer operator is given by $\cL_0 f = \frac{f \circ T{-1}}{J^sT \circ T^{-1}}$, so that integrating
along a local stable manifold $W \in \cW^s$ yields,
\[
\int_W \cL_0 f \, dm_W = \int_{T^{-1}W} f \, dm_{T^{-1}W} \, , 
\]
i.e., no power of $J^sT$ remains to allow us to sum over components of $T^{-1}W$.  Thus we cannot partition
$T^{-1}W$ according to homogeneity strips or the sum would immediately diverge.
For this reason, in the following, homogeneity strips are not used and we simply partition $T^{-1}W$ according to the
discontinuities of $T$.

Recall that $\cS_{\pm n} = \cup_{i=0}^n T^{\mp i} (\cS_0)$ denotes the singularity set for $T^{\pm n}$, $n \in \mathbb{Z}$,
Define
\[
\cM_{-k}^n = \mbox{ connected components of $M \setminus (\cS_{-k} \cup \cS_n)$. }
\]
The elements of $\cM_0^n$ are the domains of continuity for $T^n$ while $\cM_{-k}^0$ are the domains of 
continuity for $T^{-k}$.  We define the topological entropy of $T$ in terms of these sets.

\begin{defin}[Topological entropy of $T$]
\label{def:top ent}
Define $h_* = \lim_{n \to \infty} \frac 1n \log \# \cM_0^n$.
\end{defin}

Similar to the definition of $P_*(t)$ in Definition~\ref{def:top press}, the limit in the definition above exists since 
$\log \# \cM_0^n$ is subadditive, i.e. $\# \cM_0^{m+n} \le \# \cM_0^n \cdot \# \cM_0^m$.  In addition, 
since $\# \cM_0^n = \# \cM_{-n}^0$, we have immediately that $h_*(T) = h_*(T^{-1})$.

We note that $h_*$ counts the exponential rate of growth of the number of domains of continuity of $T^n$ and does not depend on the choice of metric.  However, if one defines $(n,\epsilon)$-separated sets and $(n,\epsilon)$-spanning sets in the usual way using the Euclidean metric, then the corresponding Bowen definitions of 
topological entropy using these sets equal $h_*$ \cite[Theorem~2.3]{BD1}.

Similar to \eqref{eq:pstar}, since $\cM_0^1$ is a generating partition for $T$ and the entropy of a partition is bounded by its cardinality, the same classical argument yields the lower bound \cite[Lemma 3.6]{BD1},
\begin{equation}
\label{eq:hstar}
h_* \ge \sup \{ h_\mu(T) : \mu \mbox{ is a $T$-invariant probability} \}.
\end{equation}

As in Section~\ref{sec:top t>0} the main work consists of constructing a measure $\mu_0$ whose entropy equals
$h_*$, and so attains the supremum.


\subsection{Uniform growth estimates for $t=0$}
\label{sec:growth t=0}

Before studying the action of the transfer operator, we prove some uniform growth estimates for stable curves.
The growth estimates for stable curves in the case $t=0$ follow a similar path to those described in
Section~\ref{sec:growth t>0}, except that the one-step expansion Lemma~\ref{lem:one step} is no longer
available.  Instead, we use the following linear complexity bound due to Bunimovich.

For $n \ge 1$, $x \in M$, let $N(\cS_n, x)$ denote the number of singularity curves in $\cS_n$ that meet at $x$.
Define $N(\cS_n) = \sup_{x \in M} N(\cS_n ,x)$.

\begin{lem}[Linear Complexity Bound \cite{BSC1}]
\label{lem:complexity}
There exists $K>0$, depending only on the finite horizon configuration of scatterers, such that $N(\cS_n) \le Kn$
for all $n \ge1$.
\end{lem}

This bound is the key to establishing the fragmentation lemmas which are the analog of Lemma~\ref{lem:long carry}.
Namely, choose $n_0$ such that
$n_0^{-1} \log (Kn_0 +1) < h_*$.  Then choose $\delta_0>0$ such that any $W \in \cW^s$ of length at most 
$\delta_0$ is cut into at most $Kn_0 +1$ pieces by $\cS_{-n_0}$.

For $\delta \le \delta_0$ and $W \in \cW^s$, define $\cG_n^\delta(W)$ to be the set of preimages of $W$ under $T^{-n}$, 
with long pieces subdivided according to the length $\delta$. Finally, define
$L^\delta_n(W) = \{ W_i \in \cG_n^\delta(W) : |W| \ge \delta/3 \}$ and
$Sh_n^\delta(W) = \cG_n^\delta(W) \setminus L_n^\delta(W)$.

\begin{lem}[Prevalence of long elements]
\label{lem:fragment}
a) For all $\ve >0$ there exists $n_1, \delta >0$ such that for all $n \ge n_1$,
\[
\# Sh_n^\delta(W) \le \ve \# \cG_n^\delta(W) \qquad \mbox{for all $W \in \cW^s$ with $|W| \ge \delta/3$}.
\]
b) Let $\ve = 1/4$ and choose the corresponding $\delta_1 >0$ from part (a).  Define
\[
\begin{split}
&L_s(\cM_0^n) := \{ A \in \cM_0^n : \diam^s(A) \ge \delta_1/3 \} \\
& L_u(\cM_{-n}^0) := \{ A \in \cM_{-n}^0 : \diam^u(A) \ge \delta_1/3 \} \, .
\end{split}
\]
There exists $c_0 >0$ such that for all $n \ge 1$,
\[
\# L_s(\cM_0^n) \ge c_0 \delta_1 \# \cM_0^n \quad \mbox{ and } \quad
\# L_u(\cM_{-n}^0) \ge c_0 \delta_1 \# \cM_{-n}^0 \, .
\]
\end{lem}

With $\ve= 1/4$ as in Lemma~\ref{lem:fragment}-(b) above, we fix the corresponding $\delta_1>0$ from part (a)
in what follows.  For simplicity, we denote $\cG_n^{\delta_1}(W) = \cG_n(W)$.

The prevalence of long elements in both $\cG_n(W)$ and $\cM_0^n$ allows one to establish a uniform
lower bound on the cardinality of $\cG_n(W)$ as well as the exact exponential growth of $\# \cM_0^n$, as
summarized in the following proposition.

\begin{prop}[Exact exponential growth of $\#\cM_0^n$]
\label{prop:exact exp}
There exist $c_1 > 0$, and $C_2 \ge 1$ such that for all $n \ge 1$,
\begin{itemize}
  \item[a)] for any $W \in \cW^s$ with $|W| \ge \delta_1/3$,
\[
\# \cG_n(W) \ge c_1 \# \cM_0^n ;
\]
  \item[b)]  $e^{nh_*} \le \# \cM_0^n \le C_2 e^{n h_*}$.
\end{itemize}
\end{prop} 

\begin{proof}[Sketch of Proof]
The proof of part (a) follows as in the proof of Proposition~\ref{prop:uniform growth}-(a), covering $M$ with
Cantor rectangles such that any stable or unstable curve of length $\delta_1/3$ properly crosses at least one $R_i$ and using the prevalence of long pieces in $\cG_n(W)$ to ensure the eventual crossing of a reference rectangle.
See \cite[Proposition~5.5]{BD1} for details.

Part (a) and Lemma~\ref{lem:fragment}-(a) imply a type of supermultiplicativity of $\# \cM_0^n$ since,
\[
\# \cG_{n+k}(W) \ge \sum_{V_j \in L_n^{\delta_1}(W)} \# \cG_k(V_j) \ge \# L_n^{\delta_1} c_1 \# \cM_0^k
\ge \tfrac{3\delta_1}{4} \# \cG_n \# \cM_0^k \ge \tfrac{3 c_1^2}{4} \# \cM_0^n \# \cM_0^k \, ,
\]
and $\# \cM_0^{n+k} \ge \delta_1^{-1} \cG_{n+k}(W)$.  Then part (b) follows immediately as in
\cite[Proof of Proposition~4.6]{BD1}.
\end{proof}

Proposition~\ref{prop:exact exp} yields the following result that stable curves grow at a uniform exponential rate given by
the topological entropy.

\begin{cor}
There exists $C>0$ such that for all $W \in \hW^s$ with $|W| \ge \delta_1/3$, and for all $n \ge n_1$,
\[
C e^{n h_*} \le |T^{-n}(W)| \le C^{-1} e^{n h_*} \, .
\]
\end{cor}
\begin{proof}
The upper bound is straightforward using Proposition~\ref{prop:exact exp}-(b) since $|T^{-n}(W)| \le \delta_1 \# \cG_n(W) \le \# \cM_0^n$.  For the lower bound, we combine Proposition~\ref{prop:exact exp}-(a) with the estimate,
\begin{equation}
\label{eq:length}
|T^{-n}(W)| = \sum_{W_i \in \cG_n(W)} |W_i| \ge \tfrac{\delta_1}{3} \# L_n^{\delta_1}(W) 
\ge \tfrac{\delta_1}{4} \# \cG_n(W) \, .
\end{equation}
\end{proof}


\subsection{Modified Banach spaces for $t=0$}
\label{sec:modified}

With the uniform growth estimates proved in the previous section, we are ready to introduce the relevant Banach
spaces adapted to the transfer operator corresponding to $t=0$.  Recalling \eqref{eq:transfer}, the operator is
\[
\cL_0 f = \frac{f \circ T^{-1}}{J^s T \circ T^{-1} } \, .
\]
Unfortunately, $\cL_0$ does not have a spectral gap on the Banach spaces we shall define, whose regularity
must be weakened due to the lack of homogeneity strips.  In order to bound on the action of $\cL_0$
sufficiently, we introduce a new assumption on the billiard table. 

For $n >0$ and $\vf_0 \in (0, \pi/2)$, define
\[
s_0(n, \vf_0) = \sup_{x \in M} \frac 1n \sum_{i=0}^{n-1} \Id_{\{ |\vf| \ge \vf_0 \}} \circ T^i(x) \, ,
\]
to be the maximum proportion of `nearly tangential' collisions of any orbit of length $n$.  Due to the finite
horizon condition, one can always choose $\vf_0$ and $n$ so that $s_0(n, \vf_0)<1$, and if there are 
no triple tangencies, so that $s_0(n, \vf_0) \le 2/3$.  Our assumption on the billiard table is the following:
\begin{equation}
\label{eq:s0}
\mbox{There exists $\vf_0 < \pi/2$ and $n_0 \ge 1$ such that 
$h_* > s_0(n_0, \vf_0)  \log 2$.}
\end{equation}

The factor $\log 2$ comes from the fact that if $W$ is a local stable manifold, $|T^{-1}W| \le C|W|^{1/2}$ and this
upper bound can be attained near tangential collisions \cite[Section~4.9]{chernov book}.
Assumption \eqref{eq:s0} requires that the exponential rate $h_*$ of the global complexity, as described by
Proposition~\ref{prop:exact exp}-(b), dominates the
local rate of recurrence to singularities of individual orbits.

\begin{rem}
\label{rem:generic}
Assumption \eqref{eq:s0} is verified for open sets of parameters for two families of finite
horizon tables in \cite[Section 2.4]{BD1}.  Moreover, the condition is not known to fail for any finite horizon 
Sinai billiard table.

Indeed, it is conjectured (see \cite[Conjecture~3.3]{Balint Toth}) that the complexity of the table as defined in Lemma~\ref{lem:complexity} is in fact bounded for generic 
finite horizon billiard tables, i.e. there exists $K >0$ such that $N(\cS_n) \le K$ for all $n \ge 1$.
Under this assumption, it is proved in \cite{DK} that $s_0$ can be chosen arbitrarily small, so that
condition \eqref{eq:s0} necessarily holds for tables with bounded complexity.
\end{rem}


\subsubsection{Definition of Banach Spaces for $t=0$}

We recall briefly the definition of Banach spaces constructed for $t=0$ in \cite{BD1}, pointing out the main differences with those for $t>0$ defined in Section~\ref{sec:geo banach}.  As already alluded to, homoegeneity
strips are not used and some of the  H\"older regularity is weakened to a logarthmic modulus of continuity.

Using Assumption~\ref{eq:s0}, we choose $\alpha, \beta, \varsigma > 0$ and $\gamma >1$ such that 
\[
\beta < \alpha \le 1/3, \quad 2^{s_0 \gamma} < e^{h_*}, \quad \varsigma < \gamma.
\] 
Recall $n_0$ chosen after the complexity bound, Lemma~\ref{lem:complexity}, and increase
it if necessary so that 
\[
\frac1{n_0} \log (K n_0+1) < h_* - \gamma s_0 \log 2 \, ,
\]
where $K$ is the constant from the complexity bound.  Fix $\delta_0$ so that any $W \in \cW^s$ with $|W| \le \delta_0$ is cut into at most $Kn_0 +1$ pieces by $\cS_{-n_0}$.  When we invoke the geometric
estimates of Section~\ref{sec:growth t=0}, we choose $\delta_1$ less than this $\delta_0$.

For $f \in C^1(M)$, we define the weak norm precisely as in \eqref{eq:weak} except that the supremum is 
taken over $W \in \cW^s$ rather than $W \in \cW^s_{\bH}$.

As before, the strong norm has two parts.  The {\em strong stable norm} of $f$ is defined as
\begin{equation}
\label{eq:stable t=0}
\| f \|_s = \sup_{W \in \cW^s} \sup_{\substack{\psi \in \cC^\beta(W) \\ |\psi |_{\cC^\beta(W)} \leq |\log |W| \, |^\gamma}} \int_W f \, \psi \, dm_W \, ,
\end{equation}
and (recalling the definitions of distance between functions from Section~\ref{sec:geo banach}, and dropping any 
reference to homogeneity strips) the \emph{strong unstable norm} of $f$ as
\begin{equation}
\label{eq:unstable t=0}
\|f\|_u = \sup_{\ve \leq \ve_0} \sup_{\substack{W_1, W_2 \in \cW^s \\ d_{\cW^s}(W_1,W_2) \leq \ve}}
\sup_{\substack{|\psi_i|_{\cC^\alpha(W_i)} \leq 1 \\ d_0(\psi_1,\psi_2)=0}}
| \log \ve|^{\varsigma} \left| \int_{W_1}  \! \! \! f \psi_1 \,  - \int_{W_2} \! \! \! f \psi_2 \,  \right| \, .
\end{equation}
Finally, the \emph{strong norm} of f is defined to be 
$\| f \|_{\cB} = \| f\|_s +  \|f\|_u$.  As before, the weak space $\cB_w$ and the strong space $\cB$ are defined
as the completions of $C^1(M)$ in the respective norms.

\begin{rem}
\label{rem:norm changes}
Comparing \eqref{eq:stable t>0} with \eqref{eq:stable t=0}, we see that the weight $|W|^{-1/p}$ has been replaced
by the weight $| \log|W| \, |^\gamma$.  This is done to avoid a blow up of the norm in the case when $W \in \cW^s$
is such that $|T^{-1}W| \sim |W|^{1/2}$.  Namely, with the H\"older weight,
\[
\int_W \cL_0 f \, \psi \, dm_W = \int_{T^{-1}W} f \, \psi \circ T \, dm_{T^{-1}W} \le \|f \|_s \frac{|W|^{1/2p}}{|W|^{1/p}}
\, ,
\]
and taking the supremum over $W \in \cW^s$ yields infinity.  Yet with the weight $|\log |W||^\gamma$, the same
calculation yields $\| \cL_0 f \|_s \le 2^\gamma \| f\|_s$.

Since the weight in the strong stable norm allows us to control the norm on short pieces created by singularities
(see Remark~\ref{rem:meaning} and the use of $\| \cdot \|_s$ in Proposition~\ref{prop:LY}-(c)), the change in
\eqref{eq:stable t=0} forces us to use the logarithmic modulus of continuity in the definition of the unstable
norm \eqref{eq:unstable t=0}.  It is this weakening that prevents us from proving genuine Lasota-Yorke
inequalities for these norms.  See Lemma~\ref{lem:pseudo LY} below.
\end{rem}
 
As in Proposition~\ref{prop:LY}-(a) and (b) the inclusions are continuous and the unit ball of $\cB$ is compact in
$\cB_w$.  The key difference is the corresponding Lasota-Yorke type inequalities,
which fail to contract $\| \cdot \|_u$ \cite[Proposition~4.7]{BD1}.

\begin{lem}
\label{lem:pseudo LY}
Suppose $h_* > s_0 \log 2$.  There exists $C>0$ and $\sigma <1$ such that for all $f \in \cB$ and $n \ge 0$,
   \[
    \begin{split}
    | \cL_0^n f |_w & \le C e^{n h_*} |f|_w \, , \\
    \| \cL_0^n f \|_s & \le C e^{n h_*} \big( \sigma^n \| f \|_s +  |f|_w \big) \\
    \| \cL_0^n f \|_u & \le C e^{n h_*} \big(  \| f\|_u  +  \| f\|_s \big) \, .
    \end{split}
    \]
\end{lem}

The uniform upper bounds on $\# \cM_0^n$ given by Proposition~\ref{prop:exact exp} are essential for the proof of this lemma.
While the above inequalities are not sufficient to prove the quasi-compactness of $\cL_0$ on $\cB$, they
do imply that the spectral radius of $\cL_0$ is at most $e^{h_*}$.  For the lower bound, \eqref{eq:length} 
together with
Proposition~\ref{prop:exact exp} implies that if $|W| \ge \delta_1/3$, then
\[
\| \cL_0^n 1 \|_s \ge | \cL_0^n 1|_w \ge \int_W \cL_0 1 \, dm_W = \sum_{W_i \in \cG_n(W)} |W_i|
\ge \tfrac{\delta_1}{4} \# \cG_n(W) \ge \tfrac{\delta_1 c_1}{4} e^{n h_*} \, ,
\]
so that the spectral radius of $\cL_0$ is precisely $e^{h_*}$.


\subsubsection{Construction of the measure $\mu_0$}

Despite the lack of contraction, the inequalities of Lemma~\ref{lem:pseudo LY} together with Proposition~\ref{prop:exact exp} are sufficient to allow us to construct an invariant measure using compactness.
The key points are summarized in the following proposition.

\begin{prop}[{\cite[Proposition~7.1]{BD1}}]
\label{prop:limit point}
There exists $\nu \in \cB_w$ and $\tnu \in \cB_w^*$ such that $\cL_0 \nu = e^{h_*} \nu$ and $\cL_0^* \tnu = e^{h_*} \tnu$.  $\nu$ and $\tnu$ are nonnegative Radon measures satisfying $\langle \nu, \tnu \rangle \neq 0$ and
$\| \nu \|_u \le C$.
\end{prop}

\begin{proof}[Sketch of proof]
The main point in the construction is that Proposition~\ref{prop:exact exp} and Lemma~\ref{lem:pseudo LY}
imply that the sequence $e^{-n h_*} \cL_0^n 1$ is uniformly bounded away from 0 and infinity in the strong
norm.  Thus defining,
\[
\nu_n = \frac 1n \sum_{k=0}^{n-1} e^{-k h_*} \cL_0^k 1 \, ,
\]
we have a sequence which is uniformly bounded in $\cB$.  By compactness, $(\nu_n)_n$ contains a subsequence
which converges to an element $\nu \in \cB_w$.  By construction, each $\nu_n$ is nonnegative, and thus so is $\nu$, making it a Radon measure.   Then since $\cL_0$ is continuous on $\cB$, one obtains from the definition that
$\cL_0 \nu = e^{h_*} \nu$.

Moreover, although $\nu$ is not necessarily an element of $\cB$, it does inherit the bound $\| \nu \|_u \le C$ since
it is the limit of functions with uniformly bounded strong norm. In particular, it is important here that the test functions
in the definitions of $| \cdot |_w$ and $\| \cdot \|_u$ are the same.  See \cite[eq.~(7.3)]{BD1}.

Similarly, we define a sequence,
\[
\tnu_n = \frac 1n \sum_{k=0}^{n-1} e^{-n h_*} (\cL_0^*)^k d\musrb \, .
\]
Since $| \tnu_n(f) | \le C|f|_w$ for all $f \in \cB_w$ and $n \ge 1$, we have $\tnu_n$ uniformly bounded
in $(\cB_w)^* \subset \cB^*$.  Again by compactness, we can find a subsequence converging to 
$\tnu \in \cB^*$ and satisfying $\cL_0 \tnu = e^{h_*} \tnu$.  By the density of $\cB$ in $\cB_w$, $\tnu$ also
extends to an element of $\cB_w^*$.  Finally, the uniform lower bounds given by Proposition~\ref{prop:exact exp}
imply that $\langle \nu_n , \tnu_n \rangle$ is uniformly bounded away from 0, yielding 
 $\langle \nu, \tnu \rangle \neq 0$.
\end{proof}

This construction immediately allows us to define $\mu_0$, which turns out to be the measure of maximal
entropy.  The following theorem summarizes the key results.

\begin{thm}[{\cite[Theorem~2.4]{BD1}}]
\label{thm:mu0}
Let $h_* > s_0 \log 2$.  Then
\[
h_* = \max \{ h_\mu(T) : \mbox{ $\mu$ is a $T$-invariant probability measure on $M$} \},
\]
and there exists a unique invariant probability measure $\mu_0$ with the following properties:
\begin{itemize}
  \item[a)] $h_{\mu_0}(T) = h_*$;
  \item[b)] $\mu_0$ gives positive mass to all open sets and zero mass to any smooth curve;
  \item[c)] $\mu_0$ is $T$-adapted;\footnote{ An invariant measure is called $T$-adapted if 
  $\int | \log d(x, \cS_{\pm 1}) | \, d\nu < \infty$.  For dispersing billiards, this is equivalent to 
  $\int \log \| DT \| \, d\nu < \infty$.   In general, these are independent notions, and both are used in the literature,
  cf. \cite[eq. (1.1) and (1.2)]{KS86}. }
  \item[d)]  $\mu_0$ is Bernoulli.
\end{itemize}
\end{thm}

\begin{proof}[Idea of Proof]
Using Proposition~\ref{prop:limit point}, the definition of $\mu_0$ is straightforward, 
\[
\mu_0(\psi) = \frac{\langle \psi \nu, \tnu \rangle}{\langle \nu, \tnu \rangle}, \quad \mbox{for $\psi \in C^1(M)$.}
\]
This map is nonnegative for nonnegative $\psi$ and so defines a measure.   Invariance follows since
$\nu$ and $\tnu$ are eigenfunctions of $\cL_0$ and $\cL_0^*$, respectively, for the eigenvalue $e^{h_*}$.

The properties of $\mu_0$ rely heavily on the fact that $\| \nu \|_u \le C$.  This permits the
following chain of reasoning, which establishes the key properties $\mu_0$.
\begin{itemize}
  \item {\em Hyperbolicity of $\mu_0$.}  The unstable norm bound implies that $\ve$-neighborhoods $N_\ve(S)$ of any stable or unstable curve $S$ satisfy
$\nu(N_\ve(S)) \le C |\log \ve|^{-\gamma}$.  This carries over to $\mu_0$ as well \cite[Lemma~7.3]{BD1} and
implies that $\mu_0$-a.e. $x$ has stable and unstable manifolds of positive length.  This same bound
establishes that $\mu_0$ is $T$-adapted.
  \item {\em Absolute continuity of the holonomy.}  The hyperbolicity of $\mu_0$ allows us to cover a full measure set of $M$ with locally maximal Cantor rectangles.  On such rectangles, the bound $\| \nu \|_u \le C$
  implies that the holonomy map between stable manifolds in each Cantor rectangle is absolutely continuous
  with respect to both $\nu$ and $\mu_0$ \cite[Corollary~7.9]{BD1}.
  \item {\em Ergodicity of $\mu_0$.}  Absolute continuity allows one to follow the usual Hopf argument for local 
  ergodicity and then use the topological mixing of the map to conclude ergodicity \cite[Proposition~7.16]{BD1}.
  These last two properties also combine with the uniform lower bounds of Proposition~\ref{prop:exact exp}
  to prove the full support of $\mu_0$, completing the proof of item (b).
  \item {\em Bernoulli property.}  The topological mixing of the map together with absolute continuity on Cantor rectangles establish that $(T^n, \mu_)$ is ergodic for all $n$.  This can be extended first to show that the Pinsker
  partition is trivial, establishing $K$-mixing.  Then following the strategy of \cite{chernov haskell}, one boosts
  $K$-mixing to Bernoulli by showing that the partition defined by $\cM_{-1}^1$ is very weakly Bernoulli, and
  then applying a theorem of Ornstein and Weiss \cite{ornstein weiss}; see \cite[Proposition~7.19]{BD1}.  This establishes item (c).
\end{itemize}

An estimate similar to Proposition~\ref{prop:bowen} shows that the measure of Bowen balls is bounded by
$\mu_0(B(x,n, \ve)) \le C e^{-n h_*}$.  This, coupled with \cite[Main Theorem]{brin katok}, implies as before
that $h_{\mu_0}(T) \ge h_*$.  Since the reverse inequality follows from \eqref{eq:hstar}, we conclude
$h_{\mu_0}(T) = h_*$ and both statement (a) and the variational principle are proved.

The only remaining claim is the uniqueness of the measure of maximal entropy.  This is proved in
\cite[Section~7.7]{BD1}, adapting the Bowen argument for uniqueness of the measure of maximal entropy.
The classical argument uses a lower bound on the measure of Bowen balls, $\mu_0(B(x,n,\ve) \ge c e^{-n h_*}$.
While this fails for billiards due to the rate of approach of typical orbits to the singularity sets, it holds that `most' Bowen balls
satisfy a good lower bound along a sufficiently dense sequence of times \cite[Lemmas~7.2 and 7.3]{BD1}.
This turns out to be sufficient to adapt Bowen's proof of uniqueness to this setting.
\end{proof}

\begin{rem}
While Theorem~\ref{thm:mu0} provides a fairly complete set of results regarding the measure of maximal entropy,
it uses the generic assumption $h_* > s_0 \log 2$ as well as a fair amount of machinery.  The recent preprint 
\cite{carrand direct} provides an alternative (shorter) proof of the existence of a measure of maximal entropy
under the (similarly generic) assumption that $\mu(\cS_{\pm 1})=0$ for all invariant measures $\mu$, by proving
the upper semi-continuity of the entropy function.   
\end{rem}


\subsubsection{Rate of Mixing of $\mu_0$}
\label{sec:MME rate}
The lack of a proof of a spectral gap for $\cL_0$ means that the usual method to prove exponential decay of
correlations is not available in this setting.  However, the question of rate of mixing 
is partly answered in \cite{DK}, under
the slightly stronger assumption that $h_* > s_0 \log 4$, which still holds generically (compare with \eqref{eq:s0}).
The main result there \cite[Theorem~2.1]{DK} is that for any $\ve \in (0, \frac{h_*}{s_0\log 2} - 2)$,
$\gamma \in (0,1]$ and $u, v \in C^\gamma(M)$, there exists $C_{\gamma, \ve}>0$ such that
\[
\left| \int u \, v\circ T^n \, d\mu_0 - \int u \, d\mu_0 \int v \, d\mu_0 \right| \le C_{\gamma, \ve} |u|_{C^\gamma}
|v|_{C^\gamma} n^{-\frac{h}{s_0 \log 2} + 2 + \ve}\, \mbox{ for all $n \ge 0$.}
\]
Since according to Remark~\ref{rem:generic}, $s_0$ can be chosen arbitrarily small for Sinai billiard tables with bounded complexity, the above
result predicts a super-polynomial rate of decay for tables satisfying this generic condition.

The main tool used in \cite{DK} is the construction of a symbolic Young tower built over a reference rectangle.  In 
contrast to the usual construction which requires an effective bound on the measure of the set of points which has not yet returned by time $n$, the construction adopted here simply counts itineraries and estimates the number of
cylinders returning at time $n$, which is proportional to $e^{n h_*}$, as well as those cylinders making a
`prime return' at time $n$, which is at most $e^{n h_*} / n^{\frac{h_* - \ve}{s_0 \log 2}}$.  This later expression is a 
type of bound on the pressure at infinity, which allows one to establish the announced rate of mixing in the 
symbolic model.  This is then passed to the billiard map via the natural semi-conjugacy.

\section{More General Potentials: Accessing the Billiard Flow}
\label{sec:abramov}

The family of geometric potentials described in Section~\ref{sec:geometric} is not the only interesting
choice to study for dispersing billiards.  Carrand \cite{Ca} studies transfer operators with an abstract
class of potentials $t g$, where $g$ is a piecewise H\"older continuous function on $M$ satisfying certain
conditions with respect to the singularities of the billiard map.  These are formally stated as (SSP1) and (SSP2)
\cite[Definitions 3.2 and 3.5]{Ca}, which resemble Lemma~\ref{lem:fragment}-(a), but with sums over
$W_i \in \cG_n(W)$
weighted by $e^{t S_n g}$ rather than the cardinality of $\cG_n(W)$.  Under these assumptions, Carrand
obtains results analogous to those in Theorem~\ref{thm:mu0}:  the construction of a unique equilibrium state for this potential for small $t$, 
without proving a spectral gap for the associated transfer operator.

\subsection{The MME for the billiard flow}

The main motivation of \cite{Ca} is to study the family $\phi_t = - t \tau$, where $\tau$ is the first collision time
under the billiard flow (assuming still the finite horizon condition).  This allows one to obtain some results
for the flow while still using techniques applied to the billiard map.  Setting
\[
\cP(t) = \sup \{ h_\mu(T) - t \int \tau \, d\mu : \mu \mbox{ is a $T$-invariant probability} \},
\]  
\cite{Ca} studies equilibrium states for this family.  The equilibrium state corresponding to $t=0$ is the measure
of maximal entropy for the map discussed in Section~\ref{sec:modified}.  If $t = \htop(\Phi_1)$, where $\Phi_1$ is the time-one map of the billiard flow,\footnote{This value of $t$ is finite since $\Phi_1$ is continuous.} then an equilibrium state for the map lifts to an MME for the flow.
This follows using Abramov's formula, since viewing $\Phi_t$ as a suspension over $T$ with roof function $\tau$,
any ergodic $\Phi_1$-invariant probability measure $\nu$ can be written as
\[
\nu = \frac{\mu}{\int \tau d\mu} \otimes Leb\, ,
\]
for some ergodic $T$-invariant probability measure $\mu$.  Then $h_\nu(\Phi_1) \int \tau d\mu = h_\mu(T)$.
Since $\Phi_1$ is a continuous map of a compact metric space, it is a classical result \cite{walters} that
\[
\sup \{ h_\nu(\Phi_1) : \nu \mbox{ is a $\Phi_1$-invariant probability} \} = \htop(\Phi_1).
\]
Thus by the Abramov formula, $\cP(\htop(\Phi_1))=0$.  Then if $\mu$ is an equilibrium state for this value of $t$,
necessarily, $h_\nu(\Phi_1) = \frac{h_\mu(T)}{\int \tau \, d\mu} = \htop(\Phi_1)$, so that $\nu$ is an 
MME for the billiard flow.  The argument also works in reverse \cite[Corollary~2.6]{Ca}.

Unfortunately, the value of $t = \htop(\Phi_1)$ may not be small, so the results of \cite{Ca} do not
apply directly.  This is remedied in \cite{bcd}, which uses a bootstrapping argument from \cite{BD2} similar to that described in
Remark \ref{rem:interpolation} to show that the equilibrium state for $t=\htop(\Phi_1)$ can be obtained under the condition,
\[
\htop(\Phi_1) \tau_{\min} >   s_0 \log 2 \, .
\]
Recalling Remark~\ref{rem:generic} again, this condition is satisfied for generic Sinai billiard tables.
As a result, such billiard flows enjoy a unique measure of maximal entropy, which is Bernoulli.

\section{Sequential Billiards}
\label{sec:sequential}

In this section, we describe how to treat sequential billiards, that is, a particle that moves in a billiard that may change shape. 
This corresponds to considering a sequence of different maps rather than the power of a single map. If we follow the evolution of a density for such a system in time, we clearly obtain a sequence of different transfer operators rather than the power of a single transfer operator. 

One possibility to address such a problem is to consider random systems (e.g. a composition of maps chosen according to a probability measure) and aim for almost sure results. Seminal examples are \cite{DFGV1, DFGV2}. However, this approach cannot be helpful for a sequential system, where one is interested in a specific sequence. Such a scenario is relevant in many applications and also as a technical tool to study skew products or fast-slow systems, e.g. \cite{DS16, DS18,DLPV}. 
     
Another standard approach to these situations is to use perturbative methods. Of course, this may work only if the changes are small. Noteworthy examples in which such a strategy is successfully applied can be found in \cite{syz13}, \cite{dpz20}.

However, if the system undergoes large changes over a brief period, perturbation theory cannot be applied.
This situation is similar to the product of different matrices, having in common only the fact that their elements are all positive. A powerful tool for studying this problem is to note that all such matrices strictly contract the cone of vectors with positive entries.
This has an analog for hyperbolic maps, as was first shown in \cite{Li95}. A basic tool for studying this problem is the Hilbert metric. For the reader's convenience, we briefly recap it in the next section.
\begin{rem}
An alternative to the Hilbert metric is Dolgopyat's standard pairs technology. A spectacular example of application of such a strategy can be found in \cite{CD09}. However, we will not review this approach here since its emphasis it is not on the transfer operator.
\end{rem}
\subsection{Convex cones and Hilbert metric}\ \label{sec:convex_cone}\\
Let $\bV$ be a normed vector space, and $\cC\subset\bV$ a strictly convex closed one-sided cone.\footnote{In fact, the following is true also for vector spaces $\bV$ with more general topologies, see \cite[Appendix D]{DKL} for details.} That is $v\in \cC$ implies $\lambda v\in \cC$ for all $\lambda>0$ and $\cC\cap (-\cC)=\{0\}$.  
\begin{defin}
For each pair $x,y\in \cC$, $\Theta(x,y)=0$ if they are linearly dependent. If they are linearly independent, consider the line $\ell=\{\lambda
x+(1-\lambda y)\;|\;
\lambda\in\bR\}$ passing through $x$ and $y$. Let $\{u,v\}=\partial
\cC\cap\ell$  and define\footnote{Remark that $u,\,v$ can also be $\infty$.}
\begin{equation}\label{eq:hilbert_zero}
\Theta(x,y)=\left|\ln\frac{\|x_*-u_*\|\|y_*-v_*\|}{\|x_*-v_*\|\|y_*-u_*\|}\right|
\end{equation}
\end{defin}
To convince oneself that the above is a metric, it is sufficient to check the triangle inequality; this can be done by looking at Figure \ref{fig:hilbert} since any three points $x,y,z$ lie on a plane.
\begin{figure}[ht]
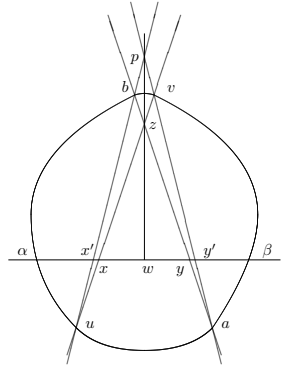
\ 
\centering
\scalebox{.6}{
\setlength{\unitlength}{1mm}
\put(-30,30){\line(1,0){60}}
\put(-17,9){\line(1,3){25}}
\put(17,9){\line(-1,3){25}}
\put(-17,7){\line(1,4){20}}
\put(17,7){\line(-1,4){20}}
\put(0,30){\line(0,1){50}}
\qbezier(0,10)(-10,10)(-15,15)
\qbezier(-15,15)(-25,25)(-25,40)
\qbezier(-25,40)(-25,55.1)(-2.14,66.4)
\qbezier(-2.14,66.4)(0,67.1)(2.14,66.4)
\qbezier(0,10)(10,10)(15,15)
\qbezier(15,15)(25,30)(25,40)
\qbezier(25,40)(25,55.1)(2.14,66.4)
\put(-10,27){$x$}
\put(7,27){$y$}
\put(-.5,27){$w$}
\put(-14,31){$x'$}
\put(13,31){$y'$}
\put(1,59){$z$}
\put(-13,15){$u$}
\put(17,15){$a$}
\put(-28,31){$\alpha$}
\put(26,31){$\beta$}
\put(-5,67){$b$}
\put(5,67){$v$}
\put(-3,74){$p$}
}
\caption{Hilbert metric}
\label{fig:hilbert}
\end{figure}
The Hilbert metric is the logarithm of a cross-ratio. Since the cross ratio is a projective invariant, it follows that $\Theta(x,y)=\Theta(\lambda x,\mu y)$ for all $\lambda, \mu>0$.
Moreover, if one of the points belongs to $\partial \cC$, then the distance to any other point is infinity. That is, the boundary is at infinity in this metric. This provides a criterion to say that a set $A$ is strictly contained in $\cC$ even if the space $\bV$ is infinite dimensional: we say that $A$ is strictly contained in $\cC$ if the diameter of $A$, in the Hilbert metric, is finite.

The above geometric definition is very intuitive and visual, but not very convenient for actual computations. Luckily, the Hilbert metric can also be described in a more algebraic manner, which emphasizes the relation of the present setting with order structures and Banach lattices. Here we present a simplified list of facts that will suffice for our present needs; see  \cite[Appendix D]{DKL} for a more general and complete theory.

To start with, note that a convex cone is equivalent to an order structure in $\bV$, in fact, let $f\preceq g$ iff $g-f\in\cC$. Since the cone is convex, the order structure is compatible with the linear structure: if $f\preceq g$ and $u\preceq v$, then $f+u\preceq u+v$ and $f,-f\in\cC$ iff $f=0$.\footnote{ Conversely, given an order structure compatible with the vector space, $\cC=\{v\in\bV\;:\; v\succeq 0\}$ is a convex cone.}
We can then define
\begin{equation}\label{eq:diameter}
\aligned
\alpha (f,\,g)=&\sup\{\lambda\in\mathbb{R}^+\;|\; \lambda f\preceq g\}\\
\beta(f,\,g)=&\inf\{\mu\in\mathbb{R}^+\;|\;g\preceq \mu f\}\\
\Theta(f,\,g)=&\log\left[\frac{\beta(f,\,g)}{\alpha(f,\,g)}\right]
\endaligned
\end{equation}
where we take $\alpha=0$ and $\beta=\infty$ if the corresponding sets are 
empty.  It is an exercise to prove that the definitions \eqref{eq:hilbert_zero} and \eqref{eq:diameter} are equivalent.

The key fact about the Hilbert metric is the following theorem, which is an extension of a theorem by G. Birkhoff 
(see, for example, \cite[Theorem~1.1]{Li95}).
\begin{thm}\label{thm:hilbert} Let $\bV_1$, and $\bV_2$ 
be two Banach spaces.
$\cL :\bV_1\to \bV_2$ a linear map such that 
$\cL (\cC_1)\subset \cC_2$, for two convex 
cones $\cC_1\subset\bV_1$ and $\cC_2\subset\bV_2$ with 
$\cC_i\cap-\cC_i=\{0\}$.
Let $\Theta_i$ be the Hilbert metric corresponding to the cone
$\cC_i$. Setting 
$\Delta=\sup\limits_{f,\,g\in \cL(\cC_1)}\Theta_2(f,\,g)$
we have
$$
\Theta_2(\cL f,\,\cL g)
\leq\tanh\left(\frac \Delta 4\right)\Theta_1(f,\,g)
\qquad \forall f,\,g\in\cC_1 ,
$$
where $\tanh(\infty)\equiv 1$.
\end{thm}
An obvious problem is to relate the above estimate to a norm contraction. This can often be achieved via the following result.
\begin{lem}\label{lem:relnorm} Let $(\bV,\|\cdot\|)$, $\bV \supset \cC$, be an order-preserving Banach space and $\ell$ an order-preserving functional. That is, for each $f,\,g\in\bV$, 
\[
\begin{split}
-f\preceq g\preceq f& \Longrightarrow \|f\|\geq\|g\|\\
 g\preceq f&\Longrightarrow \ell(f)\geq \ell (g).
\end{split}
\]
Then, for all $f,\,g\in\cC$ with $\ell(f) = \ell(g)>0$, 
$$
\|f-g\|\leq\left(e^{\Theta(f,\,g)}-1\right) \min\{\|f\|, \|g\|\} .
$$
\end{lem}
Of course, it may be possible that the Banach norm is not order preserving, yet if the order is Archimedean,\footnote{ By Archimedean we mean that there exists $\be\in\bV$ such that, for all elements $f\in\bV$ there exists $\lambda>0$ such that $f\preceq\lambda \be$.} then we can define the norm
\begin{equation}\label{eq:cone_norm}
\|f\|_*=\inf\{\lambda\;:\; -\lambda\be\preceq f\preceq\lambda \be\}.
\end{equation}
One can check that $\|\cdot\|_*$ is norm preserving. In fact, it dominates any order-preserving norm
\cite[Lemma~D.7]{DKL} .

The above allows us to connect the present theory with the study of the spectral properties of an operator, although the strength of the cones lies in the possibility of studying the product of different operators.
\begin{thm}\label{thm:gap_h} Let $(\bV, \|\cdot\|)$ be an order-preserving Banach space and $\cL:L(\bV,\bV)$ an order-preserving operator. Assume
\[
\Delta=\sup_{f,g\in\cC}\Theta(\cL f,\cL g)<\infty,
\]
then, setting $\nu = r(\cL)$, the spectral radius of $\cL$, and $\chi:=\tanh\left(\frac \Delta 4\right)$, there exists $\oldh_*\in \bV$ and $\ell\in \bV^*$ such that $\cL(f)=\nu \oldh_*\ell(f)+Q f$ where $\ell(\oldh_*)=1$, $Q \oldh_*=0$, $\ell(Q f)=0$, for all $f\in\bV$, and $\|Q^n\|\leq \chi^{n-1}\nu^n\Delta$.
\end{thm}

\subsection{Families of billiards and the problem}\ \label{sec:bill-families}\\
In this section, we briefly describe a simple family of billiards to which the following section applies. We consider billiard flows as described in \ref{sec:flowdef} in which the scatterers $B_i$ are pairwise disjoint. We will consider a set $\cQ$ of billiards tables such that each element of $\cQ$ has the same number of obstacles $B_i$, each with a boundary of fixed length $\ell_i$. We also require that there exists $\tau_*, \cK_*,E_*$ such that, for any element of $\cQ$, the distance between any two obstacles is larger than $\tau_*$. In addition, the curvature of the boundary of any obstacle is larger than $\cK_*$ and smaller than $\cK_*^{-1}$, while the third derivative of the boundary is bounded by $E_*$.

The requirement that the boundaries have the same length is not essential, but it simplifies the exposition, since it implies that all the Poincaré sections are the same: $M = \cup_i [0,\ell_i)\times [-\frac \pi2, \frac \pi2]$. We can then call $\cF$ the set of associated Poincarè maps $T:M\to M$.
See \cite{dl cones} for a more precise definition.

For any $x_0\in M$ and sequence $\{T_i\}_{i\in\bN}\subset \cF$, we are interested in studying the statistics of $x_n:=T_n\circ\cdots\circ T_1(x_0)$. 
Thus, given an initial probability distribution $\rho\in\cC^1(M,\bR_+)$ for $x_0$, we would like to understand, for each $\vf\in C^1(M,\bR)$,\footnote{In fact, one can consider as well a sequence of observables $\vf_n$, but let us keep it simple.}
\[
\bE(\vf(x_n)):=\int_M \vf(x_n)\rho(x_0) \mu_{SRB}(dx_0).
\]
Calling $\cL_i$ the transfer operator associated to the map $T_i$, we have
\[
\bE(\vf(x_n)):=\int_M \vf(x)\cL_n\cdots \cL_1\rho(x) \mu_{SRB}(dx).
\]
Not surprisingly, $\bE(\vf(x_n))\to \int_M\vf\mu_{SRB}$, but the problem is to estimate the speed of convergence. Moreover, setting 
\[
\vf_n(x)=\vf(x)-\bE(\vf(x_n)),
\] 
an important problem is to understand the size of the ergodic sums
\begin{equation}\label{eq:ergodic_sum}
S_n(x_0)=\sum_{k=0}^{n-1}\vf_n(x_n).
\end{equation}
In analogy with the case in which $\cF$ consists of just one billiard, one expects that $S_n$, even when properly normalized, converges to a limiting behavior only as a random variable. The natural normalization is given by $\sigma_n$, where
\[
\sigma_n^2=\bE(S_n^2),
\]
We are thus motivated to study the characteristic function
\begin{equation}\label{eq:char_funct}
\Phi_n(\lambda)=\bE(e^{i\lambda \sigma_n^{-1}S_n})=\int_M \cL_{n,\lambda,n}\cdots\cL_{n,\lambda, 1}\rho(x_0) \mu_{SRB}(dx_0)
\end{equation}
where, using the standard Nagaev-Guivarc'h trick, we have introduced the operators
\begin{equation}\label{eq:complex_potential}
\cL_{n,\lambda, k}g(x):=\cL_k(e^{i\lambda\sigma_n^{-1}\vf_n}g)(x).
\end{equation}

\subsection{Cones for billiards}\ \label{sec:cone def} \\
To study sequential billiards taken from the family  $\cF(\tau_*, \cK_*, E_*)$, it is necessary to consider structures relevant to the whole family of billiards, not just a single one. In particular, we will integrate on stable curves in 
$\widehat \cW^s_{\bH}$, as defined at the beginning of Section \ref{sec:growth t>0}, rather than the
exact stable manifolds $\cW^s_\bH$ of one particular map. More precisely, we may choose a single family $\widehat \cW^s_{\bH}$ such that $T^{-1} \widehat \cW^s_{\bH} \subset \widehat \cW^s_{\bH}$ for all $T \in \cF(\tau_*, \cK_*, E_*)$.  Thus for a fixed length scale $\delta>0$, we define 
\[
\cW^s_-(\delta) = \{ W \in \widehat\cW^s_{\bH} : |W| \le 2\delta \} \quad \mbox{and} \quad
\cW^s(\delta) = \{ W \in \widehat\cW^s_{\bH} : |W| \in [\delta, 2\delta] \} \, .
\]
Also, it will be convenient to consider positive test functions. We therefore define, for $\alpha \in (0,1]$, $a \ge 1$ and $W \in \cW^s_-(\delta)$, we define the following cone of test functions:  let $d(\cdot , \cdot)$ define the distance on $W$ induced by arclength,
\[
\cD_{a, \alpha}(W) := \left\{ \psi \in C^0(W) : \psi > 0, \frac{\psi(x)}{\psi(y)} \le e^{a d(x,y)^\alpha} \right\}.
\]
We can now define, for $g \in C^1(M)$, define the following two quantities,
\begin{equation}
\label{eq:tri def}
\tri g \tri_+ = \sup_{\stackrel{\scriptstyle W \in \cW^s(\delta)}{\psi \in \cD_{a,\beta}(W)}} \frac{\left|\int_W g \psi \, dm_W\right|}{\int_W \psi \, dm_W} ,
\hskip1.3cm
\tri g \tri_- = \inf_{\stackrel{\scriptstyle W \in \cW^s(\delta)}{\psi \in \cD_{a,\beta}(W)}} \frac{\int_W g \psi \, dm_W}{\int_W \psi \, dm_W} .
\end{equation}

Denote the average value of 
$\psi$ on $W$ by $\fint_W \psi\, dm_W=\frac{1}{|W|}\int_W \psi \, dm_W$.  
Since all of our integrals on $W \in \widehat \cW^s$
will be taken with respect to the arc-length $dm_W$, to keep our notation concise, we will drop the measure from our 
integral notation in the following. In fact, the above definition makes sense for all functions on $M$  whose restriction to each $W \in \cW^s$
is integrable with respect to the arclength measure $dm_W$ and for which $\tri\cdot\tri_+$ is finite, we call such a set $\cA$.

For exponents, $\alpha, \beta, \gamma, q \in (0,1)$, satisfying $\alpha \le 1/3$, $q < 1/2$, and
$\gamma \le \min \{ \alpha - \beta, q \}$, and constants $a, c, A, L >1$, $\delta >0$, and recalling 
$d_{\cW^s}$ from Section~\ref{sec:Bt>0}, we define the cone
\begin{align}
\cC_{c,A, L}(\delta)  =  \Bigg\{ & g \in \cA\setminus \{0\} : \nonumber\\
&\tri g \tri_+\leq L\tri g \tri_- ;
\label{eq:cone 1} \\
& \sup_{W \in \cW^s_-(\delta)} \sup_{\psi \in \cD_{a, \beta}(W)} |W|^{-q}\frac{|\int_W g \psi|}{\fint_W\psi}  \le  A \delta^{1-q} \tri g \tri_- ;
\label{eq:cone 2} \\
&\forall\, W^1, W^2 \in \cW^s_-(\delta):  d_{\cW^s}(W^1, W^2) \le \delta,  \forall \psi_i \in \cD_{a, \alpha}(W^i): d_*(\psi_1, \psi_2)=0,  \nonumber\\
&\left|\frac{\int_{W^1} g \psi_1}{\fint_{W^1}\psi_1}  - \frac{\int_{W^2} g \psi_2}{\fint_{W^2}\psi_2} \right|\leq
d_{\cW^s}(W^1, W^2)^\gamma \, \delta^{1-\gamma}   c A \tri g \tri_-  \Bigg\} .
\label{eq:cone 3}
\end{align}
Here, the distance $d_*(\psi_1, \psi_2)$ between test functions is defined with a slight difference
to the distance $d_0$ defined in \eqref{eq:psi_dist}.  If $d_{\cW^s}(W_1, W_2)< \infty$, we define
\[
d_*(\psi_1, \psi_2) = |\psi_1 \circ G_{W_1} \| G_{W_1}' \| - \psi_2 \circ G_{W_2} \| G_{W_2}' \| |_{C^0(I_{W_1} \cap I_{W_2}) } \, ,
\]
where $\| G_W' \| = \sqrt{1 + (d\vf_W/dr)^2}$, recalling \eqref{eq:graph}.  As we will show below, the difference
between using $d_*$ and $d_0$ is a technical one, made for convenience, and makes no difference to the eventual
estimates.

\begin{rem}
One can understand the three cone conditions as follows.  The first condition \eqref{eq:cone 1} enforces 
a type of positivity condition: $\sup/\inf$ should be bounded.  The second condition \eqref{eq:cone 2} controls the amount of
weight assigned to short pieces, i.e. $|\int_W g | \le C_g |W|^{q}$, useful for controlling the effect of cutting due to
singularities.  The third condition \eqref{eq:cone 3} requires a type of averaged H\"older continuity in the direction transverse to the stable curves, i.e. in the unstable cone.  One may consider the second and third conditions 
to be the
projective analogues of the stable and unstable norm definitions \eqref{eq:stable t>0} and \eqref{eq:unstable t>0}. See the proof of Theorem~\ref{thm:cone equiv} for more details on such a correspondence.
\end{rem}

\subsection{Loss of memory}\ \\
The loss of memory follows from cone contraction. To obtain cone contraction, it is necessary that the parameters of the cone satisfy several technical conditions. Such conditions essentially boil down to the requirement that $L,A,c$ be large and $\delta$ small enough (see \cite[Theorem 2.3]{dl cones} for the precise conditions).
If such conditions are satisfied, then we have the following.
\begin{thm}[{\cite[Theorem 2.3]{dl cones}}]\label{thm:loss_of_mem}
There exists $N\in\bN$ and $\chi\in (0,1)$ such that, for any sequence of billiard maps $\{T_i\}_{i=1}^N\subset \cF$,
\[
\cL_N\cdots \cL_1\cC_{c,A,L}(\delta)\subset \cC_{\chi c,\chi A,\chi L}(\delta).
\]

\end{thm}

The contraction of the parameters $c$ and $A$ follows from the hyperbolicity of the map and the control of the
cutting due to the singularities.  The proof for these parameters is similar to the
proof of the Lasota-Yorke inequalities in Proposition~\ref{prop:LY} for $t=1$.  The contraction of the parameter $L$
relies on a different mechanism:  the mixing of the maps $T_i$.  Indeed, its proof relies on a uniform mixing
argument for sequences of maps $T_N \circ \cdots \circ T_1$ on rectangles of scale $\delta$.
See \cite[Section~6]{dl cones} for details.

One can check that the diameter of $ \cC_{\chi c,\chi A,\chi L}(\delta)$ in $\cC_{c,A,L}(\delta)$ is finite, hence we can apply the previous results to this situation.
Recall that we can choose $\be=1$ in the definition \eqref{eq:cone_norm},  so that $\| 1 \|_* = 1$. Moreover, by construction $\cL_k1=1$. Thus
\[
\|\cL_k g\|_*\leq \|g\|_*\|\cL_k 1\|_*=\|g\|_*.
\]
Thus, $\|\cL_k\|_*= 1$ for all $k$.

Theorem \ref{thm:loss_of_mem}  implies that there exists $C_1>0$ such that, for all $\oldh_1,\oldh_2\in\cC^1\cap \cC_{c,A,L}(\delta)$, $\vf\in\cC^1$, such that $\int \oldh_1=\int \oldh_2=1$, $\|\oldh_1\|_*\geq \|\oldh_2\|_*$, $n\in\bN$ and sequence $\{T_i\}_{i\in\bN}$ we have, by Lemma \ref{lem:relnorm} and Theorem \ref{thm:hilbert},
\[
\begin{split}
&\left|\int_M (\oldh_1 -\oldh_2)\vf\circ T_n\circ\cdots\circ T_1 d\mu_{SRB}\right|
=\left|\int_M \cL_n\cdots \cL_1(\oldh_1 -\oldh_2)\cdot \vf d\mu_{SRB}\right|\\
&\leq C_2\|\cL_n\cdots \cL_1(\oldh_1 -\oldh_2)\|_*\|\vf\|_{\cC^1}\leq C_2\left[ e^{\Theta(\cL_n\cdots \cL_1\oldh_1, \cL_n\cdots \cL_1 \oldh_2))}-1\right]\|\vf\|_{\cC^1}\|\oldh_1\|_*\\
&\leq C_2\left[ e^{\beta^nC_3}-1\right]\|\vf\|_{\cC^1}\|\oldh_1\|_{\cC^1}\leq C_2C_3\beta^{n}\|\vf\|_{\cC^1}\|\oldh_1\|_{\cC^1}.
\end{split}
\]
Which implies that the initial distribution of $x_0$ is forgotten exponentially fast.

\subsection{Central limit theorem}\ \\
Finally, we address the problem of the fluctuations of the ergodic sums \eqref{eq:ergodic_sum}. By equation \eqref{eq:char_funct}, the problem of computing the characteristic function seems very similar to the one addressed in the previous section, yet there is a crucial difference: the transfer operators \eqref{eq:complex_potential} are complex.
At first sight, the problem is very serious: the Hilbert metric seems to rely on the order structure of the real numbers, an order not present in the complex plane.
Surprisingly, Rugh showed that it is possible to extend the ideas of invariant cones to the complex case \cite{Ru10}. Rugh's work has been further refined by Dubois \cite{Du09, Du11}.  
\subsubsection{Complex cones}\ \\
The idea is to construct a complex cone starting from a real one. Let $\cB_\bR$ be a real Banach space with a real cone $\cC_\bR$ as in Section \ref{sec:convex_cone}, and assume that the order associated to the cone is Archimedean: there is an element $\be\in \cC_\bR$ such that for each $\oldh \in\cB_\bR$ there exists $\lambda>0$ such that $\lambda \be\succeq \oldh $. We ask the cone to have the following, rather general, form: given a set $\cS\in \cB_{\bR}'$ such that $\ell(\oldh)=0$ for all $\ell\in\cS$ implies $\oldh=0$, define
\begin{equation}\label{eq:com_cone_def0}
\cC_\bR=\{\oldh \in \cB_{\bR}\setminus \{0\}\;:\; \ell(\oldh )\geq 0, \forall \ell\in\cS\}.
\end{equation}  
In addition, assume there exists a functional $\bbm$ and $\kappa\in (0,1)$ such that, for all $\oldh \in\cC_\bR$,
\begin{equation}\label{eq:maybetoomuch}
\bbm(\be)=1\;;\quad \bbm(\oldh )\geq \kappa \|\oldh \|_*.
\end{equation}
Then we can define the Banach space $\cB_{\bC}$ obtained by complexifying $\cB_\bR$ and a complex cone $\cC_\bC:=\bC_* \cdot (\cC_\bR+i\cC_{\bR})$, where $\bC_*=\bC\setminus\{0\}$.
We may also write $\cC_\bC = \bC_1 \cdot (\cC_\bR + i \cC_\bR)$, where $\bC_1 = \{ z \in \bC : |z| =1 \}$. We also define the dual cone by 
\[
\cC'_\bC=\{\ell\in \cB'_\bC\;:\; \ell(\oldh)\neq 0\quad \forall \oldh \in\cC_\bC\} .
\]
In order to construct a metric on the cone, for each $\oldh ,\oldf\in\cC_{\bC}$ define
\begin{equation}\label{eq:gauge}
E_\cC(\oldh ,\oldf)=\left\{\frac{\ell(\oldh )}{\ell(\oldf)}\;:\; \ell\in\cC'_{\bC}\right\} \, .
\end{equation}
We can then define the object that takes the place of the Hilbert metric in the present context,
\begin{equation}\label{eq:gauge_dist}
\delta_\cC(\oldh ,\oldf)=\ln\frac{\sup_{z\in E_\cC(\oldh ,\oldf)}|z|}{\inf_{z\in E_\cC(\oldh ,\oldf)}|z|}=\sup_{z,w\in E_\cC(\oldh ,\oldf)}\ln\left|\frac z w\right| \, .
\end{equation}
The above construction allows the equivalent of the Hilbert metric contraction.
\begin{thm}[{\cite[Theorem 3.1 (iii)]{Du11}}]\label{thm:complex_contraction}
Let $\cB_1,\cB_2$ be complex Banach spaces with complex cones $\cC_1,\cC_2$, as above. Let $\cL_\bC\in L(\cB_1,\cB_2)$, such that $\cL_\bC(\cC_1)\subset \cC_2$ and 
$\Delta=\sup_{g,f\in\cC_1}\delta_{\cC_2}(Lg,Lf)$. Then, for all $\oldh ,\oldf\in\cC_1$ we have
\[
\delta_{\cC_2}(\cL_\bC \oldh ,\cL_\bC \oldf)\leq \tanh(\Delta/4)\delta_{\cC_1}(\oldh ,\oldf).
\]
\end{thm}
Once more, we want to connect the contraction in the above sense with a norm contraction. This can be done thanks to the following.
\begin{lem}\label{lem:compare}
Let $\oldh ,\oldf\in\cC_{\bC}$, such that $\bbm(\oldh)=\bbm(\oldf)$, $|\bbm(\oldf)|=1$.
Then
\[
\|\oldh-\oldf\|_* \leq \frac{\sqrt2}{\kappa}\delta_{\cC}(\oldh ,\oldf).
\]
\end{lem}
It remains only to establish the strict contraction of the cone. Luckily, in many cases, one can use the following criteria,
 which is \cite[Theorem~5.17]{dl cones}, an adaptation of \cite[Theorem~4.5]{Du11}.
\begin{thm} \label{thm:needed} 
Let $\cL\in L(\cB_{1,\bR},\cB_{2,\bR})$ and $\cL_\bC\in L(\cB_{1,\bC},\cB_{2,\bC})$.
Assume that $\cL(\cC_{1,\bR})\subset \cC_{2,\bR}$ and $\diam_H(\cL(\cC_{1,\bR})):=\Delta_\bR<\infty$. 
If there exists $\ve \in (0, \frac{\kappa^2}{12 \sqrt 2}e^{-2\Delta_\bR})$, such that
for all $\ell \in \cS_2$ and all $\oldh \in \cC_{1,\bR}$,
\begin{equation}\label{eq:complex_bound}
\left|\ell(\cL_{\bC} \oldh)-\ell(\cL \oldh)\right|\leq \ve \ell(\cL \oldh).
\end{equation}
Then $\cL_{\bC}(\cC_{1,\bC})\subset \cC_{2,\bC}$, and we have
\[
\diam_{\delta_{\cC_{2, \bC}}} (\cL_{\bC}(\cC_{1,\bC}))\leq 8\Delta_\bR +
 2 \ln[3\sqrt 2\kappa^{-2}]+ \tfrac{\sqrt{2}}{3} \kappa^2e^{-2\Delta_\bR} . 
\]
\end{thm}

\subsubsection{CLT for sequential systems}\ \\
Consider the setting described in section  \ref{sec:bill-families}. Then, to study the characteristic function, we have to study sequences of operators with complex weight $\cL_{k-1,\lambda, n}$ defined in \eqref{eq:complex_potential}. In \cite{dl clt} it is shown that there exists $\lambda_0>0$ such that, for all $|\lambda|\leq \lambda_0$, it is possible to apply Theorem \ref{thm:needed} to prove that the complex cone associated to the cone $\cC_{c, A,L}(\delta)$ contracts strictly for maps in the set $\cF$.
Then Theorem \ref{thm:complex_contraction} and Lemma \ref{lem:compare} imply that there exists $\lambda_0, c, K>0$ and elements $g_{k,j,\lambda}\in\cB$, $\ell_{k,j,\lambda}\in\cB'$, $k,j\in\bN$, such that, for all $|\lambda|\leq \lambda_0\sigma_n$ and $k,j,l,n\in\bN$,\footnote{ For $\ell\in \cB'_j$, we have $\|\ell\|_{j}' =\sup_{\|\oldh \|_j\leq 1}|\ell(\oldh)|$.}
\[
\begin{split}
&\int_{M_k} g_{k,j,\lambda}d\mu_k=\ell_{k,j,\lambda}(\be)=1, \\
&\|g_{k,j,\lambda}\|_{k}+ \|\ell_{k,j,\lambda}\|_{j}' \leq K,  \\
&|\ell_{k,j,\lambda}(g_{j,l,\lambda})|\geq K^{-1}.
\end{split}
\]
Morever, there exist $\alpha_{k,j,\lambda}\in\bC$, $|\alpha_{k,j,\lambda}|\leq K$ such that for all  $j,k,n\in\bN$, $j \le k \le n$,  and $g\in \cB$ we have 
\begin{equation}
\label{eq:O3 decay}
\| \cL_{k-1,\lambda, n}\cdots \cL_{j,\lambda, n} \oldh -\alpha_{k,j,\lambda}g_{k,j,\lambda}\ell_{k,j,\lambda}(\oldh)\|_k \leq K |\alpha_{k,j,\lambda}|e^{-c(k-j)}\|\oldh \|_j.
\end{equation}
The above is the equivalent of the exponential loss of memory for the real operators. 
This allows us to apply an abstract result  (\cite[Theorem 2.7]{dl clt}) to establish the CLT. Let us mention just the special case relevant for the present discussion.

\begin{thm}\label{thm:main_bis}
Assume $\vf$ in \eqref{eq:ergodic_sum} satisfies, for each $\oldh \in\cB$, and $j\leq 3$,
 \[
 \|\cL_k (\vf^j \oldh)\|_* \leq K^j\|\oldh \|_*.
 \] 
Then, there exist a constant $\lambda_1>0$ and an analytic function $A_n$ such that, for each $n\in\bN$ and $
|\lambda|\leq \lambda_1\sigma_n(\ln \sigma_n)^{-1}$,
we have
\[
\begin{split}
&\phi(\lambda):=\bE\left(e^{i\lambda \sigma_n^{-1} S_n}\right)=e^{-\frac{\lambda^2}2+A_n(\lambda)}\\
&\left|A_n(\lambda)\right|\leq C_\varpi n\left[ \lambda^3\sigma_n^{-3}(\ln\sigma_n)^2 +\sigma_n^{-4}|\lambda|\right]\\
&\left|A_n'(\lambda)\right|\leq  C_\varpi n\left[ \lambda^2\sigma_n^{-3}(\ln\sigma_n)^2 +\sigma_n^{-4}\right].
\end{split}
\]
\end{thm}
The above implies that the ergodic sum $\sigma_{n}^{-1}S_n$ converges to a standard normal random variable, but we are interested also in the error term. Let 
\[
G(x)=\int_{-\infty}^x\frac{e^{-\frac{y^2}2}}{\sqrt{2\pi}}dy
\]
 be the standard normal distribution and $F(x)$ the distribution of $\sigma_{n}^{-1}S_n$. Then, by \cite[Equation (3.13) of Chapter XVI.3]{feller:2}] we have that for all $T\in\bR_+$,
\[
\left| F(x)-G(x)\right|\leq \frac 1{\pi}\int_{-T}^T\left|\frac{\phi(\xi)-e^{-\frac{\xi^2}2}}{\xi}\right| d\xi+\frac{24}{\pi T}.
\]
Choosing $T_n\sim \frac{\sigma_n^3}{ n(\ln \sigma_n)^2}$, and using Theorem \ref{thm:main_bis}, yields that there exists $C>0$ such that
\begin{equation}\label{eq:final_est2}
\left|F_n(x)-\frac{1}{\sqrt{2\pi}}\int_{-\infty}^xe^{-\frac {y^2}2}dy\right|\leq C \sigma_n^{-3}(\ln \sigma_n)^2 n.
\end{equation}
Note that such an equation provides non-trivial information only if $\sigma_n\gg n^{\frac 13}$. To obtain better results, it would be necessary to obtain estimates of higher momenta in terms of the second moment $\sigma_n^2$. This has been achieved for one-dimensional expanding maps in \cite{DH25}, but the extension of these ideas to hyperbolic systems is lacking.

\subsection{Application to concrete examples}\ \label{sec:application_sequential} \\
The theory described in the previous sections allows us to study many types of systems. Here we provide some relevant examples.
\subsubsection{Billiards with holes}\ \\
 For a small hole $H$ with some regularity properties (essentially, their boundary cannot be exactly aligned to the stable direction), it is possible to prove, see \cite[Section 8.1]{dl cones},
\[
\Id_{H^c}[\cC_{c,A,L}(\delta)]\subset  \cC_{c',A',L'}(\delta),
\]
where the parameters $c',A',L'$ are larger, but close to $c,A,L$ when the hole is small. In other words, the multiplication by the characteristic function of the complement of the hole is well-behaved with respect to the cone structure. Then Theorem \ref{thm:loss_of_mem} implies, for holes sufficiently small,
\[
\cL(\Id_{H^c}[\cC_{c,A,L}(\delta)])\subset  \cC_{\chi' c,\chi' A,\chi' L}(\delta)
\]
for some $\chi'\in(0,1)$ which, by the above-described general theory, implies loss of memory for the systems 
with small holes. 

The situation with large holes is more complex.
Even in the case of a single map, satisfactory results are missing for large holes (see \cite{LM-D} for partial results). However, something can be done for sparse holes. That is, holes that appear rarely in the sequence, and the intervening maps have sufficient mixing properties. Essentially, we can ask that, calling $T_n$ a sequence of maps between the appearances of holes and $\cL_n$ the associated sequence of operators, for each $W\in\cW^s(\delta)$
 \[
 |W|^{-1}\int_{W} \cL_n\Id_{H^c}{\; \geq\;} \frac 12(1 - \musrb(H)) \, .
 \]
 This allows us to obtain again, see \cite[Section 8.2]{dl cones},
 \[
\cL_n(\Id_{H^c}[\cC_{c,A,L}(\delta)])\subset  \cC_{\chi' c,\chi' A,\chi' L}(\delta).
\]
\subsubsection{Chaotic scattereres}\ \\
The above results can be applied to the study of chaotic scattering when eclipses are possible. Chaotic scattering refers to a situation in which there is a group of obstacles in a compact region of the plane, and one is interested in the situation in which a beam of particles comes from infinity, collides with the scatterers, and then continues toward infinity. 
The term {\em eclipse} refers to the fact that the convex hull of two scatterers has a non-empty intersection with a third scatterer. If this does not happen, we say that the system satisfies the {\em no-eclipse condition}. In this situation, it is well known that the system has a natural Markov partition that can be used to code the system and obtain the desired results; in some sense, the system has no discontinuities, see \cite{Mo91,Mo07}. On the contrary, when an eclipse is present, there is no simple way to code the system, and an approach using transfer operators becomes advantageous.

The goal is to understand the velocity distribution of the scattered particles. Namely, let $(r_n,\theta_n)$ be the coordiantes giving the position in the Poincarè section and the direction of the velocity at the $n$-th collision. Also, let $\cL$ be the transfer operator associated with the Poincarè map. We would like to compute
\begin{equation}
\label{eq:P def}
\bP_f(\theta_n\in[\theta_1, \theta_2]) = \int_M \Id_{H_\Theta} [\cL\Id_{H^c}]^n g\, d\musrb \, ,
\end{equation}
where $g$ is the distribution of the initial beam of particles, $H$ is the set of trajectories that will not experience another collision, and $H_\Theta\subset H$ is the set of trajectories that have velocity in the direction $\Theta:=[\theta_1, \theta_2]$.

Unfortunately, this system has a large hole. As we discussed, understanding such systems is an open problem. The situation can be treated if the obstacles are {\em boxed}; that is, they are contained in, e.g., a cubical box that lets them in, but then traps them (i.e., the particles collide elastically against the boundaries of the box) for a time $N$ large enough. Then, as discussed above, the dynamics is described by the operator $\mathring{\cL}:=\cL^N\Id_{H^c}$, where now $\cL$ is the transfer operator associated to the Poincarè map of the particle trapped in the box, which strictly contracts the cone. Hence, we have the spectral decomposition $\mathring\cL^n g=\nu^n \mu_*\ell(g)+\cO(\nu^n\vartheta^n)$ for some $\nu,\vartheta\in (0,1)$ and measures $\mu_*,\ell$,  \cite[Section 8.4]{dl cones}. It follows that the Probability $\bP(\Theta)$ that the particle exists from the box with a velocity in the direction $\theta\in \Theta$ is given by
\[
\bP(\Theta)=\sum_{n=1}^\infty\int_M \Id_{H_\Theta} \mathring{\cL}^n g\, d\musrb=\int_M \Id_{H_\Theta} (\Id-\mathring{\cL})^{-1} \mathring{\cL}g\, d\musrb
\sim \nu(1-\nu)^{-1} \mu_*(\Theta)\ell(g).
\]
\subsubsection{Random Lorentz gas}\ \\
The Lorentz gas is an important physical model of diffusion. It comes in many different versions depending on the type of obstacles one considers. If the obstacles are convex bodies, then it is an example of a dispersing billiard. The first rigorous result for such systems is due to Buninovich and Sinai when the obstacles are periodic \cite{bs, bsc} and the system has {\em finite horizon} (that is, there exists a maximal distance at which the particle can travel without meeting an obstacle) where it is shown that, under a diffusive rescaling, the motion of the particle behaves like a Brownian motion. Many other, more refined, results have then been obtained both for the finite-horizon case, e.g., \cite{DSV08, PS10, P19}, and for the infinite-horizon case, e.g., \cite{B92, SV07, PT21}. However, real systems are not periodic as they have defects. It is thus of interest to study non-periodic Lorentz gases. Surprisingly, the situation for the non-periodic case is much harder; diffusion has been established only for local perturbations of a periodic pattern \cite{DSV09}. If the periodic structure is missing across all space, no results are available. The only exception is the low-density Grad-Boltzmann limit, where, contrary to the current situation, the case of obstacles Poisson distributed \cite{G69, S78, BBS83} turns out to be easier to treat than the periodic case \cite{MS11}. 

A possible way to consider an intermediate situation is when part of the periodic structure is retained, for example, there are periodic cells containing a central obstacle in a random position, see Figure \ref{Fig6}-a.

\begin{figure}[ht]
\begin{minipage}{.3 \linewidth} 
\hspace{-1.5cm}	
\begin{tikzpicture}[scale=0.35]
\fill[gray!20!white] (0,0) circle (1.5);
\draw (0,0) circle (1.5);
\fill[gray!20!white] (4,0) circle (1.5);%
\draw (4,0) circle (1.5);
\fill[gray!20!white] (8,0) circle (1.5);%
\draw (8,0) circle (1.5);
\fill[gray!20!white] (0,-4) circle (1.5);
\draw (0,-4) circle (1.5);
\fill[gray!20!white] (4,-4) circle (1.5);%
\draw (4,-4) circle (1.5);
\fill[gray!20!white] (8,-4) circle (1.5);%
\draw (8,-4) circle (1.5);
\fill[gray!20!white] (0,-8) circle (1.5);
\draw (0,-8) circle (1.5);
\fill[gray!20!white] (4,-8) circle (1.5);%
\draw (4,-8) circle (1.5);
\fill[gray!20!white] (8,-8) circle (1.5);%
\draw (8,-8) circle (1.5);
\fill[gray!20!white] (2.2,-2) circle (1);
\draw (2.2,-2) circle (1);
\node at (2.2,-2) {{$\scriptscriptstyle B_\omega(a)$}};
\fill[gray!20!white] (6,-1.8) circle (1);
\draw (6,-1.8) circle (1);
\node at (6,-1.8) {{$\scriptscriptstyle B_\omega(b)$}};
\fill[gray!20!white] (2,-6.3) circle (1);
\draw (2,-6.3) circle (1);
\node at (2,-6.3) {{$\scriptscriptstyle B_\omega(0)$}};
\fill[gray!20!white] (6,-5.7) circle (1);
\draw (6,-5.7) circle (1);
\node at (6.1,-5.7) {{$ \scriptscriptstyle B_\omega(c)$}};
\node at  (5,-11) {$a=(1,0); b=(1,1); c=(1,0)$};
\end{tikzpicture}
\end{minipage}
\begin{minipage}{.3 \linewidth}
\begin{tikzpicture}[scale=0.50]
\draw[dashed] (0,0)--(8,0);
\draw [dashed](0,0)--(0,-8);
\draw[dashed](0,-8)--(8,-8);
\draw[dashed](8,-8)--(8,0);
\draw[very thick, dashed] (3,0)--(5,0);
\draw[very thick, dashed] (3,-8)--(5,-8);
\draw[very thick, dashed] (0,-3)--(0,-5);
\draw[very thick, dashed] (8,-5)--(8,-3);
\fill[gray!20!white] (0,0)--(0,-3.5) arc (-90:0:3.5)--cycle;
\draw[very thick] (0,-3.5) arc (-90:0:3.5);
\fill[gray!20!white] (0,-8)--(3.5,-8) arc (0:90:3.5)--cycle;
\draw[very thick] (3.5,-8) arc (0:90:3.5);
\fill[gray!20!white] (8,-8)--(4.5,-8) arc (180:90:3.5)--cycle;
\draw[very thick] (4.5,-8) arc (180:90:3.5);
\fill[gray!20!white] (8,0)--(4.5,0) arc (180:270:3.5)--cycle;
\draw[very thick] (4.5,0) arc (180:270:3.5);
\fill[gray!20!white] (4.3,-4.4) circle (1.5);
\draw[very thick] (4.3,-4.4) circle (1.5);
\node at (1.7,-1.2) {$C_2$};
\node at (6.3,-1.2) {$C_1$};
\node at (1.7,-6.2) {$C_3$};
\node at (6.4,-6.2) {$C_4$};
\node at (4,-4) {$C_5$};
\node at (8.7,-4) {$\hat R_1$};
\node at (-.8,-4) {$\hat R_3$};
\node at (4,-.7) {$\hat R_2$};
\node at (4,-7.3) {$\hat R_4$};
\node at (.3,-1.5) {$r$};
\node at (5,-4.7) {$\rho$};
\draw (4.3,-4.4)--(5.8,-4.4); 
\end{tikzpicture}
\end{minipage}
\\ \vskip.1cm
{\footnotesize a: {\it  Configuration of random obstacles $B_\omega(z)$}\hskip 1.2cm b: {\it Poincar\'e section $C_i$ and gates $\hat R_i$\hspace{.5cm}}}
\caption{\label{Fig6}}
\end{figure}
In this case, we can call the lines $\hat R_i$, at the boundaries of two adjacent cells {\em gates}, see Figure \ref{Fig6}-b. It follows that, given a path $\{z_1,\dots,z_n\}$, $z_i \in \mathbb{Z}^2$, that specifies which cells are visited and an initial density distribution $g$ in the starting cell $z_1$, the density distribution in the cell $z_n$ at the end of the path is again given by a product of transfer operators, and again we have a system with large holes and the difficulty already explained. We can therefore adopt a similar solution to the previous cases: if a particle enters a cell, then the gate closes for a time $N$ and then opens to allow the particle to escape at the $kN$ collision, $k\in \bN$. This setting is called {\em lazy gates} in \cite[Section 8.5]{dl cones} where it is proven that the system exhibits exponential loss of memory. 

This is still far from proving diffusion for a particle moving in such an environment. However, if one considers an observable that does not depend on the cell, e.g. the position of the particle inside a cell, then \cite{dl clt} yields a CLT conditioned on the path taken in the lattice.

\section{Equivalence of Cones and Norms} \label{sec:equivalence}
In this section, we prove a result that has not previously been published.
It is natural to ask if the norms defined in Section~\ref{sec:Bt>0}, using the set 
$\widehat \cW^s$ rather than $\cW^s$, are equivalent to the norm $\| \cdot \|_*$, defined by the ordering of the cone
in \eqref{eq:cone_norm}.
In this section, we provide a positive answer to this question, which is summarized in the following theorem.

For convenience, we let $C_s := \sqrt{1 + (\cK_*^{-1} + \tau_*^{-1})^2}$ denote the maximum
slope of curves in the family $\widehat \cW^s$.

\begin{thm}
\label{thm:cone equiv}
Let $\cB$ denote the Banach space defined in Section~\ref{sec:Bt>0} with $\widehat \cW^s$ in place of $\cW^s$, and let
$\cC_{c, A, L}(\delta)$ be as defined in Section~\ref{sec:cone def}.\\
Assume $q=\frac 1p$, $A>4$, $Ac> (1+C_s)$, and $\gamma\leq \frac 1{p+1}$.\footnote{ By $p$ we mean the constant in \eqref{eq:stable t>0}.}\\
Then the norms $\|\cdot\|_\cB$ and $\|\cdot\|_*$ are equivalent on $\cB$.
\end{thm}
We postpone the proof of the above Theorem to the end of the section.
\begin{rem}
One could equally well define both the cone and the norms on the set of real stable manifolds $\cW^s$ of a single map
and obtain an analogous result.  However, we will work with $\widehat \cW^s$ since this choice allows for the application
of both methods to sequential systems (although restricted to the SRB case $t=1$).
\end{rem}

\begin{rem}
This section points to the fact that the construction of a Banach space in which one can prove the Lasota-Yorke inequalities and the construction of a strictly invariant cone are largely equivalent. Yet, the two methods have their peculiarities: the Banach space approach yields information on the essential spectrum and may also work if the system is not mixing; on the contrary, the cone method yields only a spectral gap and can work only for mixing systems, but has much more flexibility if one wants to consider sequential systems.
\end{rem}

The proof of the theorem rests on the following two lemmas, which provide the required bounds.

\begin{lem}
\label{lem:cone bound}
Assuming the conditions of Theorem~\ref{thm:cone equiv},
there exists $K>0$ such that, for each  each $g\in C^1(M)$ we have,
\[
\| g\|_*\leq K\delta^{\gamma-1}\|g\|_\cB,
\]
where we have choosen $\be = 1$ in definition \eqref{eq:cone_norm}. 
\end{lem}
\begin{proof}
For $g\in C^1$, we must compute for which $\lambda>0$ we have $\lambda \pm g, \in \cC_{c,A, L}(\delta)$.
Note that if $W \in \cW^s(\delta)$ and $\psi\in \cD_{a,\beta}(W)$, then
\[
\frac{\psi(x)-\psi(y)}{d(x,y)^\beta}\leq \psi(x)\frac{1-e^{-ad(x,y)^\beta}}{d(x,y)^\beta}\leq \|\psi\|_\infty  a.
\]
Thus $|\psi|_{C^\beta}\leq (1+a) |\psi|_\infty$. The analogous estimate holds for $\psi \in \cD_{a,\alpha}(W)$.
Also, $|\psi|_\infty/ \fint_W \psi \le e^{a2^\beta \delta^\beta} \le 2$ provided $\delta$ is small enough for
fixed $a$.
Hence, recalling equation \eqref{eq:stable t>0},
\[
\begin{split}
&\frac{\left|\int_W( \lambda\pm g )\psi \, dm_W\right|}{\int_W \psi \, dm_W}\leq \lambda+\frac{\left|\int_W g \psi \, dm_W\right|}{\int_W \psi \, dm_W}\leq 
\lambda+(2+2a) |W|^{\frac{1}{p} -1}\|g\|_s \\
&\frac{\left|\int_W( \lambda\pm g )\psi \, dm_W\right|}{\int_W \psi \, dm_W}\geq \lambda-\frac{\left|\int_W g \psi \, dm_W\right|}{\int_W \psi \, dm_W}\geq 
\lambda-(2+2a) |W|^{\frac{1}{p} -1}\|g\|_s . 
\end{split}
\]
Thus the first cone condition  \eqref{eq:cone 1} is satisfied provided
\[
\lambda  \geq \frac{L+1}{L-1} (2+2a) \delta^{\frac{1}{p} -1}\|g\|_s  .
\]
If $L>3$, we further increase $\lambda$ so that
\begin{equation}
\label{eq:lower tri}
\tri\lambda\pm g\tri_-\geq \lambda - (2a+2) \delta^{\frac{1}{p} -1}\|g\|_s       \geq \lambda/2.
\end{equation}
Next, for $\psi \in \cD_{a,\alpha}(W)$, it is convenient to define $\psi_W:=\psi\left[\fint_W\psi\right]^{-1}$, and note that $|\psi_W|_{\cC^\alpha}\leq 2+2a$, as above.
For $p=1/q$, we compute, recalling again \eqref{eq:stable t>0}, for $W \in \cW^s_-(\delta)$, 
\[
 |W|^{-q}\frac{|\int_W (\lambda\pm g) \psi|}{\fint_W\psi} \leq |W|^{1-q}\lambda+  |W|^{-q}\left|\int_W g \psi_W\right|
 \leq  (2\delta)^{1-q}\lambda+ (2+2a)\|g\|_s .
 \]
 It follows that the second cone condition \eqref{eq:cone 2} is satisfied if
 \[
 (2 \delta)^{1-q}\lambda+ (2+2a)\|g\|_s\leq A\delta^{1-q}\lambda/2.
 \]
 That is
 \[
 \lambda\geq \frac{4(1+a) \delta^{-1+q}}{A-2^{2-q}}\|g\|_s,
 \]
providing $A > 4$.

 To conclude, we must estimate for $W^1, W^2 \in \cW^s_-(\delta)$, $\psi_i \in \cD_{a,\alpha}(W^i)$,
 with $d_{\cW^s}(W^1, W^2) \le \delta$ and $d_*(\psi_1, \psi_2) =0$,
 \begin{equation}
 \label{eq:lambda split}
 \begin{split}
 \left|\frac{\int_{W^1} (\lambda\pm g) \psi_1}{\fint_{W^1}\psi_1}  - \frac{\int_{W^2} (\lambda\pm g)\psi_2}{\fint_{W^2}\psi_2} \right|\leq&
 \lambda ||W^1| - |W^2|| \\
 &+ \left|\int_{W^1}  g \psi_{1,W}  - \int_{W^2}  g\psi_{2,W} \right| \,.
\end{split}
\end{equation}
Without loss of generality, we may assume $\fint_{W^1} \psi_1 = 1$ and $|W^1| \le |W^2|$, the other cases being similar.
A direct computation shows that (see \cite[eq.~(5.8)]{dl cones})
\begin{equation}
\label{eq:useful}
\left|\,|W^1|-|W^2|\,\right|\leq  (1+C_s)  d_{\cW^s}(W^1,W^2).
\end{equation}
Thus the estimate is trivial if $|W^2|^{1/p} \le d_{\cW^s}(W^1, W^2)^\gamma$ for then \eqref{eq:lambda split} can be estimated by
\[
 \left|\frac{\int_{W^1} (\lambda\pm g) \psi_1}{\fint_{W^1}\psi_1}  - \frac{\int_{W^2} (\lambda\pm g)\psi_2}{\fint_{W^2}\psi_2} \right|\leq
 \lambda (1+C_s) d_{\cW^s}(W^1, W^2) + 2|W^2|^{1/p} \| g \|_s (2a+2) \,.
\]
Recalling \eqref{eq:lower tri}, the third cone condition \eqref{eq:cone 3} is satisfied if
\[
\lambda (1+C_s) d_{\cW^s}(W^1, W^2)^\gamma \delta^{1-\gamma} + (4a+4) d_{\cW^s}(W^1, W^2)^\gamma \| g \|_s
\le d_{\cW^s}(W^1, W^2)^\gamma \delta^{1-\gamma} c A \lambda/2 \, ,
\]
which implies,
\[
\lambda \ge \frac{8a+8}{cA - 2(1+C_s)} \delta^{\gamma -1} \| g \|_s \, .
\]
It remains to estimate \eqref{eq:lambda split} in the case
\begin{equation}
\label{eq:long W2}
|W^2|^{1/p} \ge d_{\cW^s}(W^1, W^2)^\gamma \, .
\end{equation}
Now we must correct for the discrepancy between the use of $d_0$ and $d_*$ in measuring the
difference between test functions on $I_{W_1} \cap I_{W_2}$. Let $\overline W^2 = G_{W^2}(I_{W^1} \cap I_{W^2})$ and let
$W^2_i$ denote the at most two connected components of $W^2 \setminus \overline W^2$.  
By assumption, $d_*(\psi_1, \psi_2)=0$.  Thus,
\[
\begin{split}
&d_0 \left( \psi_{1,W}, \, \psi_2^* \right) = 0, \mbox{ where } \\
&\psi_2^*:= \psi_{2, W}  \frac{\| G_{W^2}' \| \circ G_{W^2}^{-1}}{\| G_{W^1}' \| \circ G_{W^2}^{-1}} \frac{\fint_{W^2} \psi_2}{\fint_{W^1}  \psi_1}.
\end{split}
\]
Then, recalling $\fint_{W^1} \psi_1 = 1$,
\begin{equation}
\label{eq:test split}
\begin{split}
&\hskip-18pt\left|\int_{W^1} g\psi_{1,W}  - \int_{W^2}  g\psi_{2,W} \right|
 \le \left| \int_{W_1} g \psi_{1,W}  - \int_{\overline W^2} g \psi_2^*  \right| \\
&+ \left| \int_{\overline W^2} g \psi_{2, W} \left( \frac{\| G_{W^2}' \| \circ G_{W^2}^{-1} }{\| G_{W^1}' \|  \circ G_{W^2}^{-1} } \fint_{W^2} \psi_2 - 1 \right) \right| 
+ \left| \sum_{i=1}^2 \int_{W^2_i} g \psi_{2,W} \right|
\\
 \le& \| g \|_u d_{\cW^s}(W^1, W^2)^\gamma \max\left\{  |\psi_{1,W}|_{C^\alpha(W^1)} , \,  |\psi^*_{2}|_{C^\alpha(\overline W^2)} \right\} \\
& + 
  \| g \|_s |W^2|^{1/p} |\psi_{2,W}|_{C^\beta(W^2)} \left|1 - \frac{\| G_{W^2}' \|  \circ G_{W^2}^{-1} }{\| G_{W^1}' \| \circ G_{W^2}^{-1} } \fint_{W^2} \psi_2 \right|_{C^\beta(\overline W^2)} \\
  & + \sum_{i=1}^2 \| g \|_s |W^2_i|^{1/p} |\psi_{2,W}|_{C^\beta(W^2_i)} .
\end{split}
\end{equation}
As before, $|\psi_{i,W}|_{C^\alpha(W^i)} \le 2+2a$, and similarly for $|\psi_{2,W}|_{C^\beta(W^2)}$.
Also, $|W^2_i| \le C_s d_{\cW^s}(W^1, W^2)$, by definition of the distance $d_{\cW^s}$.
Next,
\[
1 - \frac{\| G_{W^2}' \|  \circ G_{W^2}^{-1} }{\| G_{W^1}' \| \circ G_{W^2}^{-1} } \fint_{W^2} \psi_2 
 =  \left( 1 - \frac{\| G_{W^1}' \|  \circ G_{W^1}^{-1} }{\| G_{W^2}' \| \circ G_{W^1}^{-1} } \right) \fint_{W^2} \psi_2
 + \left( 1 - \fint_{W^2} \psi_2 \right) \, .
\]
Using \cite[eq. (5.10)]{dl cones} together with \eqref{eq:long W2}, we have
\[
\begin{split}
\left| 1 - \fint_{W^2} \psi_2 \right| &= |W^2|^{-1} \left| |W^2| - \int_{W^2} \psi_2 \right|
\le |W^2|^{-1} 6C_s d_{\cW^s}(W^1, W^2) \\
&\le 6C_s d_{\cW^s}(W^1,W^2)^{1-p\gamma}
\end{split}
\]
and $1-p\gamma \ge \gamma$ since we have assumed $\gamma \le 1/(p+1)$.  In particular,
$\fint_{W^2}\psi_2 \le 2$ for $\delta$ small enough.\\
To estimate $\left| 1 - \frac{\| G_{W^2}' \|  \circ G_{W^2}^{-1} }{\| G_{W^1}' \| \circ G_{W^2}^{-1} } \right|_{C^\alpha(\overline W^2)}$,
we use the fact that $\| G_{W^1}' \| \ge 1$ to write,
\[
\left| 1 - \frac{\| G_{W^2}' \|  \circ G_{W^2}^{-1} }{\| G_{W^1}' \| \circ G_{W^2}^{-1} } \right|
= \frac{1}{\sqrt{1 + (\vf_{W^1}')^2}} \frac{| \vf_{W^1}' - \vf_{W^2}'| |\vf_{W^1}' + \vf_{W^2}' |}{\sqrt{1 + (\vf_{W^1}')^2}+
\sqrt{1 + (\vf_{W^2}')^2}} \le d_{\cW^s}(W^1, W^2) .
\]
Then, for $x, y \in \overline W^2$, we use the above to estimate on the one hand,
\[
\left| \frac{\| G_{W^2}' \|  \circ G_{W^2}^{-1}(x) }{\| G_{W^1}' \| \circ G_{W^2}^{-1}(x) } - \frac{\| G_{W^2}' \|  \circ G_{W^2}^{-1}(y) }{\| G_{W^1}' \| \circ G_{W^2}^{-1}(y) } \right| \le 2 d_{\cW^s}(W^1, W^2),
\]
while on the other hand, setting $r = G_{W^2}^{-1}(x)$ and $s = G_{W^2}^{-1}(y)$,
\[
\begin{split}
&\left| \frac{\| G_{W^2}' \|  \circ G_{W^2}^{-1}(x) }{\| G_{W^1}' \| \circ G_{W^2}^{-1}(x) } - \frac{\| G_{W^2}' \|  \circ G_{W^2}^{-1}(y) }{\| G_{W^1}' \| \circ G_{W^2}^{-1}(y) } \right|
 \le \left| \frac{\| G_{W^2}'(r) \|   }{\| G_{W^1}' (r) \| } - \frac{\| G_{W^2}' (s)\|  }{\| G_{W^1}'(r) \|  } \right|\\
&\phantom{\left| \frac{\| G_{W^2}' \|  \circ G_{W^2}^{-1}(x) }{\| G_{W^1}' \| \circ G_{W^2}^{-1}(x) } - \frac{\| G_{W^2}' \|  \circ G_{W^2}^{-1}(y) }{\| G_{W^1}' \| \circ G_{W^2}^{-1}(y) } \right|\leq }
+ \left| \frac{\| G_{W^2}'(s) \| }{\| G_{W^1}' (r) \| } - \frac{\| G_{W^2}' (s)\| }{\| G_{W^1}' (s)\|  } \right| \\
 \le & |\vf_{W^2}'(r) - \vf_{W^1}'(s)| + C_s |\vf_{W^1}'(r) - \vf_{W^2}'(s)| \le (1+C_s) B |r-s|,
\end{split}
\]
where $B>0$ denotes the maximum of $|\vf_{W^i}''$, which depends only on the family $\widehat \cW^s$. 
Since $|r-s| \le d(x,y)$, we see that the difference is bounded by the minimum of these two estimates, which is maximized when the two are equal, i.e. when  
$2d_{\cW^s}(W^1, W^2)= (1+C_s)B d(x,y)$.  Thus the $\alpha$-Holder constant is bounded by
$B^\alpha (1+C_s)^\alpha 2^{1-\alpha} d_{\cW^s}(W^1, W^2)^{1-\alpha}$.
We conclude
\begin{equation}
\label{eq:G close}
\left| 1 - \frac{\| G_{W^2}' \|  \circ G_{W^2}^{-1} }{\| G_{W^1}' \| \circ G_{W^2}^{-1} } \right|_{C^\alpha(\overline W^2)}
\le 2 B^\alpha (1+C_s)^\alpha d_{\cW^s}(W^1, W^2)^{1-\alpha} \, .
\end{equation}
This implies in particular that there exists $D>1$ such that
\begin{equation}
\label{eq:G ratio}
D^{-1} \le \left| \frac{\| G_{W^2}' \|  \circ G_{W^2}^{-1} }{\| G_{W^1}' \| \circ G_{W^2}^{-1} } \right|_{C^\alpha(\overline W^2)}
\le D \, .
\end{equation}
Collecting these estimates in \eqref{eq:test split} and recalling \eqref{eq:useful}, we estimate \eqref{eq:lambda split}
by
\begin{equation}\label{eq:u-norm-compare}
\begin{split}
&\left|\frac{\int_{W^1} (\lambda\pm g) \psi_1}{\fint_{W^1}\psi_1}  - \frac{\int_{W^2} (\lambda\pm g)\psi_2}{\fint_{W^2}\psi_2} \right|
\leq \lambda (1+C_s) d_{\cW^s}(W^1, W^2) \\
&+ \| g \|_u d_{\cW^s}(W^1, W^2)^\gamma 2D (2a+2) \\
&+ \|g \|_s (2\delta)^{1/p} (2a+2) \big[4B^\alpha (1+C_s)^\alpha d_{\cW^s}(W^1, W^2)^{1-\alpha} \\
&  + 6C_sd_{\cW^s}(W^1, W^2)^\gamma \big] + 2 (2a+2) \|g\|_s C_s^{1/p}d_{\cW^s}(W^1, W^2)^{1/p} \,.
\end{split}
\end{equation}
Since $\gamma \le 1-\alpha$ and $\gamma \le 1/p$ by assumption, and using $d_{\cW^s}(W^1, W^2) \le \delta$
and \eqref{eq:lower tri},
the third cone condition \eqref{eq:cone 3}  is satisfied if
\[
\begin{split}
\lambda (1+C_s) & d_{\cW^s}(W^1, W^2)^\gamma \delta^{1-\gamma}  + \| g\|_u d_{\cW^s}(W^1, W^2)^\gamma D(4a+4) \\
& + \|g \|_s d_{\cW^s}(W^1, W^2)^\gamma (4a+4) [2B^\alpha (1+C_s)^\alpha + 3C_s + C_s^{1/p} ] \\
& \le d_{\cW^s}(W^1, W^2)^\gamma \delta^{1-\gamma} c A \lambda / 2 \, ,
\end{split}
\] 
which is implied by
\[
\lambda \ge \frac{2D(4a+4) \| g \|_u + \| g \|_s (8a+8) [2B^\alpha (1+C_s)^\alpha + 3C_s + C_s^{1/p} ]}{\delta^{1-\gamma} (cA - 2(1+C_s)) } ,
\]
from which the lemma follows.
\end{proof}

To prove the bound in the other direction, we introduce the concept of weakly order-preserving.
We say that a norm 
$\|\cdot\|$ on $\cB$ is {\em weakly order preserving} if there exists a $\upsilon>0$ such that for all $f,g\in C^1$ 
\[
-f\preceq g\preceq f \Longrightarrow \|f\|\geq\upsilon \|g\|.
\]
\begin{lem}\label{lem:weak_order} 
Under the assumptions of Theorem~\ref{thm:cone equiv}, the norm $\|\cdot\|_\cB$ is weakly order preserving.
\end{lem}
\begin{proof}
Let $f,g\in C^1$ such that $f\pm g\succeq 0$. Then for each $W\in\cW^s(\delta)$ and $\psi\in \cD_{a,\beta}(W)$ we have by \eqref{eq:cone 1}, $\int_W(f\pm g)\psi\geq 0$. Hence
\[
\int_W f \psi\geq \left|\int_W g \psi\right|.
\]
As in the proof of Lemma~\ref{lem:cone bound}, we define $\psi_W = \psi \big[ \fint_W \psi \big]^{-1}$ and recall
that $| \psi_W |_{C^\beta(W)} \le 2+2a$.  It then follows immediately that
\[
\tri f \pm g \tri_- \le \tri g \tri_+ + \tri f \tri_+ \le 2 \tri f \tri_+ \le 2 \delta^{\frac{1}{p}-1} (2+2a) \| f \|_s =: C_1 \|f \|_s \, .
\]

Next, for $W\in \cW^s_-(\delta)$, $\psi\in \cD_{a,\beta}(W)$, we have by \eqref{eq:cone 2} 
\[
\begin{split}
& |W|^{-q}\frac{|\int_W (f\pm g) \psi|}{\fint_W\psi}  \le  A \delta^{1-q} \tri f\pm g \tri_-\leq A \delta^{1-q} C_1\| f \|_s
= A(4a+4) \| f \|_s  \\
& |W|^{-q}\frac{|\int_W (f\pm g) \psi|}{\fint_W\psi} \geq |W|^{-q}\frac{|\int_W  g \psi|}{\fint_W\psi}-|W|^{-q}\frac{|\int_W  f \psi|}{\fint_W\psi} .
\end{split}
\]
Accordingly,
\begin{equation}
\label{eq:good bound}
|W|^{-q}\frac{|\int_W  g \psi|}{\fint_W\psi}\leq (A (4a+4) + 2a+2)  \| f \|_s.
\end{equation}
To estimate $\| g \|_s$, we will need to translate between functions in $C^\beta(W)$ and $\cD_{a,\beta}(W)$.  To this end, for 
$\psi\in C^\beta(W,\bR)$, $|\psi|_{C^\beta} \le 1$, we write $\psi=\psi_+-\psi_-$, where $\psi_\pm\geq 0$. 
Clearly $\|\psi_\pm\|_{C^\beta}\leq 1$. Moreover,
\begin{equation}\label{eq:psi_to_cone}
\frac{\psi_\pm(x)+a^{-1}}{\psi_\pm(y)+a^{-1}}\leq 1+a|\psi_\pm(x)-\psi_\pm(y)|\leq 1+a|x-y|^\beta\leq e^{a|x-y|^\beta}.
\end{equation}
Thus $\psi_\pm(x)+a^{-1}\in  \cD_{a,\beta}(W)$.  
Applying \eqref{eq:good bound}, we obtain 
\[
\begin{split}
|W|^{-1/p} \left|\int_W g\psi\right|&\leq |W|^{-1/p} \left|\int_W g(\psi_++a^{-1})\right| + |W|^{-1/p} \left|\int_W g(\psi_-+a^{-1})\right|\\
&\le (A(4a+4)+2a+2) \| f\|_s \left( \fint_W (\psi_+ + a^{-1}) + \fint_W (\psi_- + a^{-1}) \right) \, .
\end{split}
\]
Taking the appropriate suprema yields
\begin{equation}
\label{eq:stable bound}
\|g\|_s\leq (A(4a+4) + 2a+2) \|f \|_s 2(1+a^{-1}) =: C_2\|f\|_s.
\end{equation}

To bound the strong unstable norm of $g$, we must estimate the difference in \eqref{eq:unstable t>0} for
$W^1, W^2 \in \cW^s_-(\delta)$ with $\psi_i \in C^\alpha(W^i)$, $|\psi_i|_{C^\alpha(W^i)} \le 1$ and 
$d_0(\psi_1, \psi_2)=0$.  Without loss of generality, we may assume $|W^2| \ge |W^1|$.  As the in the
proof of Lemma~\ref{lem:cone bound}, the estimate is trivial if $|W^2|^{1/p} \le d_{\cW^s}(W^1, W^2)^\gamma$
since then using \eqref{eq:stable bound},
\[
\left|\frac{\int_{W^1} g \psi_1}{\fint_{W^1}\psi_1}  - \frac{\int_{W^2} g \psi_2}{\fint_{W^2}\psi_2} \right|
\le 2|W^2|^{1/p} \| g \|_s (2a+2) \le d_{\cW^s}(W^1, W^2)^\gamma (4a+4) C_2 \| f\|_s \, .
\]
It remains to estimate \eqref{eq:unstable t>0} under the assumption that
\begin{equation}
\label{eq:long again}
|W^2|^{1/p} \ge d_{\cW^s}(W^1, W^2)^\gamma \, .
\end{equation}
Now, for all  $W^1, W^2 \in \cW^s_-(\delta)$ such that $ d_{\cW^s}(W^1, W^2) \le \delta$ and $\psi_i \in \cD_{a, \alpha}(W^i)$ such that $d_*(\psi_1, \psi_2)=0$, we have
\[
\begin{split}
&\left|\frac{\int_{W^1} g \psi_1}{\fint_{W^1}\psi_1}  - \frac{\int_{W^2} g \psi_2}{\fint_{W^2}\psi_2} \right|-\left|\frac{\int_{W^1} f \psi_1}{\fint_{W^1}\psi_1}  - \frac{\int_{W^2} f \psi_2}{\fint_{W^2}\psi_2} \right|
\leq \left|\frac{\int_{W^1} (f\pm g) \psi_1}{\fint_{W^1}\psi_1}  - \frac{\int_{W^2} (f\pm g)  \psi_2}{\fint_{W^2}\psi_2} \right|\\
& \hspace{0.5 in} \leq d_{\cW^s}(W^1, W^2)^\gamma \, \delta^{1-\gamma}   c A \tri (f\pm g) \tri_- \leq 
d_{\cW^s}(W^1, W^2)^\gamma \, \delta^{1-\gamma}   c A C_1\|f\|_s.
\end{split}
\]
Using \eqref{eq:u-norm-compare}, with $\lambda=0$ (applied to $f$ instead of $g$), we obtain that there exists $C_3>0$ such that
\begin{equation}
\label{eq:star bound}
\begin{split}
\left|\frac{\int_{W^1} g \psi_1}{\fint_{W^1}\psi_1}  - \frac{\int_{W^2} g \psi_2}{\fint_{W^2}\psi_2} \right|
\leq &
d_{\cW^s}(W^1, W^2)^\gamma \, \delta^{1-\gamma}   c A C_1\|f\|_s  \\
&+ \| f \|_u d_{\cW^s}(W^1, W^2)^\gamma 4(2a+2) \\
&+ \|f \|_s (2\delta)^{1/p} (2a+2) \big[4B^\alpha (1+C_s)^\alpha d_{\cW^s}(W^1, W^2)^{1-\alpha} \\
&  + 6C_sd_{\cW^s}(W^1, W^2)^\gamma \big] + 2 (2a+2) \|f\|_s C_s^{1/p}d_{\cW^s}(W^1, W^2)^{1/p}\\
&\leq C_3(\| f \|_u +\|f\|_s)d_{\cW^s}(W^1, W^2)^\gamma .
\end{split}
\end{equation}
To conclude, as in the proof of Lemma~\ref{lem:cone bound}, we must reconcile the difference
between $d_0(\psi_1, \psi_2)=0$ and $d_*(\psi_1, \psi_2)=0$.

As an intermediate step, suppose $\psi_i \in \cD_{a,\alpha}(W^i)$ with $d_0(\psi_1,\psi_2)=0$.
  Then, defining $\overline W^1 = G_{W^1}(I_{W^1} \cap I_{W^2})$ and 
\[
\psi_1^* = \psi_1 \frac{\| G_{W^2}' \| \circ G_{W^1}^{-1}}{\| G_{W^1}' \| \circ G_{W^1}^{-1} } \, ,
\]
we have $\psi_1^* \in C^\alpha(\overline W^1)$ and $d_*(\psi^*_1, \psi_2)=0$.  
Remark that by \eqref{eq:G ratio}, we have $\left| \psi_1^* \big[ \fint_{\overline W^1} \psi_1^* \big]^{-1} \right|_{C^\alpha(\overline W^1)} \le D (2a+2)$.
Thus, letting $W^1_i$ denote the
at most two components of $W^1 \setminus \overline W^1$, 
\begin{equation}
\label{eq:u split}
\begin{split}
\left| \int_{W^1} g \psi_1 - \int_{W^2} g \psi_2 \right|  \le& |\psi_2|_\infty  \left| \frac{ \int_{\overline W^1} g \psi_1^*}{\fint_{\overline W^1} \psi_1^*} - \frac{ \int_{W^2} g \psi_2}{\fint_{W^2} \psi_2 } \right| \\
&+ \left| \frac{ \int_{\overline W^1} g \psi_1^* }{\fint_{\overline W^1} \psi_1^* }  \left( \fint_{\overline W^1} \psi_1^* - \fint_{W^2} \psi_2  \right) \right|
  + \left| \int_{\overline W^1} g (\psi_1 - \psi_1^*) \right|
+ \sum_{i=1}^2 \left| \int_{W^1_i} g \psi_1 \right| \\
 \le & |\psi_2|_{C^\alpha} C_3 (\| f \|_u + \| f \|_s) d_{\cW^s}(W^1, W^2)^\gamma \\
&+ |\overline W^1|^{1/p} C_2 \| f\|_s D(2a+2) 
\left| \fint_{\overline W^1} \psi_1^* - \fint_{W^2} \psi_2  \right| \\
& + |\overline W^1|^{1/p} C_2 \|f \|_s |\psi_1 - \psi_1^*|_{C^\beta(\overline W^1)} + \sum_{i=1}^2 |W^1_i|^{1/p} C_2 \| f\|_s |\psi_1|_{C^\beta(W^1)} \, ,
\end{split}
\end{equation}
where for the first term we have used \eqref{eq:star bound} and for the last three terms, we have used 
the definition of $\| g\|_s$ and \eqref{eq:stable bound}.

For the third term on the right side of \eqref{eq:u split}, \eqref{eq:G close} implies
\[
\begin{split}
|\psi_1 - \psi_1^*|_{C^\beta(\overline W^1)} & \le |\psi_1|_{C^\beta(W^1)} \left| 1 - \frac{\| G_{W^2}' \| \circ G_{W^1}^{-1}}{\| G_{W^1}' \| \circ G_{W^1}^{-1} } \right|_{C^\beta(\overline W^1)} \\
& \le |\psi_1 |_{C^\beta(W^1)}  2 B^\beta (1+C_s)^\beta d_{\cW^s}(W^1, W^2)^{1-\beta} \, .
\end{split}
\]
For the fourth term on the right side of \eqref{eq:u split}, we use $|W^1_i| \le C_s d_{\cW^s}(W^1, W^2)$.

It remains to estimate the difference between averages in the second term on the right side of \eqref{eq:u split}.
Using $\psi_1 \circ G_{W^1} = \psi_2 \circ G_{W^2}$
on $I_{W^1} \cap I_{W^2}$ (by definition of the distance $d_0(\cdot, \cdot))$ yields 
$\int_{\overline W^1} \psi_1^* = \int_{\overline W^2} \psi_2$, where $\overline W^2 = G_{W^2}(I_{W^1} \cap I_{W^2})$.  Thus,
\[
\begin{split}
\left| \fint_{\overline W^1} \psi_1^* - \fint_{W^2} \psi_2  \right|
& \le \left| \frac{1}{|\overline W^1|} - \frac{1}{|W_2|} \right| \int_{\overline W^2}  \psi_2
+ \frac{1}{|W_2|} \int_{W^2 \setminus \overline W^2} \psi_2 \\
& \le |\psi_2|_\infty \frac{|\overline W^2|}{|\overline W^1|} \frac{| |W^2| - |\overline W^1| |}{|W^2|} + 2C_s |\psi_2|_\infty \frac{d_{\cW^s}(W^1, W^2)}{|W^2|} \\
& \le |\psi_2|_\infty [C_s (1+C_s) + 2C_s ] d_{\cW^s}(W^1, W^2)^{1-p\gamma} ,
\end{split}
\]
where we have used $\frac{|\overline W^2|}{| \overline W^1|} \le C_s$,  \eqref{eq:useful} and the assumption \eqref{eq:long again}.
As before, $1-p\gamma \ge \gamma$ since we have assumed $\gamma \le 1/(p+1)$.

Gathering these estimates in \eqref{eq:u split} yields, since $1-\beta \ge \gamma$,
\begin{equation}
\label{eq:summary}
\begin{split}
d_{\cW^s}&(W^1, W^2)^{-\gamma}  \left| \int_{W^1} g \psi_1 - \int_{W^2} g \psi_2 \right| 
 \le  |\psi_2|_{C^\alpha} C_3 (\| f \|_u + \| f \|_s) \\
& + (2\delta)^{1/p} C_2 \| f\|_s D(2a+2) 
 |\psi_2|_{C^\alpha(W^2)} [3C_s + C_s^2 ] \\
& + |\overline W^1|^{1/p} C_2 \|f \|_s |\psi_1|_{C^\beta(W^1)} 2B^\beta(1+C_s)^\beta + 2 C_2 \| f\|_s |\psi_1|_{C^\beta(W^1)} \\
& =: C_4 \left[ \max \{ |\psi_1|_{C^\alpha(W^1)},  |\psi_2|_{C^\alpha(W^2)} \} \right]( \| f \|_u + \| f \|_s) \, ,
\end{split}
\end{equation}
for all $\psi_i \in \cD_{a,\alpha}(W^i)$ with $d_0(\psi_1,\psi_2)=0$.

It remains to generalize the estimate to all $\psi_i \in C^\alpha(W^i)$ with $|\psi_i|_{C^\alpha(W^i)} \le 1$
and $d_0(\psi_1, \psi_2)=0$.
Let $\psi_i\in C^\alpha(W_i,\bR)$, $|\psi_i|_{C^\alpha(W_i)}\leq 1$, such that $d_0(\psi_1,\psi_2)=0$. Letting $\psi_i^\pm$ be the positive and negative parts, respectively, we have $|\psi_i^\pm|_{C^\alpha(W_i)}\leq 1$ and  $d_0(\psi^\pm_1,\psi^\pm_2)=0$. Consequently, setting $\bar \psi_i^\pm:=\psi_i^\pm+a^{-1}\in\cD_{a,\beta}(W_i)$
using \eqref{eq:psi_to_cone}, we have
$d_0(\bar \psi_1^\pm, \bar \psi_2^\pm)=0$,  $|\bar \psi^\pm_i|_{C^\alpha(W_i)}\leq 1+a^{-1}$, and, recalling \eqref{eq:summary},
\[
\begin{split}
 \frac{\left| \int_{W^1} g \psi_1 - \int_{W^2} g \psi_2 \right|}{d_{\cW^s}(W^1, W^2)^{\gamma} }\leq &
\sum_{\sigma\in\{+,-\}} \frac{\left| \int_{W^1} g \bar\psi^\sigma_1 - \int_{W^2} g \bar\psi^\sigma_2 \right|}{d_{\cW^s}(W^1, W^2)^{\gamma} } 
\leq C_4(2+ 2 a^{-1})( \| f \|_u + \| f \|_s).
\end{split}
\]
The above implies
\[
\|g\|_u\leq C_4(2+2 a^{-1})( \| f \|_u + \| f \|_s),
\]
and completes the proof of the lemma.
\end{proof}

The two lemmas allow us to complete the proof of the theorem.

\begin{proof}[Proof of Theorem~\ref{thm:cone equiv}]
We first check the equivalence for $g\in C^1$. By the definition of 
$\|\cdot\|_*$ (and still choosing $\be = 1$) it follows that
\[
-\|g\|_*\preceq g\preceq \|g\|_*.
\]
Then Lemma \ref{lem:weak_order} implies $\upsilon \|g\|_{\cB}\leq \| \|g\|_* 1\|_{\cB}\leq \|g\|_* \| 1\|_{\cB} $. The reverse inequality follows from Lemma~\ref{lem:cone bound}.

Next, recall that we have defined $\cB$ as the completion of $C^1$ in the
$\| \cdot \|_{\cB}$ norm and that the cone is defined for functions in $\cA\supset \cC^1$ 
(c.f. \cite[Lemma~7.6]{dl cones}), hence the norm $\|\cdot\|_*$ is defined for functions in $C^1$.\footnote{Note that here, by $C^1$ we mean simply the set of continuously differentiable functions and not the Banach space with the $C^1$ norm.} 
Accordingly, the identity yields an isomorphism $\phi$ between the normed spaces $X_0:=(C^1,\|\cdot\|_\cB)$ and $X_1=(C^1,\|\cdot\|_*)$. The above discussion, in this language, is equivalent to the assertion that there exists $C_\sharp\geq 1 $ such that $\|\phi\|_{X_0\to X_1}\leq C_\sharp$ and $\|\phi^{-1}\|_{X_1\to X_0}\leq C_\sharp$. 

By density, such an isomorphism extends uniquely to an isomorphism $\Phi$ between the Banach spaces $\cB$ and $\cB_*$ (the completion of $X_1$) with the same bounds its norm. Using $\Phi$ to identify the two Banach spaces yields the equivalence of the norms.

More concretely, since the topology induced by the above norms is stronger than the topology of the space of distributions $(C^1)^*$, the above Banach spaces can be naturally identified as the elements of a subspace 
$\bV$ of the space of distributions. So $\cB=(\bV,\|\cdot\|_\cB)$ and $\cB_*=(\bV,\|\cdot\|_*)$ and the extension 
of $\Phi$ is simply the identity on $\bV$. This readily implies the equivalence of the two norms on $\bV$.
\end{proof}

\section{Decay of Correlations for the Billiard Flow}
\label{sec:flow}

The study of the decay of correlations for hyperbolic flows poses additional difficulties compared to hyperbolic maps. To appreciate such difficulties, consider that the decay of correlations for Anosov (or, more generally, Axiom A) maps was established in various works by Sinai, Ruelle, and Bowen in the late seventies, while the first results on exponential decay of correlations for flows started to appear 20 years later \cite{Do98}, for flows with $C^1$ foliations, followed almost ten years later by \cite{Li04} for contact flows, and, after almost another 20 years, by \cite{TZ23} for general three dimensional Anosov flows. The problem for general Anosov flows in higher dimensions is still open.

To clarify the problem, suppose that the stable and unstable distributions $E^s\oplus E^u$ are jointly integrable, and let $W$ be a manifold that is tangent to $E^s\oplus E^u$ at any point. Then we can use the manifold $W$ as a Poincaré section and examine the Poincaré map $F$ associated to the billiard flow $\Phi_t$ (see section \ref{sec:flowdef} for a precise definition). If the flow is hyperbolic, so is the Poincaré map.
Of course, such a map will have discontinuities, even if the flow is smooth, yet it is typically possible to prove that it enjoys exponential decay of correlations.
As usual, we can reconstruct the flow as a suspension over the Poincaré map. To do so, we only need the return time $\tau$, which gives the flow $\psi_t(x,s)=(x, t+s)$ on the space $Y=\{(x,s)\in W\times\bR_+\;:\; t\leq \tau(x)\}/\sim$ where the equivalence relation is given by $(x,\tau(x))\sim (F(x),0)$. 

Since the distributions $E^s$, $E^u$ are invariant, if they are jointly integrable, so is the manifold $W$. Accodingly, if two close-by points $x,y\in W$ return to $W$, and are not separated by a discontinuity of the map, they must return together. In other words, $\tau$ must be piecewise constant.\\

To see the potential implication of the above fact, consider the case in which all the return times are multiples of a time $\tau_0>0$. That is, for each $x\in W$, $\tau(x)/\tau_0\in\bN$. Then we can consider the set 
\[
A=\{(x,t)\in W\times \bR_+\;:\; t\in [n\tau_0,(n+1/4)\tau_0], \textrm{ for some }n\in\bN\}.
\]
Then, for each $m\in\bN$, $\psi_{m\tau_0}(A)=A$ while $\psi_{(m+1/2)\tau_0}(A)\cap A=\emptyset$. Thus $\psi_t(A)\cap A$ keeps oscillating, hence the system is not mixing.
The obvious objection is that it then suffices to check that there is no maximum common divisor. Unfortunately, Ruelle \cite{Ru83} presents an example of a suspension with a piecewise constant ceiling having only two values, with an irrational ratio, for which the decay of correlation can be arbitrarily slow.

One might think it should be easy to rule out the piecewise-constant ceiling case. However, suppose $\tau$ is cohomologous to a piecewise constant ceiling, that is, there exists $\alpha$ such that $\tau(x)=\tau_0(x)-\alpha\circ F(x)+\alpha(x)$, where $\tau_0$ is piecewise constant. Then we have two suspensions over $W$, one, $(Y, \psi_t)$ determined by the return time $\tau$, and the other $(Y_0, \psi^0_t)$ determined by the return time $\tau_0$. Define $\Gamma:Y_0\to Y$ by $\Gamma(x,s)=\psi_{\alpha(x)+s}(x,0)$. If, $s,t>0$, $s+t\leq \tau_0(x)$, then
\[
\Gamma(\psi^0_t(x,s))=\Gamma(x,s+t)=\psi_{s+t+\alpha(x)}(x,0)=\psi_{t}\Gamma(x,s).
\]
On the other hand if $\tau_0(x)<s+t\leq\tau_0(x)+ \tau_0(F(x))$, then
\[
\begin{split}
\Gamma(\psi^0_t(x,s))&=\Gamma(F(x),s+t-\tau_0(x))=\psi_{s+t+\alpha(F(x))-\tau_0(x)}(F(x),0)=\\
&=\psi_{s+t-\tau(x)+\alpha(x)}(F(x),0)=\psi_{s+t+\alpha(x)}(x,0)=\psi_t(\Gamma(x,s)).
\end{split}
\]
The above implies $\psi_t\circ \Gamma=\Gamma\circ \psi^0_t$, for all $t\in\bR$.
Consequently,
\[
\int_Y\oldh\circ \psi_t \cdot \oldf=\int_{Y_0} (\oldh\circ\Gamma) \circ \psi^0_t\cdot  (\oldf\circ \Gamma).
\]
Accordingly, provided $\alpha$ has enough regularity, the two flows have the same decay of correlations. This shows that the rate of decay depends in a non-obvious way on the ceiling function, and that establishing quantitative decay of correlations for flows is a different ball game compared to maps. 


\subsection{Billiards as a contact flow}

Early work of Hedlund \cite{hedlund} proved that geodesic flows on closed surfaces of constant negative curvature
are mixing, thereby demonstrating that certain geometric structures can prevent the above function $\tau$ from being cohomologous to a constant.  This was later extended to manifolds of variable negative curvature in two \cite{sinai mix} and higher dimensions \cite{eberlein}.  The contact structure of the geodesic flow proved to be the geometric structure that played the key role in the above arguments.

Leveraging on this fact and using fundamental insights contained in the work of Dolgopyat \cite{Do98}, Liverani \cite{Li04} established the exponential decay of correlations for smooth hyperbolic contact flows.

Billiards are a contact flow, preserving the one form $\omega = p dq$ in the notation of Section~\ref{sec:disp_bill}. Yet, when trying to apply the ideas in \cite{Do98,Li04} to billiards, one has the extra difficulties produced by the singularities of the flow.
Indeed, the flow is continuous, but the differential blows up at tangential collisions. It is then essential to combine the ideas coming from \cite{demers liverani,demzhang11,demzhang13,demzhang14} to treat singularities with the ones coming from \cite{Do98,Li04} to deal with flows. A first step, treating flows with discontinuities, was achieved in \cite{BL12}. However, the type of singularities allowed in \cite{BL12} did not include billiard systems. The last step was achieved in \cite{bdl}, where the following theorem was established.\footnote{In fact the result in \cite{bdl} is a bit more refined.}

\begin{thm}[{\cite[Theorem 1.2]{bdl}}]\label{thm:main_flow}
Let $\Phi_t:\Omega_0\to \Omega_0$ be a finite horizon
(two-dimensional) billiard flow associated to finitely many  scatterers
$B_i$ with $\cC^3$ boundaries of positive curvature and so that the $\overline B_i$
are pairwise disjoint. Then there exist a constant $\upsilon>0$, and a constant
 $C_1>0$, such that for any $\psi\in \cC^1(\Omega_0)$ and $\oldh \in \cC^2(\Omega_0)\cap \cC^0(\Omega)$ we have\footnote{For the definition of $\Omega,\Omega_0$ and $\Phi_t$ see the beginning of Section \ref{sec:disp_bill}.}
\[
\left|\int (\psi \circ \Phi_t) \, \oldh\, dm -
\int \psi dm\int \oldh \, dm \right|\le C_{1}|\oldh|_{\cC^2(\Omega_0)}|\psi|_{\cC^1(\Omega_0)} \cdot 
e ^{-\upsilon t} \, , \forall t \ge 0 \,  . 
\]
\end{thm}
In what follows, we will review the important ideas and constructions necessary for the completion of such an 
analysis.  A more complete exposition with details in the smooth hyperbolic case can be found in
\cite[Chapter 5]{DKL}.

\subsection{Definition of norms and Banach spaces}

The basic idea to prove the above theorem is to look at the semiflow $\cL_t\oldh= \oldh\circ \Phi_{-t}$, where $\Phi_t $ is the billiard flow, on a Banach space $\cB$ equipped with an appropriate norm. The norm is quite similar to the ones in section \ref{sec:geo banach}, but there are some important differences stemming from the fact that now the phase space has one dimension more. In particular, we still consider a set of stable curves that we call $\cW^s$; these are homogeneous stable manifolds and, thanks to the contact structure, we can choose them perpendicular to the flow direction (see \eqref{eq:contact}). Hence, their projection to the Poincaré sections would correspond to the curves $\widehat\cW_{\bH}$ of the section \ref{sec:Bt>0}. We can define the unstable curves similarly. We then have to define a notion of distance. The basic idea is that we want to call close curves that, when it is iterated backward in time, have parts that get closer. Clearly, this does not happen if two curves are displaced along the flow direction, hence if there is no unstable manifold that connects them, we consider their distance infinite. If the two curves are connected by an unstable manifold, then we project them on the Poincarè section and essentially measure their distance as in section \ref{sec:Bt>0}. The distance between test functions supported on two close stable curves is defined as in \eqref{eq:psi_dist}, using the natural projection
under the flow of these curves onto $M$, the phase space of the billiard map. In fact, there is a slight technical difference with \eqref{eq:psi_dist}, but it is irrelevant and, for the sake of simplicity, we avoid stating the precise definition.

Fix $0<\alpha \le 1/3$.  For $\oldh \in \cC^1(\Omega_0)$, define the
weak norm of $\oldh$ by
\[
|\oldh|_w = \sup_{W \in \cW^s} \sup_{\substack{\psi \in \cC^\alpha(W) \\ |\psi|_{\cC^\alpha(W)} \le 1}}
\int_W \oldh \psi \, dm_W \, ,
\]
where $dm_W$ denotes arclength along $W$.
Next, choose $1 < \oldbetaparam < \infty$ and
$0<\oldq <  \min \{ \alpha, 1/\oldbetaparam,  1- \oldbeta \}$. 
For $\oldh \in \cC^1(\Omega_0)$,
define the strong stable norm of $\oldh$ by
\[
\| \oldh \|_s = \sup_{W \in \cW^s} \sup_{\substack{\psi \in \cC^\oldq(W) \\
|W|^{\oldbeta} |\psi|_{\cC^\oldq(W)} \le 1}}
\int_W \oldh \psi \, dm_W \, .
\]
Choose $0<\gamma \le \min \{ \oldbeta , \alpha-\oldq \} < 1/3$.  Define the unstable norm
of $\oldh$ by\footnote{ In previous works, $d(\psi_1,\psi_2) =0$ was replaced by
$d_\oldq(\psi_1,\psi_2) \le \ve$ where $\oldq > 0$
and the distance used the $\cC^\oldq$ instead of the $\cC^0$ norm. The two formulations are equivalent by the triangle
inequality, using the strong stable norm, and since
$d(\psi_1,\psi_2)=0$ implies $d_\oldq(\psi_1,\psi_2)=0$.}
\[
\| \oldh \|_u = \sup_{0<\ve<\ve_0} \sup_{\substack{W_1, W_2 \in \cW^s \\ d_{\cW^s}(W_1,W_2) \le \ve}} 
\sup_{\substack{\psi_i \in \cC^\alpha(W_i) \\ |\psi_i|_{\cC^\alpha(W_i)} \le 1 \\ d(\psi_1,\psi_2) =0}}
\ve^{-\gamma} \left| \int_{W_1} \oldh \psi_1 \, dm_{W_1} - \int_{W_2} \oldh \psi_2 \, dm_{W_2} \right|  \, ,
\]
for some $\ve_0>0$ small enough. The reader can check that these are the same definitions as \eqref{eq:weak}, \eqref{eq:stable t>0}, and \eqref{eq:unstable t>0}.
However, the attentive reader will have noticed that the above norm does not allow for comparing integrals along stable curves that are displaced along the foliation direction.
To this end, we define a neutral norm by
\[
\| \oldh \|_0 = 
\sup_{W \in \cW^s}
 \sup_{\substack{\psi \in \cC^\alpha(W) \\ |\psi|_{\cC^\alpha(W)} \le 1}}
\int_W  \partial_t (\oldh \circ \Phi_t)|_{t=0} \, \psi  \, dm_W   \,  .
\]
Finally, define the strong norm of $\oldh$ by 
\[
\| \oldh \|_\cB = \|\oldh\|_s + \| \oldh \|_u + \|\oldh \|_0 \, .
\]
Not surprisingly, arguing exactly as in the case of the Poincar\'e map, one obtains:
\begin{prop}{\cite[Proposition 4.1]{bdl}}
\label{prop:L_t} There exist $C>0$ and $\Lambda\in (0,1)$ such that for all $\oldh \in  \cB$ and $t > 0$,
\begin{eqnarray}
|\cL_t \oldh|_w & \le & C |f|_w \, , \label{eq:weak L} \\
\| \cL_t \oldh \|_s & \le & C  \Lambda^{-\oldq t} \| \oldh \|_s 
+ C  |\oldh|_w \, ,  \label{eq:strong stable L}  \\
\| \cL_t \oldh \|_0 & \le & C \| \oldh \|_0 \, ,   \label{eq:neutral L}\\
\| \cL_t \oldh \|_u & \le & C t^\gamma \Lambda^{-\gamma t} \| \oldh \|_u + C  \| \oldh \|_0 
+ C  \| \oldh \|_s \, .
\label{eq:strong unstable L} 
\end{eqnarray}
\end{prop}
\begin{lem}{\cite[Lemma 3.10]{bdl}}
\label{lem:compact}
The unit ball of $\cB$  is compactly embedded in $\cB_w$.
\end{lem}
Unfortunately, due to the presence of the neutral norm, which does not contract, we do not have a proper Lasota-Yorke inequality, and hence we cannot apply the Hennion-Neusbaum theory (see \cite[Appendix B]{DKL} for details). To overcome this problem, we need an extra idea: to look at the {\em resolvent operator}. 


\subsection{The generator and the resolvent}

To proceed, we need the following fact.

\begin{lem}[{\cite[Corollary 4.2]{bdl}}]
The operator
$\cL_t$ is  continuous  on
$\cB$ for any $t\ge 0$.
\end{lem}
By general semigroup theory, the above implies that the semigroup $\cL_t$ has a generator $X$ with dense domain. Let $\rho(X)$ be the resolvent of $X$ and $\sigma(X)=\bC\setminus \rho(X)$ be the spectrum of $X$. Then, for each $z\in \rho(X)$ the resolvent operator is given by
\[
R(z):=(z\Id-X)^{-1}.
\]
Note that, for $\Re(z)>0$ we have the representation\footnote{Just multiply the representation by $z\Id-X$, recall that, by definition, $\frac{d}{dt}\cL_t h=X\cL_t h$ for all $h\in \operatorname{Dom}(X)$, and integrate by parts to verify the formula.}
\begin{equation}\label{eq:resolvent}
R(z) \, g =\int_0^\infty e^{-zt} \cL_t g dt
\end{equation}
which is convergent by Proposition \ref{prop:L_t}. The relevance of the formula \eqref{eq:resolvent} is that it contains an integral in the flow direction, exactly the direction that was not possible to control for $\cL_t$. Thanks to this, one can obtain a proper Lasota-Yorke inequality for $R(z)$.
\begin{prop}
[{\cite[Proposition 5.1]{bdl}}]
\label{prop:LY R}
There exists $C\ge 1$ and $\tilde\lambda\in (0,1)$, such that for all $z \in \mathbb{C}$ with $\Re(z) = a >0$, all
$\oldh \in \cB$ and all $n \ge 0$,
\begin{eqnarray}
|R(z)^n \oldh|_w & \le & C a^{-n} |\oldh|_w   \label{eq:weak R} \\
\| R(z)^n \oldh \|_{\cB} & \le & C(a - \ln {\tilde \lambda})^{-n} \| \oldh \|_{\cB} + C a^{-n}(1+|z|) |\oldh|_w   \label{eq:strong stable R} .
\end{eqnarray}
\end{prop}
We can now apply Hennion's theory, from which it follows that the spectral radius of $R(z)$, is at most $a^{-1}$ while the essential spectral radius is at most $(a - \ln {\tilde \lambda})^{-1}$. \\
Note that if $\nu$ belongs to the point spectrum of $R(z)$, then $z-\nu^{-1}$ belongs to the point spectrum of $X$ and vice versa. This is illustrated in Figure \ref{fig1}.
The spectrum of $X$ lies outside the
solid red circle in Figure~\ref{fig1}(b), and its essential spectrum must lie outside the dashed red circle.  As we can choose the imaginary part of $R(z)$ arbitrarily, this forces
the strip between the dashed blue line ($\Re(w) = \log \tilde\lambda$) and the imaginary axis to contain
only isolated eigenvalues of finite multiplicity.  In addition, since the billiard flow is known to be mixing (e.g., see \cite{km}), $X$ cannot have eigenvalues on the immaginary line apart from zero. Note, however, that it is not yet excluded that the eigenvalues of $X$ may accumulate on the imaginary axis as $|\mbox{Im}(w)| \to \infty$. This would prevent an exponential decay of correlations since if for each $\ve>0$, there exists $\oldh\in\cB$ and $a\in (0,\ve)$ such that $X \oldh=(-a+ib)\oldh$, for some $b\in\bR$,  then $\int \vf\circ \Phi_t \oldh=\int \vf \cL_t \oldh=e^{-at+ibt}\int\vf \oldh$. 
\begin{figure}
	\centering
\begin{subfigure}[b]{0.3\textwidth}
\begin{tikzpicture}
\draw[thick] (0,-2) -- (0, 2);
\draw[thick] (-2,0) -- (2,0);
\draw[thick, red] (0,0) circle (1.5);
\draw[thick, red, dashed] (0,0) circle (1);
\put (45,-8){\scriptsize $a^{-1}$}
\draw (0,0) -- (.7071,.7071);
\put (23, 24){\scriptsize $(a-\log \lambda)^{-1}$}
\end{tikzpicture}
\vspace{16 pt}

\centering  $\hspace{10 pt} \mbox{(a)}$
\end{subfigure}
\qquad\qquad
\begin{subfigure}[b]{0.4\textwidth}
\begin{tikzpicture}
\draw[thick] (0,-2.5) -- (0, 3.2);
\draw[thick] (-1.8,0) -- (2.8,0);
\draw[thick] (-.05,1.5) -- (.05, 1.5);
\draw[thick] (1,-.05) -- (1, .05);
\draw[thick, red] (1,1.5) circle (1);
\draw[thick] (1, 1.5) circle (.01);
\fill (1, 1.5) circle (.01);
\put (25, 35){\scriptsize $z$}
\put (4, 40){\scriptsize $ib$}
\put (27, -8){\scriptsize $a$}
\draw (1, 1.5) -- (1.9, 1.9);
\put (40, 42){\scriptsize $a$}
\draw[thick, red, dashed] (1, 1.5) circle (1.5);
\draw (1,1.5) -- (1.3, 2.95);
\put (32, 88){\scriptsize $a - \log \lambda$}
\draw[thick, blue, dashed] (-.5,-2.5) -- (-.5,3.2);
\put(-23, -10){\scriptsize $\log \lambda$}
\draw (0,0) node[cross = 2 pt, rotate = 90, thick, blue]{};
\draw (-.3,2) node[cross = 2 pt, rotate = 90, thick, blue]{};
\draw (-.2,-1) node[cross = 2 pt, rotate = 90, thick, blue]{};
\draw(-.3, -2) node[cross = 2pt, rotate = 90, thick, blue]{};
\draw (-.1,1) node[cross = 2 pt, rotate = 90, thick, blue]{};
\draw (-.3,2.5) node[cross = 2 pt, rotate = 90, thick, blue]{};
\draw (-.2, 2.7) node[cross = 2 pt, rotate = 90, thick, blue]{};
\draw (-.12,2.9) node[cross = 2 pt, rotate = 90, thick, blue]{};
\draw (-.08,3.1) node[cross = 2 pt, rotate = 90, thick, blue]{};
\draw (-.08, 3.4) node[red]{\scriptsize ?};
\end{tikzpicture} \\
\centering (b)
\end{subfigure}
\caption{(a) The spectrum of $R(z)$ is in the solid red circle.
The essential spectrum is within the dashed red circle.
\\
(b) The spectrum of $X$, in blue, the essential spectrum is on the left of the vertical blue dashed line.}
\label{fig1}
\end{figure}
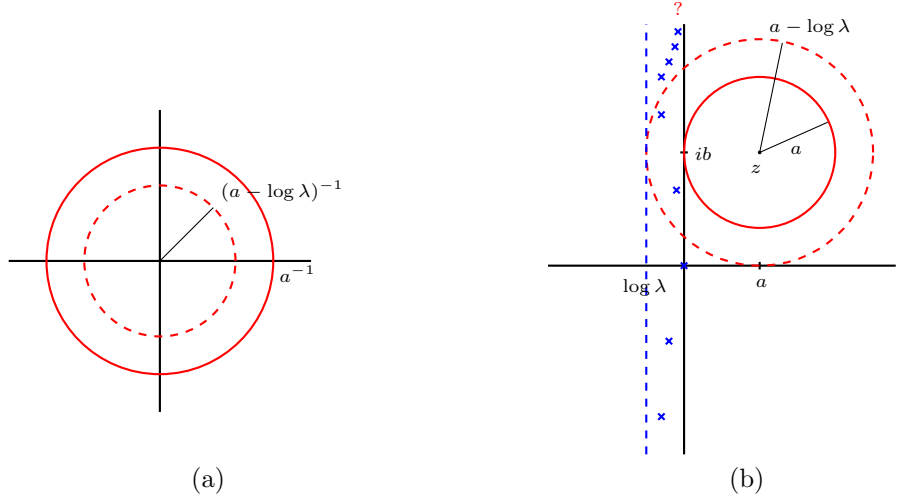
To conclude, it is then necessary to prove that there is a strip near the imaginary axis that does not contain any eigenvalue of $X$ apart from zero.
One way to establish that, as made clear from Figure \ref{fig1}, is to show that there exists $\nu>0$ such that for each $b\in\bR$ large enough, there exists an $n(b)$ such that
\begin{equation}\label{eq:dolgopyat_bound}
\|R(a+ib)^{n(b)}\|_\cB\leq (a+\nu)^{-n(b)},
\end{equation}
since this implies that the spectral radius of $R(x)$ is $(a+\nu)^{-1}$.
Unfortunately, the existence of a spectral gap of $X$ does not imply the existence of a spectral gap of $\cL_t$. Indeed, the naive idea that $\sigma(\cL_t)=\overline{e^{\sigma (X)t}}$, called the Weak Spectral mapping Theorem, does not apply in the present case. Yet, we can use a result of \cite{Butterley} as a substitute. Indeed, \cite[Theorem 1]{Butterley} implies that
\[
\cL_t=P_t+e^{\hat X t}
\]
where $\hat X$ is a finite rank operator $\sigma(\hat X)\subset \sigma(X)$, $[P_t, e^{\hat X t}]=0$, and, for each $\upsilon<\nu$, there exist $C>0$ such that $|P_t \oldh |_w\leq C e^{-\upsilon t}\|\oldh \|_\cB$ for all $\oldh \in\cB$. In turn, this implies Theorem \ref{thm:main_flow}. 

\subsection{The Dolgopyat estimate}

It remains to prove equation \eqref{eq:dolgopyat_bound}. Note that, by \eqref{eq:strong stable R}, it suffices to prove
\begin{equation}\label{eq:dolgopyat_bound_w}
|R(a+ib)^{n(b)}|_w\leq (a+\nu)^{-n(b)},
\end{equation}
for $n(b)=c\ln|b|$, with $a>0$ small and $c$ large enough. This is the so-called Dolgopyat estimate,
which seeks to leverage cancellations due to the oscillations provided by large $b$.

Iterating \eqref{eq:resolvent} we obtain
\[
R(z)^n=\int_0^\infty \frac{t^{n-1}}{(n-1)!}e^{-zt}\cL_t dt .
\]
Note that the contribution to the integral for $t\leq c_1 a^{-1} n$ is negligible, provided $c_1$ is chosen small enough. Moreover, for each $W\in \cW^s$ and $t>0$, we have that $\Phi_{-t}W=\sum_i W_i$, where $W_i\in\cW_s$, thus, for $r >0$ small enough,
\begin{equation}\label{eq:dolgo1}
\int_t^{t+r}\hskip-6 pt ds \int_{W} \frac{s^{n-1}e^{-zs}}{(n-1)!}\cL_s \oldh  \vf dm_W =\sum_i\int_t^{t+r}\hskip-6 pt ds \int_{W_i}\hskip-3pt\frac{s^{n-1}e^{-as-ibs}}{(n-1)!} \oldh \vf\circ \Phi_{-s}J_s dm_{W_i}
\end{equation}
where $J_s$ is the Jacobian of the change of variable. That is, we have the integral over the two-dimensional center-stable manifolds 
$W_i^c=\cup_{s=t}^{t+r}\Phi_{-s}W_i$. Due to the lack of regularity of the flow, the collection $\{W_i\}$ will contain several small manifolds. Yet, it is possible to prove that the majority of such manifolds are reasonably long.\footnote{This is similar to proving that the standard pair families are invariant under the flow, for readers acquainted with such language.} The estimate \eqref{eq:weak R}, that we want to improve, is obtained by taking the modulus inside the integral in the above formulae. To do better, we have to show that cancellations occur due to the oscillatory factor $e^{ibs}$. While the implementation is rather technical, the basic idea is simple: by mixing, the manifolds $W^c_i$ are rather uniformly distributed, so many are close. Let the manifolds $W_i^c, W^c_j$ be $\rho$-close, in the sense of the $d_{\cW^s}$ distance, and we want to compute
\[
\int_{W^c_i}\frac{s^{n-1}}{(n-1)!}e^{-as-ibs} \cL_{nr} \oldh \vf\circ \Phi_{-s}J_s+\int_{W^c_j}\frac{s^{n-1}}{(n-1)!}e^{-as-ibs} \cL_{nr} \oldh \vf\circ \Phi_{-s}J_s.
\]
The reason to apply \eqref{eq:dolgo1} to $\cL_{nr} \oldh $ rather than to $\oldh $, 
is that one would like a function almost constant along the unstable manifolds, so one can compare the two integrals. Unfortunately, the unstable manifolds have very poor regularity and may be too short to intersect both $W^c_i$ and $W^c_j$. This is overcome by Theorem \ref{thm:app_fol_reg}, which allows us to first discard the complement of $\Delta_\up$ (in the notation
of that theorem), with $\up$ larger than $nr$, and then compare the remaining part of the two integrals along the approximate foliation by paying a small error proportional to the derivative of $ \cL_{nr} h$ along the leaves of the foliation, which we denote $\partial_\up\cL_{nr} h$. 
By construction, and recalling  \eqref{eq:hyp_flow}, $|\partial_\up\cL_{nr} \oldh |\leq C\lambda^{-nr}|\nabla \oldh |$.

The key fact is that, along the approximate foliation, the stable curves shift in the flow direction. This is illustrated by Figure \ref{box}. The amount of shift is proportional to the length of the curves and to their distance; this is a consequence of the contact structure of the flow and the fact that the approximate foliation is constructed respecting such a contact structure (see the \cite[Proof of Lemma 8.8]{bdl} for more details).

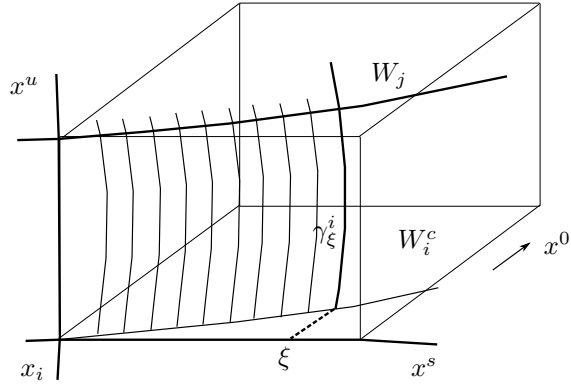
\begin{figure}
	\centering
                \includegraphics[height=2.0in]{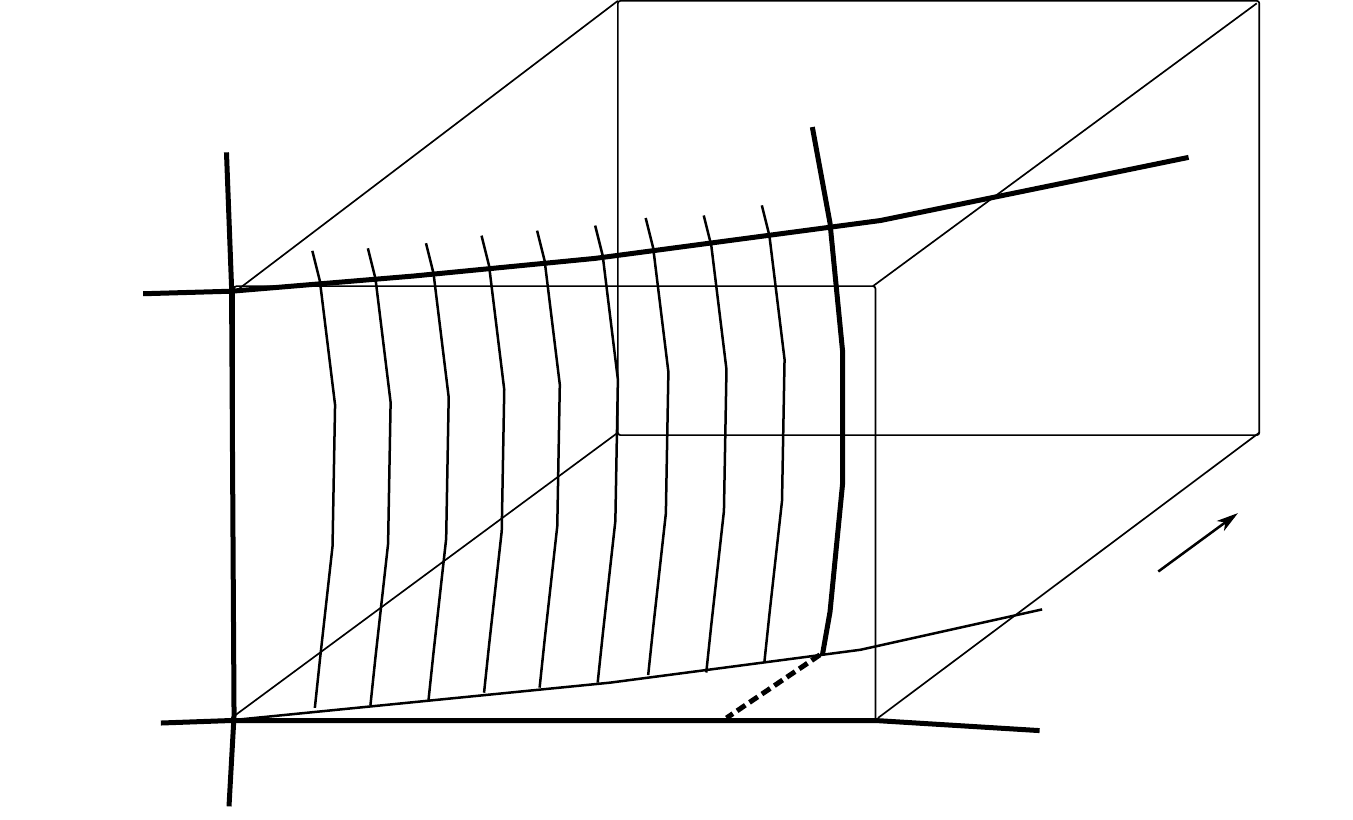}
                \put(-65,4){$x^s$}
                \put(-216,110){$x^u$}
                \put(-15,50){$x^0$}
                \put(-212,4){$x_i$}
                \put(-115,8){$\xi$}
                \put(-80,115){$W_{j}$}
                 \put(-70,52){$W^c_{i}$}
                \put(-102,56){$\gamma^i_\xi$}
  \caption{Projection along the foliation $\cF$ of $W_j$ onto $W_i^c$. The dotted line shows the displacement in the flow direction.}
  \label{box}
\end{figure}

Accordingly, the curves with constant phase in $W_j^c$, which are stable curves, project to curves that are transverse to the curves of constant phase (the stable curves in $W^c_i$).  This is due to the twist provided by the
contact structure and forces the sum of the exponentials to provide some cancellations. This is the basic mechanism one wants to exploit. To make it rigorous and effective, it is necessary to compare all the remaining nearby manifolds simultaneously, which is achieved via an $L^2$ estimate (see [Equation 8.40 and Lemma 8.8]\cite{bdl}). To perform such an estimate, first one does a rough bound of the few manifolds that are too close, and hence have too small a twist in the flow direction to be useful (\cite[Lemma 8.9]{bdl}). Then, one uses the argument above to show that there is significant phase cancellation in the remaining manifolds in each flow box
such as the one depicted in Figure \ref{box}. The one remaining issue is the fact that part of the mistake in comparing nearby manifolds is proportional to $|\nabla \oldh |$, which is not bounded by $\|\oldh \|_\cB$. To take care of such a last technical issue, it is possible to use an approximation argument, as carried out in 
\cite[Lemma 7.4]{bdl}).

\section{Open Problems}
\label{sec:open}

The functional analytic machinery surveyed in the previous sections has settled
a number of long-standing questions, but it has also opened many more.  We
collect here a (necessarily incomplete) list of problems that we find
particularly tempting, and that we believe to be within reach of the techniques
described above (or of natural extensions thereof).  In each case, we briefly
recall the relevant context and indicate what we see as the main difficulty.

\subsection{Measure of maximal entropy}
\label{sec:max quest}

The generically superpolynomial rate of mixing for the MME reviewed in Section~\ref{sec:MME rate} begs the question whether for typical Sinai billiard tables the rate is in fact exponential, as it is for the smooth
invariant measure for the billiard.  This question is made more pressing by the fact that for many classes of
smooth systems the MME mixes exponentially, even when the hyperbolicity is weak.  We mention some classes
of unimodal maps \cite{Bruin Todd, IT} and topologically transitive $C^\infty$ surface diffeomorphisms with positive
topological entropy \cite{BCS} as two notable examples.

\begin{prob}
Is the rate of decay of correlations for the measure of maximal entropy for a finite horizon Sinai billiard exponential?  Alternatively, is there a geometric condition on the billiard table that can yield improved estimates?
\end{prob}

A related question is whether the sparse recurrence condition $h_* >s_0 \log 2$ stated in \eqref{eq:s0} is necessary.  
The recent preprint
\cite{cli day} proves the existence and uniqueness of the measure of maximal entropy using a symbolic
coding that does not use this condition.  This construction, however, does not prove that the measure is Bernoulli, nor does it estimate a rate of mixing.
Another recent related work \cite{LOP} proves uniqueness of an adapted measure of maximal entropy without the
sparse recurrence condition, but does not prove existence. 

\begin{prob}
Is the sparse recurrence condition $h_* > s_0 \log 2$ necessary to establish the results of Theorem~\ref{thm:mu0}?
Alternatively, is it true that $h_* > s_0 \log 2$ holds for all finite horizon Sinai billiard tables, not just generically?
\end{prob}

An intermediate question is whether violation of the condition $h_* > s_0 \log 2$ leads to a measure of 
maximal entropy that is not adapted.  Nonadapted invariant measures with positive entropy for billiards are
constructed in \cite{CDLZ}.

\subsection{Geometric potentials}
\label{sec: geo quest}

For the family of geometric potentials, $\phi_t = - t \log J^uT$, 
Theorem~\ref{thm:thermo gap} establishes a spectral gap for $\cL_t$ and a unique equilibrium state
$\mu_t$ for all $t \in (0, t_*)$, where $t_*>1$ is defined in \eqref{eq:t star}.  

\begin{prob}
It is not known whether the restriction on $t_*$ is optimal or necessary, giving rise to the following questions.
\begin{enumerate}
	\item Are there billiard tables for which $t_* = \infty$?
	\item If $t_* < \infty$, Is there a phase transition at $t = t_*$ or can the results of Theorem~\ref{thm:thermo gap} be extended past this point?
\end{enumerate}
\end{prob}

Related to the first question, the quantity $\Lambda = 1 + 2\cK_{\min}\tau_{\min}$ is a minimum hyperbolicity constant for a single collision, yet there may be periodic orbits along which this minimum hyperbolicity is realized as a Lyapunov exponent.
For such tables, 
$P(t) \ge -t\log \Lambda$ is guaranteed for all $t \ge 0$.
Note that $P(t) > -t \log \Lambda$ for all $t \ge 0$ would imply $t_* = \infty$.

\subsection{More general dispersing billiards}
\label{sec:general quest}

The results presented above are almost entirely in the context of a finite horizon Sinai billiard table.  
This invites a host of questions related to generalizing these results to other classes of dispersing billiards.

\begin{prob}
For an infinite horizon billiard table and geometric potentials $\phi_t$, for what range of $t>0$ can the results of Theorem~\ref{thm:thermo gap} be established?
\end{prob}
While it is known that $P(0)=\infty$ for infinite horizon billiard maps, yet it is possible that some of the thermodynamic formalism may be established for $t$ closer to 1.  A first step in this direction is the
thesis \cite{G}, which establishes a spectral gap for the transfer operator $\cL_t$ for $t \in (2/3, t_*)$.

Alternatively, considering the family of potentials $\phi_t = - t \tau$ as described in Section~\ref{sec:abramov}
can give access to the infinite horizon billiard flow.  In contrast to the map, the infinite horizon billiard
flow has finite topological entropy.
 
\begin{prob}
For an infinite horizon billiard flow, can the uniqueness and further properties of a measure of maximal entropy and 
related equilibrium states for a range of $t$ be proved in analogy with those described in Section~\ref{sec:abramov}?
\end{prob}

Generalizing still further, one can ask how much of the thermodynamic formalism can be established for 
dispersing billiard tables with corner points.  Such tables are of interest in models of billiard tables permitting
energy transmission between cells \cite{BGNST}.

\begin{prob}
How much of the above program can be established for dispersing billiards with corner points?
\end{prob}

A key tool used in the analysis of finite horizon tables is the complexity bound Lemma~\ref{lem:complexity},
which is not available for billiards with corner points.  Instead, one must use the bounds derived by
De Simoi and Toth \cite{de simoi toth}, which more closely resemble a multi-step expansion version 
of Lemma~\ref{lem:one step}.  It would be necessary to determine how the introduction of the exponent
$t$ impacts this estimate.

\subsection{Sequential and open billiards (loss of memory)}
\label{sec:sequential quest}
The development of projective cones for dispersing billiards in \cite{dl cones} opens the possibility to study
sequential billiards generated by aperiodic arrangements of scatterers. Prime examples include aperiodic Lorentz gases and chaotic scattering. In the first case, one would like to establish loss of memory and a CLT, and in the second, the structure of the trapped trajectories and the scattering probability (essentially an equivalent of the scattering matrix). As explained in Section~\ref{sec:application_sequential}, both the loss of memory for the Lorentz gas and the chaotic scattering (in the general case, in which there can be an eclipse) can be studied by introducing an artificial device to ensure enough mixing (lazy gates or boxing the obstacle).  The central problem is that the cone $\cC_{c,A,L}$ is always mapped by a sequence of transfer operators in a cone $\cC_{\chi c,\chi A, 3L}$, and the extra mechanism is needed to obtain the contraction in the parameter $L$. This is not a mathematical artifact; it corresponds to the fact that the relevant transfer operators are multiplied by the characteristic function of a small set, so the resulting evolved density could, in principle, be supported on atypical trajectories. Hence, it is a major open problem.
\begin{prob}
How to apply the cone method to establish loss of memory for general aperiodic Lorentz gases and general chaotic scatterers?
\end{prob}

\subsection{Aperiodic Lorentz gases (CLT)}
\label{sec:lorentz quest}
Given the loss of memory for the aperiodic Lorentz gas (even with lazy gates), it remains the problem to establish a CLT. This is a special case of the more general case of deterministic walks in a random environment.  In \cite{AL20}, is shown that a class of aperiodic Lorentz gases in which some periodic structure is retained is essentially a random walk in a random environment with exponentially decaying memory. Unfortunately, there is no proof of a CLT for such probabilistic models. We are only aware of \cite{T86} where the transition probability depends on just one step in the past. Hence, the problem of studying random walks in a random environment with long memory is also an open problem in probability theory.

The situation is easier if the particle has a drift, and one studies the CLT for the fluctuations of the velocity around the mean. In \cite{DK23} the authors are able to establish the CLT for a deterministic random walk in a random environment where the local dynamics is an expanding map (rather than a billiard map) and the transition to different maps is arranged so that the system has a drift. This cannot be obtained in a Lorentz gas in a tube by playing with the geometry of the boundary, but it can be obtained by adding a small (electric) field to the system. However, the addition of an electric field will make the energy of the particle increase without bound, hence it is necessary to add a dissipative mechanism. This is usually done by adding a Gaussian thermostat.
\begin{prob}
Can one use projective cones constructed for dispersing billiards to establish a Central Limit Theorem for an aperiodic Lorentz tube with an electric field and a Gaussian thermostat?
\end{prob}

\subsection{Gaussian thermostats and mean-field models}
\label{sec:gaussian quest}
The study of billiards with an electric field and a Gaussian thermostat is physically very relevant.
In the seminal paper \cite{CELS}, it was proven that for a billiard on a torus subject to an electric field
and a Gaussian thermostat, the electric current exhibits linear response starting from zero electric field. In other words, calling $J(E)$ the current induced by the electric field, \cite{CELS} proves that $J'(0)$ exists. This is strictly related to Ohm's law. The situation for $E\neq 0$ is still unknown; there exists only some numerical work \cite{BDL} which suggests that $J(E)$ may not be differentiable.
\begin{prob}
Can one establish a linear response for a non-zero electric field, or characterize the regularity of $J(E)$ if not differentiable?
\end{prob}
Another natural problem is to consider the case in which more than one particle is present. This is an extremely hard problem, but one can simplify the situation by considering the case in which the particles do not interact with one another, but only via the Gaussian thermostat. This can be viewed as an example of a mean-field model.
Transfer operators have proven quite successful in investigating mean-field models consisting of coupled maps, e.g. \cite{K20, G22, BL23, BL25}. Yet, nothing is known about coupled flows.
\begin{prob}
Study the case in which $n$ non-interacting particles are subject to an external electric field and a Gaussian thermostat.
\end{prob}
The case when the particles can interact is much harder. It corresponds to studying a higher-dimensional billiard, see also Section \ref{sec:high quest}.
A series of papers by Bonetto, Lebowitz, Chernov, Korepanov, et al. \cite{BCK, BCKL, BDL} have explored such problems, using various kinds of stochastic approximations and numerical analysis, but the relevant question remains wide open.

\subsection{Billiard flows}
\label{sec:flow quest}

The use of the contact form as a fundamental mechanism for proving exponential decay of correlations for
the finite horizon billiard flows outlined in Section~\ref{sec:flow} raises the question of what other dispersing
billiard flows enjoy exponential mixing, since they all preserve the same contact form.

\begin{prob}
What classes of billiard flows have exponential decay of correlations?  Dispersing billiards with corner points or cusps?
Some classes of billiards with focusing boundaries, such as some drive-belt stadia?
\end{prob}

The assumption of finite horizon is essential here since it is known that the Sinai billiard flow with infinite horizon
mixes at the polynomial rate $1/t$ \cite{BBM}.
Under the finite horizon assumption, exponential mixing for dispersing billiard flows with corner points is expected
due to the analogous behavior for the map.  In addition, billiard tables with cusps and certain billiard tables
with focusing boundaries are known to mix at rates much faster than the map \cite{BM}, raising the 
possibility that such flows in fact enjoy exponential decay of correlations.

\subsection{Higher dimensional billiards}
\label{sec:high quest}

Our discussion in this review has been almost exclusively limited to two-dimensional billiard tables.  Yet,
some extensions to higher-dimensional billiards have been made under the additional assumption of subexponential growth of complexity of the singularity set \cite{BT2}.  This paper uses the construction of
Young towers rather than the analysis of the transfer operator presented here.

\begin{prob}
Can appropriate function spaces for the transfer operator associated with a multidimensional dispersing billiard be constructed?  If so, and in light of Proposition~\ref{lem:cone bound}, can one also construct projective cones for such 
billiards, which would allow for the study of sequential multidimensional dispersing billiards?
\end{prob}
Such a construction would enable the application of thermodynamic formalism to multidimensional dispersing
billiards, which is currently undeveloped. 

Finally, it is important to remark that a system of many interacting particles can be viewed as a multidimensional billiard, however, one which is non-uniformly hyperbolic since the balls can fail to collide for some special trajectory. The study of such systems is wide open, essentially no results exist, even though one must mention the monumental paper \cite{CD09} in which a system of two balls is studied in the limit in which the mass ratio tends to infinity.


\end{document}